\pdfoutput=1                    %% Need this for accurate compiling of tex files when submitting in ArXiv. %%

\documentclass[twoside,11pt]{article}

\usepackage{fullpage}
\usepackage[authoryear]{natbib}
\usepackage{amsthm, amsfonts, amssymb, amsmath, bbm}
\usepackage{enumerate}  %Replacing the enumitem package by the enumerate package which provides more flexibilities.
\usepackage{mathtools}
\mathtoolsset{showonlyrefs}

\newtheorem{thm}{Theorem}[section]
\newtheorem{lem}{Lemma}[section]

\newtheorem{prop}{Proposition}[section]
\newtheorem{defi}{Definition}[section]
\newcounter{myalgctr}
\newenvironment{rem}{%      define a custom environment
   \vskip1mm\indent%         create a vertical offset to previous material
   \refstepcounter{myalgctr}% increment the environment's counter
   \textbf{Remark \themyalgctr}% or \textbf, \textit, ...
   }{\hfill$\diamond$\par}  %          create a vertical offset to following material
\numberwithin{myalgctr}{section}

\numberwithin{equation}{section}

\providecommand{\norm}[1]{\left\lVert#1\right\rVert}
\DeclareMathOperator*{\argmin}{arg\,min}

\newcommand{\RIP}{\mathrm{RIP}}

\def\tcm{\textcolor{magenta}}

\usepackage{xr}
\RequirePackage[colorlinks,citecolor=blue,urlcolor=blue,linkcolor=blue]{hyperref}
\makeatletter
\def\namedlabel#1#2{\begingroup
    #2%
    \def\@currentlabel{#2}%
    \phantomsection\label{#1}\endgroup
}
\makeatother
\newcommand{\vertiii}[1]{{\vert\kern-0.25ex\vert\kern-0.25ex\vert #1
    \vert\kern-0.25ex\vert\kern-0.25ex\vert}}

\usepackage[toc,page]{appendix}

\usepackage{abstract}

\def\tcm{\textcolor{magenta}}

\def\cred{\color{red}}

\def\cyan{\color{cyan}}  %% Introducing this command to edit already red portions added by Arun, so it's easier to check my edits specifically. Can always make it red later (2/10)

\def\tcm{\textcolor{black}}

\def\cred{\color{black}}

\def\cred{}
\def\cyan{}  %% Needed to add this in Arxiv draft. (Wasn't there in IMAIAI Revision tex file. AC (5/9/2022)

\begin{document}

\title{\bf Moving Beyond Sub-Gaussianity in High Dimensional Statistics: Applications in Covariance Estimation and Linear Regression\thanks{
To appear in {\it Information and Inference: A Journal of the IMA}.} %Both authors contributed equally to this work.}
}

\author{
    Arun Kumar Kuchibhotla$^1$\thanks{Department of Statistics and Data Science, Carnegie Mellon University. (Email: {\tt
            arunku@cmu.edu})}
    \and
    Abhishek Chakrabortty$^1$\thanks{Department of Statistics, Texas A\&M University. (Email: {\tt abhishek@stat.tamu.edu}; Corresponding author) \smallskip$^1$The authors contributed equally to this work.}
}

%\date{}

\maketitle
\begin{abstract}
Concentration inequalities form an essential toolkit in the study of high dimensional statistical methods. Most of the relevant statistics literature in this regard is, however, based on the assumptions of sub-Gaussian or sub-exponential random variables/vectors. In this paper, we first bring together, through a unified exposition, various probabilistic
 inequalities for sums of independent random variables under much {\cred more general} %much weaker
 exponential type (namely sub-Weibull) tail assumptions. These results extract a part sub-Gaussian tail behavior of the sum in finite samples, matching the asymptotics governed by the central limit theorem, and are compactly represented in terms of a new Orlicz quasi-norm -- the Generalized Bernstein-Orlicz norm -- that typifies such kind of tail behaviors.

We illustrate the usefulness of these inequalities through the analysis of four fundamental problems in high dimensional statistics. In the first two problems, we study the rate of convergence of the sample covariance matrix in terms of the maximum elementwise norm and the maximum $k$-sub-matrix operator norm which are key quantities of interest in bootstrap procedures and high dimensional structured covariance matrix estimation, as well as in high dimensional and post-selection inference. The third example concerns the restricted eigenvalue condition, required in high dimensional linear regression, which we verify for all sub-Weibull random vectors through a unified analysis, and also prove a more general result related to restricted strong convexity in the process. %under only marginal (not joint) tail assumptions on the covariates. To our knowledge, this is the first unified result obtained in such generality.
In the final example, we consider the Lasso estimator for linear regression and establish its rate of convergence to be generally $\sqrt{k\log p/n}$, for $k$-sparse signals, under much weaker than usual tail assumptions (on the errors as well as the covariates), %than those in the existing literature,
while also allowing for misspecified models and both fixed and random design. To our knowledge, these are the first such results for Lasso obtained in this generality. The common feature in all our results over all the examples is that the convergence rates under most exponential tails match the usual (optimal) ones obtained under sub-Gaussian assumptions.
Finally, we also establish some complementary results on analogous tail bounds for the suprema of empirical processes indexed by sub-Weibull variables. All our results are finite sample.
\end{abstract}


%\end{frontmatter}

%\vbadness10000 %% Need to include this for the JMLR template to avoid some unnecessary bad box warnings regarding underfull vboxes.

%\end{document}
\section{Introduction and Motivation}

%\tcb{**Color conventions -- For all revision changes, I introduced some color commands in the tex file. Let's use {\cred the ``cred'' (color red) or ``tcr'' (textcolor red) for actual edits,} and {\cmag the ``cmag'' (color magenta) or ``tcm'' (textcolor magenta) for our own comments} (and include a ** in all comments so that we can locate easily). In addition the ``tcb'' command is there for blue textcolor if needed -- AC.**}

In the current era of big data, with an abundance of information often available for a large number of variables, there has been a burst of statistical methods dealing with high dimensional data. In particular, estimation and inference methods are being developed for settings with a huge number of variables often larger than the number of observations available. In these settings, classical statistical methods such as the least squares or the maximum likelihood principle usually do not lead to meaningful estimators, and regularization methods have been widely used as an alternative; see, e.g., \citet{Wainwright_Book_2019} for an overview.
These methods typically penalize the original loss function, e.g. squared error loss or the negative log-likelihood function, with a penalty on the parameter vector that reduces the ``effective" number of parameters being estimated. The theoretical analyses of most of these methods, despite all their diversities, generally obey a \emph{common unifying theme} wherein a key quantity to control is the maximum of a (high dimensional) vector of averages of mean zero random variables. Since the dimension is potentially larger than the sample size, it is important to analyze the behavior of the maximum in a non-asymptotic way. Concentration inequalities and probabilistic tail bounds form a major part of the toolkit required for such analyses.

Some of the most commonly used probabilistic tail bounds are of the exponential type, including, in particular, Hoeffding's and Bernstein's inequalities; see Section 3.1 of \cite{GINE16} for a review. In the classical versions of these inequalities, the random variables are assumed to be bounded, but this assumption can be relaxed to sub-Gaussian and sub-exponential random variables, respectively; see Sections 2.6 and 2.8 of \cite{Vershynin18}, and also \citet{Wainwright_Book_2019}. A random variable is called sub-Gaussian if its survival function is bounded by that of a Gaussian distribution. A sub-exponential random variable is defined similarly (see Section~\ref{sec:Definition}). Note that in both these cases, the moment generating function (MGF) exists in a neighborhood around zero. Most of the high dimensional statistics literature is based on the assumption of sub-Gaussian or sub-exponential random variables/vectors. But in many applications, these assumptions may not be appropriate. For instance, consider the following two simple examples that exemplify the main issues.
\begin{enumerate}[--]
\item Suppose $(X_1, Y_1), \ldots, (X_n, Y_n)$ are independent and identically distributed (i.i.d.) observations of a random vector  $(X,Y) \in \mathbb{R}^2$ and let $\hat{\beta} = \sum X_iY_i/$ $\sum X_i^2$ denote the linear regression slope estimator for regressing $Y$ on $X$. Under a possibly misspecified linear model, the estimation of the asymptotic variance of $\hat{\beta}$ involves $\sum X_i^2(Y_i - X_i\beta)^2$, where $\beta$ is the limit of $\hat{\beta}$; see \cite{Buja14} for details. It is clear that if the initial random variables $X$ and $Y$ are %only
    sub-exponential, then the random variables $X_i^2(Y_i - X_i\beta)^2$ do not have a finite MGF. The same holds even when the ingredient random variables $X$ and $Y$ are further assumed to be sub-Gaussian.\par\smallskip
\item Let $Y$ be a response variable, and $X_1, X_2$  be two covariates, all having a finite MGF in a neighborhood of zero. In many applications, it is important to consider regression models with interaction effects among the covariates, and more generally, second (or higher) order effects such as $X_1^2, X_1X_2$ etc. The presence of such second order effects %therefore
    clearly implies that the summands involved in the analyses of these linear regression estimators may not necessarily have a finite MGF anymore.
\end{enumerate}
These examples are not high dimensional in nature, but are mainly presented here as some basic examples where the core problem becomes apparent. The requirement of controlling averages defined by higher order or product-type terms, as in the second example, also arises inevitably in the case of high dimensional regression and covariance estimation; see also the recent work of \citet{yu2019reluctant} on problems of a similar flavor. The first example, apart from its relevance in inference for linear regression estimators, also appears in the problem of testing for the existence of active predictors in linear regression. This problem can be reduced to a simultaneous significance testing problem based on all the marginal regressions, as shown in \cite{McKeague15}. For this type of marginal testing problems, uniform consistency of the estimators of the variance of all the marginal regression coefficient estimators is required, thus creating the need for a non-asymptotic analysis. %In all the applications from high dimensional statistics that we discuss later, the same issue can often come up naturally despite making light exponential tail assumptions on the ingredient random variables/vectors.}

\subsection{Our Contributions}

%\tcm{May need to revisit the writing (rather organization) of this sub-section to improve visibility of our main results and contributions. Need to talk more about this - AC.}

Tail bounds for sums of independent random variables play an important role in probability and statistics. In statistics, especially in the high dimensional statistics literature, most of the applications are studied only under strong (or light) exponential tail assumptions such as sub-Gaussian or sub-exponential. Although tail bounds do exist for sums of independent random variables with ``heavy'' exponential tails (scattered mostly in the probability literature), the impact of moving from sub-Gaussian/sub-exponential (i.e. light-tailed) variables to those with heavy exponential tails on the rates of convergence and the dependence on the dimension does not seem to be well-studied in the statistics literature. These heavy exponential tailed random variables are what we call \emph{sub-Weibull} variables (see Definition~\ref{def:subWeibull}).

The first goal of our article is to provide a clear (and unified) exposition of concentration inequalities related to sub-Weibull random variables, which constitutes the \emph{first part} of our paper. It provides in one place a user-friendly off-the-shelf toolset that can be readily used in the analysis of a variety of modern statistical problems (and yet under much weaker tail conditions than those typically assumed).

In the \emph{second part}, we provide applications of these concentration inequalities for four such fundamental problems in high dimensional statistics. A detailed account of our contributions along both these lines is provided next.

\paragraph{Exposition of Tail Bounds.} The outline of our probability exposition is as follows. We first propose a new Orlicz quasi-norm called the \emph{Generalized Bernstein-Orlicz (GBO) norm} that allows for a compact representation of the results regarding sub-Weibull random variables. \cite{Geer13} introduced its predecessor, the Bernstein-Orlicz norm, that provides a formal understanding of the nature of the tail bound given by Bernstein's Inequality (see Section~\ref{sec:Definition} for details). The recent paper \cite{Well17} extends the results of \cite{Geer13} to capture the tail behavior given by Bennett's inequality. Although it was not stressed
in \cite{Geer13}, one of the main features of Bernstein's inequality is that even for sub-exponentials, it provides a part sub-Gaussian tail behavior for the sum. This, in turn, plays a key role in proving the rate of convergence of a maximum of several such sums to be the same as that in the case of sub-Gaussian variables. The GBO norm is constructed with the aim of capturing a similar tail  behavior for the general case of sub-Weibulls. The results on unbounded empirical processes from \cite{Adam08}, along with a maximal inequality (Theorem 5.2) of \cite{Chern14} and the results of \cite{LAT97}, will be exploited to provide a sequence of ready-to-use results (Theorems~\ref{prop:SumNewOrliczVex}--\ref{thm:MaximalTailBound}) about sub-Weibull random variables. This is essentially the probability contribution of the current article.
The results of \cite{Adam08} are derived based on Chapter 6 of \cite{LED91}, and those of \cite{Chern14} are based on the maximal inequality of \cite{vdV11}. All of our results are derived under the assumption of independence only and allow for \emph{non-identically distributed} ingredient variables. The extensions for the supremum of empirical processes with sub-Weibull envelope functions are further discussed in Appendix~\ref{sec:EmpProcess}. %This is essentially the probability contribution of the current article.

%\tcm{[ADD HERE COMMENTS ABOUT: (1) SUB-GAUSSIAN EXTRACTION (and variance, if you want) ensuring the `right' rates for high-d max of avgs., (2) EMPIRICAL PROCESS RESULTS] (Also consult the previous introduction to check for any other details you may want to enter).}

Lastly, we would also like to point out that we mainly focus on exponential-type tails in this paper, %In this paper, we mainly focus on exponential-type tails
since in all our high dimensional applications, a logarithmic dependence on the dimension is desired (our proof techniques, however, also apply equally to polynomial-type tails). The initial version of the current paper was available in ArXiv since 2018~\citep{kuchibhotla2018moving}. Recently, \cite{bakhshizadeh2020sharp} further explored some refinements of our tail bound results there. But a unified exposition of concentration results, coupled with a thorough demonstration of their usefulness in various important statistical applications as discussed below, is still lacking in the literature to the best of our knowledge.

\paragraph{High Dimensional Statistical Applications.} Following the exposition of concentration inequalities, we apply these probabilistic tools to four fundamental problems in high dimensional statistics. In all these examples, we establish precise tail bounds and rates of convergence, under the assumption of sub-Weibull random variables/vectors only. The results, apart from being seamlessly unified and general in terms of the underlying tail assumptions, also exhibit several interesting features and provide some key insights into the behavior of these problems. In particular, \emph{a common outcome of all our analyses is that the rates of convergence generally match those obtained under the sub-Gaussian assumption}.

Furthermore, most results in high dimensional statistics that involve \emph{random vectors} are only derived under tail assumptions on the \emph{joint} distribution of the random vector (for example, a random vector $X$ is sub-Gaussian if $\theta^{\top}X$ is uniformly sub-Gaussian over all $\theta$ of unit Euclidean norm). Although commonly adopted in the literature, such a condition %This
imposes (often implicitly) certain strong restrictions on the joint distribution, as discussed at the beginning of Section \ref{sec:Applications}. Throughout this paper, we make a formal \emph{distinction} between such a \emph{`joint' assumption} on the tail behavior of a random vector versus a much weaker \emph{`marginal' assumption} on the tail behaviors of its coordinates only; see Definitions \ref{def:JointWeibull} and \ref{def:MarginalWeibull}. %in Section \ref{sec:randomvect}.
All of our applications are also studied under such an assumption only on the marginal distributions, and often with nearly (if not exactly) similar results and convergence rates.

The description and the main implications of our results for each of the four high dimensional statistical applications we consider in this paper are enlisted below. (In all examples, $p$ denotes the ambient dimension of the random vectors and $n$ denotes the sample size.)

\begin{enumerate}[1.]
\item \emph{Covariance Estimation (Maximum Elementwise Norm).} A central part of high dimensional inference hinges on an application of the central limit theorem through a bootstrap procedure. The consistency of the bootstrap in this case requires consistent estimation of the covariance matrix in terms of the maximum elementwise norm. This norm also appears in the coupling inequality for maxima of sums of random vectors; see Theorem 4.1 of \cite{Chern14}. In Section~\ref{sec:ElementWiseMax}, we prove a finite sample tail bound (via Theorems \ref{cor:Covariance} and \ref{thm:CovarianceMaxNorm}) for the error of the sample covariance matrix in terms of this norm under the assumption of (marginally) sub-Weibull $(\alpha)$ ingredient random vectors. The rate of convergence is shown to be $\sqrt{\log p/n}$ if $\log p = o(n^{\alpha/(4-\alpha)})$; see Remark~\ref{rem:ExpectationBound}. This rate of convergence can be easily shown to be optimal in case the random vectors are standard multivariate Gaussian. Furthermore, the tail bounds presented in this section also play a central role in sparse covariance matrix estimation, as shown in \cite{Bickel08} and \cite{Cai11}. Both these papers deal with jointly sub-Gaussian random vectors, while the second paper additionally deals with fixed polynomial moments. Using our results in Section~\ref{sec:ElementWiseMax}, the problem of sparse covariance matrix estimation can be analyzed under weaker assumptions with logarithmic dependence on the dimension. Finally, the results in this section also establish the consistency of bootstrap procedures when applied to (high dimensional) marginally sub-Weibull random vectors.\par\smallskip
%%%%%%%%%%%%%%%%%%%%%%%%%%%%%%%%%%%%%%%%%%%%%%%
\item \emph{Covariance Estimation (Maximum $k$-Sub-Matrix Operator Norm).} Covariance matrices play an important role in statistical analyses through principal component analysis, factor analysis and so on.
    Clearly, for most of these methods, consistency of the covariance matrix estimator in terms of the operator norm is important. In high dimensions, however, the sample covariance matrix is known to be not consistent in the operator norm. Under such settings, in practice, one often selects a (random) subset of variables and focuses on the spectral properties of the corresponding covariance (sub)-matrix. %studies that (random) subset closely.
    In Section~\ref{subsec:SubMatrix}, we study the consistency of the sample covariance matrix of (marginal or joint) sub-Weibull $(\alpha)$ ingredient random vectors, in terms of the maximum sub-matrix operator norm with sub-matrix size $k \leq p$. %being bounded by $k \leq n$.
%This norm is given by the $k$-sparse eigenvalue which is used for restricted isometry property of the compressed sensing literature.
We show through  Theorem \ref{cor:RIPBoundUnified} %in this section
that the rate of convergence is $\sqrt{k\log(ep/k)/n}$ for most values of~$\alpha > 0$. This rate was previously obtained for the joint sub-Gaussian case by \cite{Loh12}; see Lemma 15 therein. This norm was possibly first studied by \cite{Rudelson08} for bounded random variables. The convergence rate of this norm plays a key role in studying post-Lasso least squares linear regression estimators and in structured covariance matrix estimation. The post-Lasso linear regression estimator was studied in \cite{Belloni13}, and more generally, in \cite{Kuch18} for post-selection inference. Lastly, for adaptive estimation of so-called bandable covariance matrices, a thresholding mechanism was introduced by \cite{Cai12}, where a result about maximum sub-matrix operator norm is also required. \cite{Cai12} deal with Gaussian random vectors, and using our results this method can be thus extended to sub-Weibull random vectors.
\par\smallskip
%%%%%%%%%%%%%%%%%%%%%%%%%%%%%%%%%%%%%%%%%%%%%%%
\item \emph{Restricted Eigenvalues.} \cite{Bickel09} introduced the restricted eigenvalue (RE) condition to %simultaneously
analyze the Lasso and the Dantzig selector. The RE condition concerns the minimum eigenvalue of the sample covariance matrix when the directions are restricted to lie in a specific cone (see Section~\ref{sec:RECondition} for a precise definition), and its verification forms a key step in high dimensional linear regression. A well known result in this regard is that of \cite{Rudelson13} who verified the RE condition for the covariance matrices of jointly sub-Gaussian random vectors. %with probability converging to one.
    Some extensions under weaker tail assumptions (e.g. sub-exponentials) have also been considered by \citet{LecueMendelson14}, among others; see Section~\ref{sec:RECondition} for further details. Based on our results in Section~\ref{subsec:SubMatrix}, we prove in Section \ref{sec:RECondition} that covariance matrices of both jointly and marginally sub-Weibull random vectors satisfy the RE condition with probability tending to one.
    In fact, we prove a more general result (in Theorem \ref{cor:REBoundUnified}) related to \emph{restricted strong convexity} from which the RE condition's verification follows as a consequence. %\tcm{
    To our knowledge, such unified results regarding the RE condition are not so easily accessible in the core statistics literature.
    %} %are rare in the literature.}
    \par\smallskip
%\cite{Neg12} provided a unified framework for analyzing high dimensional estimators defined as minimizers of a loss function added to a penalty. One of the basic assumptions in this framework is restricted strong convexity meaning that the second derivative of the loss function is bounded away from zero uniformly over a set of directions. In the case of linear regression and generalized linear models, this assumption essentially boils down to verifying a restricted eigenvalue (RE) condition on the sample covariance matrix.
%%%%%%%%%%%%%%%%%%%%%%%%%%%%%%%%%%%%%%%%%%%%%%%
\item \emph{Linear Regression via Lasso.} One of the most popular and possibly the first high dimensional linear regression technique is the Lasso introduced by \cite{Tibs96}. The general results of \cite{Neg12} provide an easy recipe for studying the rate of convergence of the Lasso estimator. Based on this general recipe and equipped with the verification of the RE condition, we prove in Section~\ref{sec:HDLinReg} (via Theorems \ref{thm:LassoRate} and \ref{thm:LassoPoly}) the rate of convergence of the Lasso estimator to be $\sqrt{k\log p/n}$ (the near minimax optimal rate) under sub-Weibull covariates and sub-Weibull/polynomial-tailed errors when the ``true'' regression parameter is assumed to be $k$-sparse. We also \emph{allow} for both fixed and random designs, as well as for misspecified models. %This rate was proved to be minimax optimal in \cite{Raskutti11}.
    Apart from admitting several other extensions (see Remark \ref{rem:LassoExtensions}), our results \emph{only} assume a marginal sub-Weibull property of the covariates, thus making them stronger than most existing results for Lasso which usually provide the rates under jointly sub-Gaussian/sub-exponential covariate vectors.
    %\tcm{
    To our knowledge, these are the first such results for the Lasso obtained in this generality.
    %}
\end{enumerate}

% \tcm{Need to talk about where we want to put the Emp. Process and CLT results here - AC.}

%\paragraph*{Organization}
\subsection{Organization}
The rest of this paper is organized as follows. In Section~\ref{sec:Definition}, we define the class of sub-Weibull random variables and introduce the Generalized Bernstein-Orlicz norm. A detailed discussion of several useful and basic properties of the GBO norm is deferred to Appendix~\ref{AppSec:PropGBO}. Section~\ref{sec:Indep} provides several ready-to-use bounds for sums of independent mean zero sub-Weibull random variables. Using the results of Section~\ref{sec:Indep}, the fundamental statistical applications discussed above are studied in Section~\ref{sec:Applications} (via Sections \ref{sec:ElementWiseMax}-\ref{sec:HDLinReg} dedicated respectively to these four problems). We conclude with a summary and directions for future research in Section~\ref{sec:Conclusions}.

In the \hyperref[supp_mat]{Supplementary Material} (Appendices~\ref{sec:EmpProcess}--\ref{AppSec:EmpProcess}), we include additional results and technical materials that could not be accommodated in the main article. In Appendix~\ref{sec:EmpProcess}, we %further
provide some supplementary results on tail bounds for suprema of empirical processes with sub-Weibull envelopes, and maximal inequalities based on uniform and bracketing entropy. %are presented.
Proofs of all the results in Section~\ref{sec:Definition} (along with those in Appendix~\ref{AppSec:PropGBO}) and Section~\ref{sec:Indep} are presented in Appendices~\ref{AppSec:Definition} and \ref{AppSec:Indep}, respectively.  The results of Section~\ref{sec:Applications}, as well as Appendix~\ref{sec:EmpProcess}, are proved in Appendices~\ref{AppSec:Applications} and \ref{AppSec:EmpProcess}, respectively.

\section{The Generalized Bernstein-Orlicz (GBO) Norm}\label{sec:Definition}

We first recall the general definition of an Orlicz norm for random variables. For a historical account of Orlicz norms, and sub-Guassian, sub-exponential (and sub-Weibull) variables, we refer to Section 1 of \cite{Well17} and the references therein.
\begin{defi}[Orlicz Norms]\label{def:IncOrliczNorm}
%\emph{
Let $g:\,[0, \infty) \to [0, \infty)$ be a non-decreasing function with $g(0) = 0$. The ``$g$-Orlicz norm'' of a real-valued random variable $X$ is given by
%}
\begin{equation}\label{eq:AlphaOrliczNorm}
\norm{X}_{g} := \inf\{\eta > 0:\,\mathbb{E}\left[g(|X|/\eta)\right] \le 1\}.
\end{equation}
\end{defi}
The function $\norm{\cdot}_g$ on the space of real-valued random variables is not a norm unless $g$ is additionally a convex function. We define the $g$-Orlicz norm here under the only assumption of monotonicity of $g$, since in the following, convexity is not satisfied and is also not required.
It readily follows from \eqref{eq:AlphaOrliczNorm} that
\begin{equation}\label{eq:TailgOrlicz}
\mathbb{P}\left(|X| \ge \eta g^{-1}(t)\right) \le \frac{1}{t}\quad\mbox{for all}\quad t\ge 0.
\end{equation}

Two very important special cases of $g$ are given by $\psi_2(x) := \exp(x^2) - 1$ and $\psi_1(x) := \exp(x) - 1$, which correspond to \emph{sub-Gaussian} and \emph{sub-exponential} random variables, respectively. %and $\psi_1(x) = \exp(x) - 1$ which corresponds to \emph{sub-exponential} variables.
% It is easy to see that $\psi_2(\cdot)$ and $\psi_1(\cdot)$ are both convex implying that $\norm{\cdot}_{\psi_2}$ and $\norm{\cdot}_{\psi_1}$ are both proper norms.
As a generalization, we now define \emph{sub-Weibull} random variables as follows.
\begin{defi}[Sub-Weibull Variables]\label{def:subWeibull}
A random variable $X$ is said to be sub-Weibull of order $\alpha > 0$, denoted as sub-Weibull $(\alpha)$, if %$\norm{X}_{\psi_{\alpha}} < \infty$, where
\[
\norm{X}_{\psi_{\alpha}} < \infty, \quad \mbox{where} \;\; \psi_{\alpha}(x) \; := \; \exp\left(x^{\alpha}\right) - 1\quad\mbox{for}\;\; x \ge 0.
\]
\end{defi}
Based on this definition, it follows that if $X$ is sub-Weibull $(\alpha)$, then %of order $\alpha > 0$, then it satisfies the following tail bound
\begin{equation}\label{eq:TailSubWeibull}
\mathbb{P}\left(|X| \ge t\right) \le 2\exp\left(-\frac{t^{\alpha}}{\norm{X}_{\psi_{\alpha}}^{\alpha}}\right),\mbox{ for all }t\ge 0.
\end{equation}
The right hand side here resembles the survival function of a Weibull random variable of order $\alpha > 0$, and hence the name sub-Weibull random variable. It is also clear from inequality~\eqref{eq:TailSubWeibull} that the smaller the $\alpha$ is, the more heavy-tailed the random variable is.

A simple calculation implies that a converse of the tail bound result in \eqref{eq:TailSubWeibull} also holds. It can further be shown that $X$ is sub-Weibull of order $\alpha$, if and only if, its moments satisfy
\begin{equation}\label{eq:MomentSubWeibull}
\sup_{r\ge 1}\,r^{-1/\alpha}\norm{X}_r < \infty,
\end{equation}
where $\norm{X}_r := \left(\mathbb{E}\left[|X|^r\right]\right)^{1/r}$; see Propositions 2.5.2 and 2.7.1 of \cite{Vershynin18} for similar results. Clearly, sub-exponential and sub-Gaussian random variables are sub-Weibull of orders $1$ and $2$ respectively, while bounded variables are sub-Weibulls of order $\infty$.
%As the names suggest, Weibull random variable with any scale parameter $\lambda > 0$ and a shape parameter $\alpha > 0$ is a sub-Weibull random variable of order $\alpha$.
Also, $X$ is sub-exponential if and only if $|X|^{1/\alpha}$ is sub-Weibull of order $\alpha$;
this follows readily from Definition~\ref{def:subWeibull}.

Next, to define the Generalized Bernstein-Orlicz norm, we first recall the classical Bernstein inequality for sub-exponential random variables. Suppose $X_1$, $\ldots$, $X_n$ are independent mean zero sub-exponential random variables, then
\begin{equation}\label{eq:SeparateBernstein}
\mathbb{P}\left(\left|\sum_{i=1}^n X_i\right| \ge t\right) \le \; 2 \times \begin{cases}\exp(-t^2/(4\sigma^2_n)),&\mbox{if }t < \sigma^2_n/C_n,\\
\exp(-t/(4C_n)), &\mbox{otherwise,}\end{cases}
\end{equation}
where $\sigma^2_n := 2\sum_{i=1}^n \norm{X_i}_{\psi_1}^2$ and $C_n := \max\{\norm{X_i}_{\psi_1}:\,1\le i\le n\}$; see Proposition 3.1.8 of \cite{GINE16}. Clearly, the tail of the sum behaves like a Gaussian for smaller values of~$t$ and behaves like an exponential for larger~$t$.

An equivalent way of writing inequality~\eqref{eq:SeparateBernstein} that leads to the Bernstein-Orlicz norm is
\[
\mathbb{P}\left(\left|\sum_{i=1}^n X_i\right| \ge \eta_1\sqrt{\sigma^2_n\log(1 + t)} + \eta_2C_n\log(1 + t)\right) \le \frac{1}{t},
\]
for some constants $\eta_1, \eta_2 > 0$. Comparing this inequality with~\eqref{eq:TailgOrlicz}, one can define an Orlicz norm through a function $g_{\eta}(\cdot)$ whose inverse is given by:
\[
g^{-1}_{\eta}(t) := \sqrt{\log(1 + t)} + \eta\log(1 + t),
\]
parametrized by $\eta > 0$. The corresponding Orlicz norm $\norm{\cdot}_{g_{\eta}}$ is exactly the Bernstein-Orlicz norm introduced by \cite{Geer13}. The Generalized Bernstein-Orlicz (GBO) norm is now defined analogously as follows.
\begin{defi}[Generalized Bernstein-Orlicz Norm]\label{def:GBOnorm}
Fix $\alpha > 0$ and $L \ge 0$. Define the function $\Psi_{\alpha, L}(\cdot)$ based on the inverse function
\begin{equation}\label{eq:InversePsi}
\Psi_{\alpha, L}^{-1}(t) := \sqrt{\log(1 + t)} + L\left(\log(1 + t)\right)^{1/\alpha}\quad\mbox{for all}\quad t\ge 0.
\end{equation}
The Generalized Bernstein-Orlicz (GBO) norm of a random variable $X$ is then given by $\norm{X}_{\Psi_{\alpha, L}}$ as in Definition~\ref{def:IncOrliczNorm}.
\end{defi}
\begin{rem}
\hspace{0.03in}It is easy to verify from~\eqref{eq:InversePsi} that $\Psi_{\alpha, L}(\cdot)$ is monotone and $\Psi_{\alpha, L}(0) = 0$ and so, Definition~\ref{def:IncOrliczNorm} is applicable. The function $\Psi_{\alpha, L}(\cdot)$ does not have a closed form expression in general, and is not convex for $\alpha < 1$. But $\norm{\cdot}_{\Psi_{\alpha, L}}$ is a quasi-norm; see Proposition \ref{prop:QuasiNorm} in Appendix \ref{AppSec:PropGBO}.  %Clearly, $\norm{\cdot}_{\Psi_{\alpha, L}}$ reduces to the Bernstein-Orlicz norm defined in \cite{Geer13} if $\alpha = 1$ (except for the difference of $L$ (here) and $L/2$ (there)).
%\end{rem}
%\par\smallskip
The properties proved for the Bernstein-Orlicz norm in~\cite{Geer13} also hold for the GBO norm $\norm{\cdot}_{\Psi_{\alpha, L}}$ even though the function $\Psi_{\alpha, L}(\cdot)$ is not convex for $\alpha < 1$. Several basic properties of the GBO norm, along with equivalent tail and moment bound properties and some maximal inequalities, are presented in Appendix~\ref{AppSec:PropGBO}. %along with the tail and moment equivalence properties and some maximal inequalities are presented in Appendix~\ref{AppSec:PropGBO}.
\end{rem}
\par\smallskip
The ready-to-use concentration inequality results in Section~\ref{sec:Indep} are presented in terms of the~$\norm{\cdot}_{\Psi_{\alpha, L}}$ norm and for this reason, we briefly mention here the precise nature of the \emph{tail behavior captured by the GBO norm.} If $\norm{X}_{\Psi_{\alpha, L}} < \infty$, then
\[
\mathbb{P}\left(|X| \ge \norm{X}_{\Psi_{\alpha, L}}\left\{\sqrt{t} + Lt^{1/\alpha}\right\}\right) \le 2\exp(-t)\quad\mbox{for all}\quad t \ge 0.
\]
So, for $t$ small enough, the survival function of~$X$ behaves like a Gaussian, and for~$t$ larger, the survival function behaves like a Weibull of order $\alpha$. Hence, the results from Section~\ref{sec:Indep} will imply that the tail of a sum of independent sub-Weibull random variables behaves like a \emph{combination} of a Gaussian tail and a Weibull tail.
%Another equivalent property is that $\norm{X}_{\Psi_{\alpha, L}} < \infty$ if and only if
%\[
%\sup_{p\ge 1}\frac{\norm{X}_p}{\sqrt{p} + Lp^{1/\alpha}} < \infty.
%\]
%Finally, we mention that for any $\alpha, L > 0$, $\norm{\cdot}_{\Psi_{\alpha, L}}$ is a quasi-norm.
\subsection{Sub-Weibull Random Vectors}\label{sec:randomvect}
For our applications, we consider the following two definitions of sub-Weibull random vectors.
For any vector $x\in\mathbb{R}^q$, let $x(j)$ represent the $j$-th coordinate of $x$ for all $1\le j\le q$, and let $\norm{x}_r := \left(\sum_{j=1}^q |x(j)|^r\right)^{1/r}$ denote the vector $L_r$-norm of $x$ for any $r \geq 1$. (For $r=2$, we sometimes also refer to the vector $L_2$-norm simply as the Euclidean norm.) %$\norm{x}_2$ represent the Euclidean ($L_2$) norm of $x$.
\begin{defi}[Joint Sub-Weibull Vectors]\label{def:JointWeibull}
A random vector $X\in\mathbb{R}^q$ is said to be {\cred \it jointly} sub-Weibull of order $\alpha > 0$ if for every $\theta\in\mathbb{R}^q$ of unit Euclidean norm, $X^{\top}\theta$ is sub-Weibull of order $\alpha$, and the {\cred \it joint sub-Weibull ($\alpha$) norm} of $X${\cred,  $\norm{X}_{J,\psi_{\alpha}}$ (where the subscript {\it ``$J$''} stands for {\it ``joint''}),} is given by
\[
\norm{X}_{J,\psi_{\alpha}} ~:=~ \sup_{\theta\in\mathbb{R}^q,\, \norm{\theta}_2 = 1}\,\norm{X^{\top}\theta}_{\psi_{\alpha}}.
\]
\end{defi}
\noindent This is one of the most commonly adopted type of tail assumptions on random vectors (especially with $\alpha = 2$); see Section 3.4 of \cite{Vershynin18}. As with random variables, the cases $\alpha = 1, 2$ correspond to sub-exponential and sub-Gaussian random vectors, respectively. %One possible origin of this definition is as follows. If $X\sim N_q(0, I_q)$, then $\theta^{\top}X \sim N(0, \theta^{\top}\theta)$ and hence $\theta^{\top}X$ is sub-Gaussian for every $\theta\in\mathbb{R}^q$ of unit Euclidean norm.
\begin{defi}[Marginal Sub-Weibull Vectors]\label{def:MarginalWeibull}
A random vector $X\in\mathbb{R}^q$ is said to be {\cred \it marginally} sub-Weibull of order $\alpha > 0$ if for every $1\le j\le q$, $X(j)$ is sub-Weibull of order $\alpha$, and the {\cred \it marginal sub-Weibull ($\alpha$) norm} of $X${\cred, $\norm{X}_{M,\psi_{\alpha}}$ (where the subscript {\it ``$M$''} stands for {\it ``marginal''}),} is given by
\[
\norm{X}_{M,\psi_{\alpha}} ~:=~ \sup_{1\le j\le q}\,\norm{X(j)}_{\psi_{\alpha}}.
\]
\end{defi}
\noindent Clearly, $\norm{X}_{M,\psi_{\alpha}} \le \norm{X}_{J,\psi_{\alpha}}$ for any random vector $X$, and hence, a marginal sub-Weibull property is (much) {\cred \it weaker} than a joint sub-Weibull property. A detailed comparison of the marginal and joint sub-Weibull properties is deferred to the beginning of Section~\ref{sec:Applications}.
%{\cmag **Add here a comments to clarify the ``J'' and ``M'' subscripts in the notations -- AC (2/10).**}

%\subsection{Extension to Multiple Regimes in the tail}
%%%%%%%%%%%%%%%%%%%%%%%%%%%%%%%%%%%%%%%%%%%%%%%%
%%%%%%%%%%%%%%%%%%%%%%%%%%%%%%%%%%%%%%%%%%%%%%%%
\section{Norms of Sums of Independent Random Variables}\label{sec:Indep}
The following sequence of results show the use of the $\Psi_{\alpha,L}$-norm in representing the part sub-Gaussian tail behavior in finite samples for sums of independent random variables when the ingredient random variables are sub-Weibull $(\alpha)$. %satisfy $\norm{X_i}_{\psi_{\alpha}} < \infty$ (irrespective of what $\alpha > 0$ is).
All results in this section are stated for independent random variables that are possibly non-identically distributed. Extensions to the case of dependent random variables also exist in the literature; see \cite{Merv11} and Appendix B of \cite{Kuch18}. The proofs of all the results in this section are given in Appendix~\ref{AppSec:Indep}.

\par\smallskip
The following result can be derived from Theorem 2 of \cite{LAT97}. (Note that the constants here are explicit, but they are not optimized and could possibly be improved.)
\begin{thm}\label{prop:SumNewOrliczVex}
If $X_1, \ldots, X_n$ are independent mean zero random variables %with mean zero
%such that
with $\norm{X_i}_{\psi_{\alpha}} < \infty$ for all $1\le i\le n$ and some $\alpha > 0$, then for any vector $a = (a_1, \ldots, a_n)\in\mathbb{R}^n$, the following bounds hold true:
\[
\norm{\sum_{i=1}^n a_iX_i}_{\Psi_{\alpha, L_n(\alpha)}} \le 2eC(\alpha)\norm{b}_2,
\]
and
\begin{equation}\mathbb{P}\left(\left|\sum_{i=1}^n a_iX_i\right| \ge 2eC(\alpha)\|b\|_2\sqrt{t} + 2eL_n^*(\alpha)t^{1/\alpha}\|b\|_{\beta(\alpha)}\right) \le 2e^{-t}\quad\mbox{for all}\quad t\ge 0, \label{eq:tail-bound-Latala}
\end{equation}
where $b = (a_1\norm{X_1}_{\psi_{\alpha}}, \ldots, a_n\norm{X_i}_{\psi_{\alpha}})\in\mathbb{R}^n$,
\begin{equation*}
C(\alpha) ~:=~ \max\{ \sqrt{2}, 2^{1/\alpha}\} \times \begin{cases}\sqrt{8}e^3(2\pi)^{1/4}e^{1/24}(e^{2/e}/\alpha)^{1/\alpha},&\mbox{if }\alpha < 1,\\
4e + 2(\log 2)^{1/\alpha},&\mbox{if }\alpha \ge 1,
\end{cases}
\end{equation*}
and for $\beta(\alpha) = \infty$ when $\alpha \le 1$ and $\beta(\alpha) = \alpha/(\alpha - 1)$ when $\alpha > 1$,
\begin{equation*}
L_n(\alpha) := \frac{4^{1/\alpha}}{\sqrt{2}\norm{b}_2}\times\begin{cases}\norm{b}_{\beta(\alpha)},&\mbox{if }\alpha < 1,\\
{4e\norm{b}_{\beta(\alpha)}}/{C(\alpha)},&\mbox{if }\alpha \ge 1,
\end{cases}
\end{equation*}
and for~\eqref{eq:tail-bound-Latala}, the quantity $L_n^*(\alpha) = L_n(\alpha)C(\alpha)\|b\|_2/\|b\|_{\beta(\alpha)}$.
\end{thm}
\iffalse
\begin{prop}\label{prop:SumNewOrliczCave}
If $X_1, \ldots, X_n$ are independent random variables with mean zero such that $\norm{X_i}_{\psi_{\alpha}} < \infty$ for all $1\le i\le n$ and some $\alpha \ge 1$, then for any vector $a = (a_1, \ldots, a_n)\in\mathbb{R}^n$, the following bound holds true:
\[
\norm{\sum_{i=1}^n a_iX_i}_{\Psi_{\alpha, L_n(\alpha)}} \le 2eC(\alpha)\norm{b}_2,
\]
where $b = (a_1\norm{X_1}_{\psi_{\alpha}}, \ldots, a_n\norm{X_n}_{\psi_{\alpha}})\in\mathbb{R}^n$,
\[
C(\alpha) := 4e + 2(\log 2)^{1/\alpha}\quad\mbox{and}\quad L_n(\alpha) := \frac{4^{1/\alpha}4e\norm{b}_{\beta}}{\sqrt{2}C(\alpha)\norm{b}_2},
\]
with $1/\alpha +  1/\beta = 1.$
\end{prop}
\fi
\begin{rem}\label{rem:NormBoundSharpness}\;\;{\cred (Sharpness of Theorem \ref{prop:SumNewOrliczVex}).}\hspace{0.01in}
Theorem~\ref{prop:SumNewOrliczVex} provides a useful generalization of Theorem 2.8.1 of~\cite{Vershynin18} for $\alpha \neq 1$. The transition in our result at $\alpha = 1$ is due to the fact that Weibull random variables are log-convex for $\alpha \le 1$ and log-concave for $\alpha \ge 1.$ It is worth noting that the conclusion of Theorem~\ref{prop:SumNewOrliczVex} cannot be improved in terms of dependence on $a = (a_1, \ldots, a_n)$ and are optimal in the sense that there exists distributions for $X_i$ satisfying $\norm{X_i}_{{\psi}_{\alpha}} \le 1$ for which there is a lower bound matching the upper bound;
see Theorem 2 and Examples 3.2 and 3.3 of \cite{LAT97}.
{\cred In particular, Examples 3.2 and 3.3 of~\cite{LAT97} {\cyan show} %shows
that the moment bounds implied by Theorem~\ref{prop:SumNewOrliczVex} (via Proposition~\ref{prop:EquivalenceTailMoment}) have matching lower bounds when the random variables $X_1, \ldots, X_n$ satisfy $\mathbb{P}(|X_i| \ge t) = \exp(-t^{\alpha})$ for all $t\ge0$.}
It should also be noted that these optimality results were also derived earlier by \cite{Gluskin95} and \cite{Hitczenko97}. %Theorem~\ref{prop:SumNewOrliczVex} provides a useful generalization of Theorem 2.8.1 of~\cite{Vershynin18} for $\alpha \neq 1$.
%see \cite{Nagaev79} for some moderate and large deviation results for sub-Weibull random variables.
{\cred In particular, for $\alpha \ge 1$, the corollary on page 307 of~\cite{Gluskin95} shows that the probability tail bound implied by Theorem~\ref{prop:SumNewOrliczVex} (via Proposition~\ref{prop:EquivalenceTailMoment}) is optimal in that there is a lower bound on the tail probability that only differs from the upper bound by a universal constant. We are not aware of a similar result for $\alpha < 1$. It is worth stressing here that the lower bounds mentioned in this remark should be understood in a minimax sense: there exists a distribution setting for independent random variables $X_1, \ldots, X_n$ for which the bound implied by Theorem~\ref{prop:SumNewOrliczVex} is sharp (i.e., Theorem~\ref{prop:SumNewOrliczVex} cannot be improved without further assumptions).}
\end{rem}

%{\cmag **ADD A REMARK/DISCUSSION SOMEWHERE HERE REGARDING LOWER BOUNDS (for addressing R2 Comment 3) -- we still don't about the proper details to fill in here (and in the response) -- AC.**}

\paragraph{Tail Bounds Scaling with Variance.}
The bound provided by Theorem~\ref{prop:SumNewOrliczVex} is solely in terms of $\norm{X_i}_{\psi_{\alpha}}$. It is clear, however, from the classical central limit theorem (CLT) that asymptotically the distribution of the sum (properly scaled) is determined by the variance of the sum. Although it is impossible to prove an exponential tail bound solely in terms of the variance, we expect at least the Gaussian part of the tail to depend on the variance only. This is the content of the next three results - Theorems \ref{prop:SimilarBernstein}--\ref{prop:BernsteinLargerThan1} (on norm bounds) and Theorem \ref{thm:MaximalTailBound} (on tail bounds). The proofs are based on the techniques of \cite{Adam08}. %These results can also be obtained using Corollary 1.8 of \cite{Nagaev79}.
%The bounds above are not of the same order as the variance of the sum of independent random variables but the variance (up to a constant) is expected to be the bound by the central limit theorem. It is important to get moment bounds for the sum of independent mean zero random variables that take into account the variance since in some of the applications discussed later $\norm{X_i}_{\psi_{\alpha}} \le 1$ but the variance of $X_i$ is very small. This case is similar to the hypothesis of Bernstein moment condition (see Remark \ref{rem:ProductVariables}). Also, this feature is desirable since the central limit theorem implies that the main contribution in the asymptotic distribution should be from the variance; see Remark \ref{rem:LinearKernel} and Section \ref{sec:Applications} for various applications stressing the importance of this feature. The following propositions take into account the cases with smaller variance.

%%\begin{thm}[{Bounds scaling with variance -- the case $\alpha \le 1$}]\label{prop:SimilarBernstein}
%% NOTE: When running Supp, MUST use a version WITHOUT color command. The tex file for the Supp does NOT run otherwise!! (Probably some issues with the capitalization etc.). Same happens with any color command in \paragraph title as well -- very strange!! Alternatively, can redefine the \cred command INSIDE the Supp file itself as: \def\cred{} -- 2/12/2022
\begin{thm}[{\cred Bounds Scaling with Variance -- the Case $\alpha \le 1$}]\label{prop:SimilarBernstein}
If $X_1, \ldots, X_n$ are independent  mean zero random variables %with mean zero
%such that
with $\norm{X_i}_{\psi_{\alpha}} < \infty$ for all $1\le i\le n$ and some $0 < \alpha \le 1$, then
\[
\norm{\sum_{i=1}^n X_i}_{\Psi_{\alpha, L_n(\alpha)}} \le 2e\sqrt{6}\left(\sum_{i=1}^n \mathbb{E}\left[X_i^2\right]\right)^{1/2},
\]
with
\[
L_n(\alpha) = \frac{4^{1/\alpha}K_{\alpha}C_{\alpha}}{2\sqrt{6}}(\log(n + 1))^{1/\alpha}\left(\sum_{i=1}^n \mathbb{E}\left[X_i^2\right]\right)^{-1/2}\max_{1\le i\le n}\norm{X_i}_{\psi_{\alpha}},
\]
%with
for some constants $C_{\alpha}, K_{\alpha} > 0$ depending only on $\alpha$.
\end{thm}
% \par\smallskip
%
The following result is the analogue of Theorem \ref{prop:SimilarBernstein} for the case $\alpha \ge 1$.
\begin{thm}[{\cred Bounds Scaling with Variance -- the Case $\alpha \ge 1$}]\label{prop:BernsteinLargerThan1}
If $X_1, \ldots, X_n$ are independent  mean zero random variables %with mean zero
%such that
with $\norm{X_i}_{\psi_{\alpha}} < \infty$ for all $1\le i\le n$ and some $\alpha \ge 1$, then
\[
\norm{\sum_{i=1}^n X_i}_{\Psi_{1, L_n(\alpha)}} \le 2e\sqrt{6}\left(\sum_{i=1}^n \mathbb{E}\left[X_i^2\right]\right)^{1/2},
\]
with%\vspace{-0.05in}
\[
L_{n}(\alpha) := \frac{4^{1/\alpha}C_{\alpha}}{2\sqrt{6}}(\log(n + 1))^{1/\alpha}\left(\sum_{i=1}^n \mathbb{E}\left[X_i^2\right]\right)^{-1/2} \max_{1\le i\le n}\norm{X_i}_{\psi_{\alpha}},
\]
for some constant $C_{\alpha} > 0$ depending only on $\alpha$.
\end{thm}
\paragraph{\cred Optimality of Theorems~\ref{prop:SimilarBernstein} and~\ref{prop:BernsteinLargerThan1}.}\label{para:OptimalityVariancebounds}
Theorem \ref{prop:BernsteinLargerThan1} proves a bound on the $\Psi_{1, L_n(\alpha)}$-norm irrespective of how light-tailed the initial random variables are (or in other words, how large $\alpha > 1$ is). Observe that this result reduces to the usual Bernstein's inequality for bounded random variables by taking $\alpha = \infty$. {\cred Bennet's inequality, which is a slight improvement of Bernstein's inequality~\citep{Well17}, is known to be optimal for bounded random variables, as shown in~\citet[Example 2.4]{major2005tail}.} In light of this, it seems not possible to prove Theorem \ref{prop:BernsteinLargerThan1} for a $\Psi_{\alpha, L}$-norm with $\alpha > 1$ {\cred as long as the bound is needed in terms of the variance}. Note further that even though the result uses the $\Psi_{1, L}$-norm, the parameter $L$ behaves as $(\log n)^{1/\alpha}/\sqrt{n}$ with the exponent of $\log n$ being $1/\alpha$ instead of $1$. So, this result cannot be obtained by simply applying Theorem~\ref{prop:SimilarBernstein} with~$\alpha = 1$.

%{\cmag MOVED THE PARA BELOW FROM BEFORE THEOREM 3.3 to AFTER THEOREM 3.3.}

Section 2.2 of \cite{Adam08} provides a counterexample proving that it is not possible to replace the factor $(\log (n + 1))^{1/\alpha}$ by anything of smaller order with only the hypothesis of $\norm{X_i}_{\psi_{\alpha}} < \infty$ if the norm bound is desired to be in terms of the variance itself. {\cred Formally, if we assume a bound of the form
\[
\mathbb{P}\left(\left|\sum_{i=1}^n X_i\right| \ge C\sqrt{t\sum_{i=1}^n \mathbb{E}[X_i^2]} + Ct^{1/\alpha_*}(\log n)^{u}\right) \le 3e^{-t}\quad\mbox{for all }t\ge0,
\]
holds true with some $u\ge 0$ for all independent mean zero random variables $X_1, \ldots, X_n$ satisfying $\|X_i\|_{\psi_{\alpha}} \le 1$, then $u \ge 1/\alpha$. This, again, should be understood in the minimax sense: for the result to hold for all distributions of $X_1, \ldots, X_n$, then $u$ must be at least $1/\alpha$. This follows from Section 2.2 of~\cite{Adam08} by considering (as $r\to\infty$) i.i.d. random variables $X_1 = \varepsilon_1Y_1, \ldots, X_n = \varepsilon_nY_n$ with $\mathbb{P}(Y_i = r^{1/\alpha}) = e^{-r} = 1 - \mathbb{P}(Y_i = 0)$ and $\varepsilon_1, \ldots, \varepsilon_n$ are independent Rademacher random varaibles; we refer the reader to~\cite{Adam08} for more details. Furthermore, Theorems~\ref{prop:SimilarBernstein} and~\ref{prop:BernsteinLargerThan1} can be considered optimal in light of {\cyan the} large deviation results from~\citet[Section III]{bakhshizadeh2020sharp}.}

The main advantage of Theorems~\ref{prop:SimilarBernstein} and~\ref{prop:BernsteinLargerThan1} over Theorem~\ref{prop:SumNewOrliczVex} is the appearance of the variance in the bound, as opposed to the $\norm{\cdot}_{\psi_{\alpha}}$ norm, at the cost of the $\log$ factor in $L_n(\alpha)$ (which also explains the gain in the logarithmic factor mentioned after Theorem 8 of \cite{Geer13}). This distinction \emph{can} impact the convergence rate if $\mathbb{E}(X_i^2)$ is of much smaller order than $\norm{X_i}_{\psi_{\alpha}}^2$; see Remark~\ref{rem:LinearKernel} for an example involving kernel smoothing estimators where this is indeed the case. {\cred Once again, we stress here that Theorem~\ref{prop:SumNewOrliczVex} can lead to a better tail bound if one does not care about dependence of the tail probability of the sum on the variance or if $\sqrt{\mbox{var}(X_i)}$ and $\|X_i\|_{\psi_{\alpha}}$ are of the same order.}

%{\cmag **ADD A REMARK (+ TECHNICAL DETAILS AS NEEDED) {\it SOMEWHERE HERE} (not sure about the exact best location yet) -- REGARDING SHARPNESS OF ALL THEOREMS HERE (THMs. 3.2, 3.3, and 3.4) in terms of the $\log(\cdot)$ factors. This for addressing R2 Comment 2. (Can use some or most of the material from the Overleaf file -- Sections 1 and 2 therein) -- AC.**}
\par\smallskip
Most of our examples in Section \ref{sec:Applications} involve the maximum of many averages. For this reason, we present a generally useful tail bound result for such maximums explicitly as a theorem below, although it is in fact a simple corollary of Theorems \ref{prop:SimilarBernstein} and \ref{prop:BernsteinLargerThan1} (depending on whether $\alpha \leq 1$ or $\alpha \geq 1$). %Only the case of $0 < \alpha \le 1$ is presented, and the case $\alpha > 1$ follows from Theorem \ref{prop:BernsteinLargerThan1} (below).
For a vector $v\in\mathbb{R}^q$, let $\norm{v}_{\infty}$ denote $\max_{1\le j\le q}|v(j)|$.
\begin{thm}[{\cred Tail Bounds for Maximums using Theorems~\ref{prop:SimilarBernstein}--\ref{prop:BernsteinLargerThan1}}]\label{thm:MaximalTailBound}
Suppose $X_1, \ldots, X_n$ are independent mean zero random vectors in $\mathbb{R}^q$, for any $q \geq 1$, such that for some $\alpha > 0$ and $K_{n,q} > 0$,
%\[
%\max_{1\le i\le n}\max_{1\le j\le q}\,\norm{X_i(j)}_{\psi_{\alpha}} \le K_{n,q}.
%\]
%Define
%\[
%\Gamma_{n,q} := \max_{1\le j\le q}\,\frac{1}{n}\sum_{i=1}^n \mathbb{E}\left[X_i^2(j)\right].
%\]
\[
\max_{1\le i\le n}\max_{1\le j\le q}\,\norm{X_i(j)}_{\psi_{\alpha}} \le K_{n,q}, \;\; \mbox{and define} \;\; \Gamma_{n,q} := \max_{1\le j\le q}\,\frac{1}{n}\sum_{i=1}^n \mathbb{E}\left[X_i^2(j)\right].
\]
Then for any $t\ge 0,$ with probability at least $1 - 3e^{-t}$,
\begin{align*}
\norm{\frac{1}{n}\sum_{i=1}^n X_i}_{\infty} \le 7\sqrt{\frac{\Gamma_{n,q}(t + \log q)}{n}} + \frac{C_{\alpha}K_{n,q}(\log(2n))^{1/\alpha}(t + \log q)^{1/\alpha^*}}{n},
\end{align*}
where $\alpha^* := \min\{\alpha,1\}$ and $C_{\alpha} > 0$ is some constant depending only on $\alpha$.
\end{thm}
%For $\alpha > 1$, a similar bound holds with $t + \log(p)$ instead of $(t + \log(p))^{\frac{1}{\alpha}}$ in the second term. %For both cases, all the $\log (2p)$ and $\log (2n)$ terms in the bounds may be replaced by  $\log (p+1)$ and $\log(n+1)$, respectively and the results still hold.
%For $p =1$, the bounds hold with all $\log(2p)$ terms removed.
\begin{rem}\;\;(Comparison with Existing Maximal Inequalities).\;\; One of the most important conclusions of Theorem~\ref{thm:MaximalTailBound} is a bound on the expectation of the maximum, which are usually referred to as \emph{maximal inequalities}. In particular, Theorem~\ref{thm:MaximalTailBound} yields
\begin{equation}
\mathbb{E}\left[\left\|\frac{1}{n}\sum_{i=1}^n X_i\right\|_{\infty}\right] ~\le~ C_1\sqrt{\frac{\Gamma_{n,q}\log(eq)}{n}} + C_2(\alpha, K_{n,q})\frac{(\log(2n))^{1/\alpha}(\log(eq))^{1/\alpha^*}}{n}, \label{eq:maximal-Thm-3_point_4}
\end{equation}
for a universal constant $C_1 \ge 0$ and a constant $C_2(\alpha, K_{n,q}) \ge 0$ is a constant depending only on $\alpha$ and $K_{n,q}$. (Here $e\approx 2.71$ represents the natural logarithm constant.) This bound compares favorably with the existing maximal inequalities applicable for this case. For instance, Lemma E.1 of~\cite{Chern17} yields the bound
\begin{align}
& \mathbb{E}\left[\left\|\frac{1}{n}\sum_{i=1}^n X_i\right\|_{\infty}\right] ~\le~ C_1'\sqrt{\frac{\Gamma_{n,q}\log(eq)}{n}} + C_2'(\alpha, K_{n,q})\frac{(\log(eqn))^{1/\alpha}\log(eq)}{n}, \label{eq:maximal-Chernozhukov}
\end{align}
where the constants $C_1'$ and $C_2'(\alpha, K_{n,q})$ are similar to $C_1, C_2(\alpha, K_{n,q})$ in~\eqref{eq:maximal-Thm-3_point_4}; see also Lemmas 3.4.2 and 3.4.3 of~\cite{VdvW96} for similar maximal inequalities. \eqref{eq:maximal-Chernozhukov} is the best possible inequality obtained from Lemma E.1 of~\cite{Chern17} %because
since the quantity $\sqrt{\mathbb{E}[M^2]}$ (in the referenced paper) can be bounded only by $(\log(eqn))^{1/\alpha}$ under the assumption of Theorem~\ref{thm:MaximalTailBound}. Comparing~\eqref{eq:maximal-Thm-3_point_4} and~\eqref{eq:maximal-Chernozhukov}, we note that the requirements for the average to converge to zero in the respective displays are given by: %respectively,
\[
\max\{\log(eq), (\log(2n))^{1/\alpha}(\log(eq))^{1/\alpha^*}\} = o(n) \quad \mbox{for \eqref{eq:maximal-Thm-3_point_4}},
\]
and
\[
\max\{\log(eq), (\log(eqn))^{1/\alpha}(\log(eq))\} = o(n) \quad \mbox{for \eqref{eq:maximal-Chernozhukov}}.
\]
The former is \emph{strictly better} than the latter, especially if $\log(eq) = O(n^{\gamma})$ for some $\gamma$. These two conditions match only when the random vectors are uniformly bounded vis-a-vis $\alpha = \infty$. Lemma E.1 of~\cite{Chern17} is improved by Proposition B.1 in~\cite{kuchibhotla2019least} which, in fact, is built on an earlier version of the current paper.
\end{rem}

\def\R{\mathbb{R}}
\def\P{\mathbb{P}}
\def\E{\mathbb{E}}
\def\psihat{\widehat{\psi}}
\def\psialpha{\psi_{\alpha}}

\begin{rem}\label{rem:LinearKernel}\;\;(Tail Bounds for Linear Kernel Averages: An Illustration of Theorem \ref{thm:MaximalTailBound}).\;\;
An important illustration of some of the main features of our results is in the derivation of (pointwise) deviation bounds for linear kernel average estimators (LKAEs) involving sub-Weibull variables. Such estimators are encountered in kernel smoothing based methods for non-parametric regression and density estimation.

Let $\{(Y_i, X_i): i =1, \hdots, n\}$ denote $n$ i.i.d. realizations of a random vector $(Y,X)$ having finite second moments, where $Y \in \mathbb{R}$ and $X \in \mathbb{R}^p$. Assume for simplicity that $X$ has a Lebesgue density $f(\cdot)$. Let $m(x) := \E(Y | X = x)$ and $\psi(x) := m(x) f(x)$. Let $K(\cdot): \R^p \rightarrow \R$ denote any kernel function (e.g. the Gaussian kernel on $\R^p$). Consider the following LKAE of $\psi(x)$, given by
\begin{equation*}
\psihat(x) \; := \; \frac{1}{nh^p} \sum_{i=1}^n Y_i K \left( \frac{X_i - x}{h} \right), \;\; \mbox{where} \; h \equiv h_n > 0 \; \mbox{is the bandwidth}.
\end{equation*}
Suppose $\norm{Y}_{\psialpha} \leq C_Y$ for some $\alpha, C_Y > 0$ and $g(x) := \E\left(Y^2 | X=x\right)f(x)$ is bounded, i.e. $0 \leq g(x) \leq M_Y$ for all $x$, for some constant $M_Y \geq 0$. Assume further that $K(\cdot)$ is bounded and square integrable, i.e. for some constants $C_K, R_K \geq 0$, $|K(x)| \leq C_K $ for all $x$ and $\int_{\R^p} K^2(x) dx \leq R_K$. Then, for any fixed $x \in \R^p$ and any $t \geq 0$, we have with probability at least $1 - 3 e^{-t}$,
\begin{equation}
\left| \psihat(x) - \E\{ \psihat(x) \} \right| \;\; \leq \;\; \frac{7 \,\Gamma_{Y,K} }{\sqrt{nh^p}} \sqrt{t} \; + \; \frac{C_{\alpha}\Upsilon_{Y,K}(\log(2n))^{1/\alpha}}{nh^p}t^{1/\alpha^*}, \label{bound:LinearKernel}
\end{equation}
where $\Gamma_{Y,K} := (M_Y R_K)^{\frac{1}{2}}$, $\Upsilon_{Y,K} := C_Y C_K$, $\alpha^* := \min\{\alpha,1\}$ and $C_{\alpha} > 0$ is some constant depending only on $\alpha$. \eqref{bound:LinearKernel} provides a ready-to-use deviation bound for sub-Weibull LKAEs with a convergence rate of $(nh^p)^{-1/2}$ for any $\alpha > 0$, assuming $nh^p \rightarrow \infty$ as $n \rightarrow \infty$.
Note that to extract this (sharp) rate, it is \emph{necessary} to exploit that $h^{-p}Y K\{ (X - x)/h \}$ has a variance of much \emph{smaller} order than its squared $\norm{\cdot}_{\psialpha}$ norm. The proof of \eqref{bound:LinearKernel} is given in Appendix~\ref{AppSec:Indep}.
Under standard smoothness conditions and a $q$-th order kernel $K(\cdot)$, for some $q \geq 2$, it can be shown that $|\E\{\psihat(x)\} - \psi(x)| \leq O(h^q)$ uniformly in $x$ (see, for instance, \citet{Hansen08} and references therein) and hence, a tail bound for $|\psihat(x) - \psi(x)|$ can also be obtained. The result provided here is mostly for illustration purposes and can possibly be extended in several directions; see Section \ref{sec:Conclusions} for further discussion.
\end{rem}

%\begin{rem}\label{rem:ProductVariables}\;\;(Orlicz Norms of Products of Random Variables).\;\;
\paragraph{Orlicz Norms of Products of Random Variables.}\label{rem:ProductVariables}

In all our results, the random variables are only required to be sub-Weibull of some order $\alpha > 0$. In many applications, one may need to deal with products of two or more such sub-Weibull variables. The following result (proved as Proposition \ref{prop:Product} in Appendix~\ref{AppSec:Indep})
provides a H\"{o}lder type inequality establishing a bound on the $\norm{\cdot}_{\psi_{\alpha}}$ norm of such product variables. The two examples mentioned in the introduction can also be easily dealt with using this result.
\iffalse
The classical Bernstein moment condition for a mean zero random variable $Z$ is given by
\[
\mathbb{E}\left[|Z|^p\right] \le \frac{p!}{2}\sigma^2K^{p-2}\quad\mbox{for}\quad p \ge 2.\label{eq:BMC}\tag{BMC}
\]
(See, for example, Theorem 1 of \cite{Geer13}.) This assumption is clearly satisfied for bounded random variables with $\sigma^2$ being the variance and $K$ is the bound on the random variable. All the results in this section are presented under the hypothesis of the type $\norm{X_i}_{\psi_{\alpha}} < \infty$ instead of any specialized hypothesis like the Bernstein Moment Condition \eqref{eq:BMC}. The reason for this is that in general it is hard to prove a BMC type condition with $\sigma^2$ being the variance of the random variable, especially for the product of random variables. However, the $\norm{\cdot}_{\psi_{\alpha}}$-norm bounds are easy to derive for product of random variables. More precisely, the following result holds:
\fi
If $W_i,$ $1\le i\le k$, are (possibly dependent) random variables satisfying $\norm{W_i}_{\psi_{\alpha_i}} < \infty$ for some $\alpha_i > 0$, then
\begin{equation}\label{eq:Product}
\norm{\prod_{i=1}^k W_i}_{\psi_{\beta}} \le \prod_{i=1}^k \norm{W_i}_{\psi_{\alpha_i}}\quad\mbox{where}\quad \frac{1}{\beta} := \sum_{i=1}^k \frac{1}{\alpha_i}.
\end{equation}
See also Lemma 2.7.7 of \cite{Vershynin18} for a similar result. %{\cmag **Restructure here to highlight the next part (answer to R2 comment 4) -- AC**}

\paragraph{Tail Bounds for Powers of Sub-Gaussians.}\label{rem:PowerSubGaussian}{\cred{\cyan As a simple application of the above discussion on products, coupled with our general results in this section, one can obtain tail bounds for powers of sub-Gaussians which are often useful in practice.} For example, {\cyan consider} $X_i = \varepsilon_i |G_i|^{r}, 1\le i\le n$ with $r\ge0$, Rademacher $\varepsilon_i$ and sub-Gaussian $G_i$. {\cyan Then, using \eqref{eq:Product}, $X_i$'s} satisfy $\|X_i\|_{\psi_{2/r}} \le \mathfrak{C} < \infty$ for some constant $\mathfrak{C}$ whenever $\|G_i\|_{\psi_2} \le \mathfrak{C}$. For such random variables, one can apply Theorem~\ref{prop:SumNewOrliczVex} or {\cyan Theorems~\ref{prop:SimilarBernstein}--\ref{prop:BernsteinLargerThan1}} %Theorem~\ref{prop:SimilarBernstein}
{\cyan to obtain a tail bound}.
%{\cmag **Why Thm. \ref{prop:SimilarBernstein} only? Depending on $r$, Thm.~\ref{prop:BernsteinLargerThan1} maybe needed too right? -- AC**}
Note that $\|X_i\|_{\psi_{2/r}} = \|G_i\|_{\psi_2} \le \mathfrak{C}$ and $\mathbb{E}[X_i^2] = \mathbb{E}[G_i^{2r}] \le \mathfrak{C}r^r$. %for some absolute constant $\mathfrak{C} < \infty$.
This implies that the standard deviation and the $\psi_{2/r}$-norm of the random variables are of the same order if $r$ is treated as a constant and the $2r$-th moment of $G_i$ is of the same order as $\|G_i\|_{\psi_2}^{2r}$. {\cyan Then,} Theorem~\ref{prop:SumNewOrliczVex} implies
\begin{equation}\label{eq:PowerSubgaussianTailBound}
\mathbb{P}\left(\left|\sum_{i=1}^n \varepsilon_i|G_i|^r\right| \ge \mathfrak{C}_r(nt)^{1/2} + \mathfrak{C}_rt^{r/2}n^{(1-r/2)_+}\right) \le 2e^{-t}\quad\mbox{for all}\quad t\ge0.
\end{equation}
Here, $\mathfrak{C}_r$ is a constant depending only on $r$ and $(u)_+ = \max\{u, 0\}$. In this case, {\cyan Theorems~\ref{prop:SimilarBernstein}--\ref{prop:BernsteinLargerThan1}} %Theorem~\ref{prop:BernsteinLargerThan1}
may yield a sub-optimal result because it does not account for the fact that the standard deviation and the $\psi_{2/r}$-norm are of the same order. If{\cyan, however, the} $2r$-th moments of $G_i$'s are not of the same order as $\|G_i\|_{\psi_2}^{2r}$, then Theorem~\ref{prop:SimilarBernstein} {\cyan or \ref{prop:BernsteinLargerThan1} (as the case may be)} yields a {\it better} tail bound. {\cyan Finally, note that we consider the symmetrized form  $\varepsilon_i |G_i|^{r}$ involving the Rademacher $\epsilon_i$'s here to ensure the random variables are all mean zero. A similar bound as \eqref{eq:PowerSubgaussianTailBound} continues to hold if $\varepsilon_i |G_i|^{r}$ is replaced by $|G_i|^{r} - \mathbb{E}(|G_i|^{r})$. Furthermore, the form $|G_i|^{r}$ with absolute value is considered to ensure it is well defined for any $r \ge 0$. A similar bound as \eqref{eq:PowerSubgaussianTailBound} continues to hold for $G_i^{r} - \mathbb{E}(G_i^{r})$ whenever $r$ is any positive integer.}} %{\cmag ** Add here that absolute value is not needed. $X_i^r$ is also fine, as long as $r$ is positive integer.}
%\end{rem}{eq:PowerSubgaussianTailBound}

%%%%%%%%%%%%%%%%%%%%%%%%%%%%%%%%%%%%%%%%%%%
%%%%%%%%%%%%%%%%%%%%%%%%%%%%%%%%%%%%%%%%%%%

\section{Applications in High Dimensional Statistics}\label{sec:Applications}

\paragraph{\cred Outline.}\label{para:Sec4Outline}
In this section, we study in detail the four fundamental statistical applications mentioned in the introduction{\cred, through Sections \ref{sec:ElementWiseMax}--\ref{sec:HDLinReg}. Below we first provide a {\it high-level organization} %of this section
-- in terms of the problems considered in each sub-section, and pointers to the corresponding main results and key discussions. A more detailed outline for each is provided within %at the beginning of
the respective sub-sections themselves.
\begin{enumerate}
\item Section \ref{sec:ElementWiseMax}  -- {\it Covariance matrix estimation in maximum elementwise norm}. (Main results: Theorems \ref{cor:Covariance} and Theorem \ref{thm:CovarianceMaxNorm} (in Section \ref{rem:CovarianceMaxNorm}); key discussions: Remarks \ref{rem:ExpectationBound} and \ref{rem:Boot}.)

\item Section \ref{subsec:SubMatrix} -- {\it Covariance matrix estimation in maximum $k$-sub-matrix operator norm, and the res- tricted isometry property (RIP)}. (Main result: Theorem \ref{cor:RIPBoundUnified}; key discussions: Remarks \ref{rem:RIPRates}--\ref{rem:RIPLiteratureComparison}.)

\item Section \ref{sec:RECondition} -- {\it The restricted eigenvalue (RE) and restricted strong convexity (RSC) conditions}. (Main result: Theorem \ref{cor:REBoundUnified}; key discussions: Remarks \ref{rem:Wang-Tewari}--\ref{rem:REExpTails}, as well as the results and associated discussions in Section \ref{rem:RECond} on verification of the RE condition for general sub-Weibulls.)

\item Section \ref{sec:HDLinReg} -- {\it High dimensional linear regression via Lasso}. (Main results: Theorems \ref{thm:LassoRate} and \ref{thm:LassoPoly}; key discussions: Remarks \ref{rem:RateLasso} and \ref{rem:RateLassoPoly}, as well as the general oracle inequality in Remark \ref{rem:LassoExtensions}.)
\end{enumerate}}

%{\cmag **Restructure here (add an enumerated list of Section organizations and the corresponding main results). Also, introduce next part as a separate visible paragraph -- AC.**}

\vspace{-0.1in}
\paragraph{{\cred A Discussion on Sub-Weibull Random Vectors: Joint vs. Marginal.}} Before proceeding to these applications, we provide a brief discussion that suggests that for random vectors the joint sub-Weibull property (Definition \ref{def:JointWeibull}), although commonly adopted in the literature (especially for the sub-Gaussian case; e.g., see \citet{Vershynin18}), is a much more restrictive assumption than the marginal one (Definition \ref{def:MarginalWeibull}). A careful examination of the joint sub-Weibull property implies an ``almost independence'' restriction on the coordinates for a dimension-free bound on the joint sub-Weibull norm.

As a simple (albeit a bit extreme) example, consider the random vector $X\in\mathbb{R}^q$ where all the coordinates are exactly the same $X(1) = \cdots = X(q)$. In this case, it is clear that
\begin{equation}\label{eq:PolyDep}
\norm{X}_{J, \psi_{\alpha}} = \sup_{\theta\in\mathbb{R}^q,\, \norm{\theta}_2 = 1}\,\norm{\theta}_1\norm{X(1)}_{\psi_{\alpha}} = \sqrt{q}\norm{X(1)}_{\psi_{\alpha}}.
\end{equation}
Although this is a pathological example, it shows that if the coordinates of $X$ are highly dependent, then the random vector \emph{cannot} have a ``small'' joint sub-Weibull norm; see Section 3.4 of \cite{Vershynin18} for a similar discussion. For all the high dimensional applications we consider, the (polynomial) dependence on the dimension in~\eqref{eq:PolyDep} can render the rates useless. Note that even though a Gaussian vector $X \in \R^q$ is jointly sub-Gaussian, $\norm{X}_{J,\psi_2}$
  will depend on the maximum eigenvalue of $\Sigma := \mbox{Cov} (X)$, which may not be dimension-free if $X$ has correlated components (e.g., if $\Sigma$ is an equicorrelation matrix). %The example above also clarifies that even if all the coordinates are uniformly bounded, the joint norm can be quite large.

The ``almost independence'' restriction implied by the joint sub-Weibull property may not necessarily be satisfied in practice and it is also hard to find results for high dimensional statistical methods in the literature under \emph{marginal} sub-Gaussian/sub-exponential tails. So, we consider both the marginal and the joint sub-Weibull assumptions in deriving the tail bounds as well as the rates of convergence in all the following statistical applications.

\subsection{Covariance Matrix Estimation: Maximum Elementwise Norm}\label{sec:ElementWiseMax}

\paragraph{{\cred Outline.}}\label{para:Sec4.1Outline} {\cred \hspace{-0.05in}In this section, we consider concentration properties of covariance matrices for sub-Weibulls under the maximum elementwise norm, which plays a crucial role in various high dimensional inference problems as well as in bootstrap. Our main result here is Theorem \ref{cor:Covariance} (along with Theorem \ref{thm:CovarianceMaxNorm} in Section \ref{rem:CovarianceMaxNorm} that further allows for data dependent centering). It proves a finite sample tail bound %for the sample covariance matrix in terms of this norm
 under the assumption of only {\it marginally} sub-Weibull $(\alpha)$ ingredient random vectors. Remark~\ref{rem:ExpectationBound} provides useful discussions on its implications and shows, in particular, the rate of convergence to be $\sqrt{\log p/n}$ if $\log p = o(n^{\alpha/(4-\alpha)})$. %; see Remark~\ref{rem:ExpectationBound}.
This rate can be easily shown to be optimal in case the random vectors are standard multivariate Gaussian. %Furthermore, using these results, we show
Finally, we discuss applications of these results in sparse covariance matrix estimation (Remark \ref{rem:SparseCov}) and in establishing consistency of bootstrap (Remark \ref{rem:Boot}) for (high dimensional) marginally sub-Weibull random vectors. Below we introduce the problem setup, followed by our results.} %and discussions.}

\par\smallskip
Suppose $X_1$ $,\hdots, X_n$ are independent random vectors in $\mathbb{R}^p$. Define the (gram) matrices
\begin{equation}\label{eq:CovarianceDefinition}
\hat{\Sigma}_n := \frac{1}{n}\sum_{i=1}^n X_iX_i^{\top}\quad\mbox{and}\quad \Sigma_n := \frac{1}{n}\sum_{i=1}^n \mathbb{E}\left[X_iX_i^{\top}\right].
\end{equation}
Note that $\hat{\Sigma}_n$ is unbiased for $\Sigma_n$. Assuming that $X_i$'s have mean $0$, $\Sigma_n$ is also the covariance matrix of the $\sqrt{n}$-scaled sample mean, $\sqrt{n}\bar{X}_n $, and $\hat{\Sigma}_n$ is a natural estimator of $\Sigma_n$. %As shown in Section \ref{sec:HDCLT}, it is necessary to control the elementwise maximum norm between the empirical and population covariance matrices to prove bootstrap consistency.

Define the \emph{elementwise maximum norm} of $\hat{\Sigma}_n - \Sigma_n$ as
\begin{align*}
\Delta_n := \vertiii{\hat{\Sigma}_n - \Sigma_n}_{\infty}
%&:= \max_{1\le j\le k\le p}\,\left|\hat{\Sigma}_n(j,k) - \Sigma_n(j,k)\right|,\\
&= \max_{1\le j\le k\le p}\,\left|\frac{1}{n}\sum_{i=1}^n \left\{X_i(j)X_i(k) - \mathbb{E}\left[X_i(j)X_i(k)\right]\right\}\right|.
\end{align*}
As shown in Remark 4.1 of~\cite{Chern17}, it is necessary to control $\Delta_n$, the elementwise maximum norm between the empirical and population covariance matrices, to establish consistency of the multiplier bootstrap.

Theorem \ref{cor:Covariance} below (proved in Appendix~\ref{AppSec:CovarianceMaxNorm}), the main result of this section, controls $\Delta_n$ under only
a \emph{marginal} sub-Weibull $(\alpha)$ assumption.
%is as follows.
Only the case $\alpha \le 2$ is considered here (the case $\alpha > 2$ can be derived similarly from Theorems~\ref{prop:BernsteinLargerThan1} and \ref{thm:MaximalTailBound}). Recall Definition~\ref{def:MarginalWeibull}.
\begin{thm}\label{cor:Covariance}
Let $X_1, \ldots, X_n$ be independent {\cred marginally sub-Weibull} random vectors in $\mathbb{R}^p$ satisfying
\begin{equation}\label{assump:Covariance}
\max_{1\le i\le n}\norm{X_i}_{M,\psi_{\alpha}} \le K_{n,p}  < \infty\quad\mbox{for some}\quad 0 < \alpha \le 2.
\end{equation}
%Set
%\[
%A_{n,p}^2 := \max_{1\le j\le k\le p}\,\frac{1}{n}\sum_{i=1}^n \mathrm{Var}\left(X_i(j)X_i(k)\right).
%\]
Fix $n,p \ge 1$. Then for any $t\ge 0$, with probability at least $1 - 3e^{-t}$,
\[
\Delta_n \le 7A_{n,p}\sqrt{\frac{t + 2\log p}{n}} + \frac{C_{\alpha}K_{n,p}^2(\log(2n))^{2/\alpha}(t + 2\log p)^{2/\alpha}}{n},
\]
where $C_{\alpha} > 0$ is a constant depending only on $\alpha$, and $A_{n,p}^2$ is given by
\[
A_{n,p}^2 := \max_{1\le j\le k\le p}\,\frac{1}{n}\sum_{i=1}^n \mathrm{Var}\left(X_i(j)X_i(k)\right).
\]
\end{thm}
\begin{rem}\label{rem:ExpectationBound}\;\;(Rate of Convergence).\;\;
Firstly, we reiterate that Theorem \ref{cor:Covariance} {\cred \it only} requires a marginal sub-Weibull assumption on the $X_i$'s, as in \eqref{assump:Covariance}. Next, it is clear from Theorem~\ref{cor:Covariance} that the rate of convergence of $\Delta_n$ is given by
\[
\Delta_n = O_p\left(\max\left\{A_{n,p}\sqrt{\frac{\log p}{n}}, K_{n,p}^2\frac{(\log n)^{2/\alpha}(\log p)^{2/\alpha}}{n}\right\}\right).
\]
Thus if $(\log p)^{2/\alpha - 1/2} = o(\sqrt{n}(\log n)^{-2/\alpha})$, then $\Delta_n = O_p\left(A_{n,p}\sqrt{\log p / n}\right)$.
%\[
%(\log p)^{2/\alpha - 1/2} = O\left(\frac{\sqrt{n}}{(\log n)^{2/\alpha}}\right),
%\]
%then
%\[
%\Delta_n = O_p\left(A_{n,p}\sqrt{\frac{\log p}{n}}\right).
%\]
It is easy to verify under assumption~\eqref{assump:Covariance} that $A_{n,p} \le C_{\alpha}K_{n,p}^2$; see Proposition 2.5.2 of \cite{Vershynin18} for a proof. Note that if $\alpha = 2$, i.e. $X_i$'s are marginally sub-Gaussian, then the (known) rate of convergence is $\sqrt{\log p/n}$. Thus, the key implication of the above calculations is that the \emph{rate of convergence can match that of the sub-Gaussian case} for a wide range of $\alpha > 0$. This is the main importance of the tail bounds stated in Section~\ref{sec:Indep} and the \emph{same phenomenon is observed in all %the
subsequent applications {\cred in Sections \ref{subsec:SubMatrix}--\ref{sec:HDLinReg}} too.}
Also, it is clear that the same result continues to hold under a (stronger) joint sub-Weibull assumption.
%Note that $\Delta_n$ is a scaled average which would be $O(1)$ if $p = O(1)$ as $n$ tends to infinity. Also, by definition
%\[
%A_{n,p}^2 = \max_{1\le j\le k\le p}\, \frac{1}{n}\sum_{i=1}^n \mathrm{Var}\left(X_i(j)X_i(k)\right),
%\]
%which is usually $O(1)$. Theorem \ref{cor:Covariance} implies the following expectation bound. For $n, p\ge 1$, under the hypothesis of Theorem \ref{cor:Covariance}, using inequality \eqref{eq:MaximalInequalityFirstMoment}
%\[
%\mathbb{E}\left[\Delta_n\right] \le C_{\alpha}\left\{A_{n,p}\sqrt{\log(1 + p^2)} + {K_{n,p}^2}(\log(1 + p^2)\log(1 + n))^{2/\alpha}n^{-1/2}\right\},
%\]
%for a constant $C_{\alpha}$ depending only on $\alpha$. This bound can loosely be written as
%\[
%\mathbb{E}\left[\Delta_n\right] \le O\left(\max\left\{A_{n,p}\sqrt{\log(p)}, n^{-1/2}\left(\log p\right)^{2/\alpha}(\log n)^{2/\alpha}\right\}\right).
%\]
%It follows that
%\[
%\Delta_n = O_p\left(\sqrt{\log p}\right)\quad\mbox{if}\quad (\log p)^{2/\alpha - 1/2} = O(n^{1/2}(\log n)^{-2/\alpha}).
%\]
%Hence, our results provide a rate of $\sqrt{\log p/n}$ for the maximum elementwise distance between $\hat{\Sigma}_n$ and $\Sigma_n$ under a weaker exponential tail assumption.
\end{rem}
\begin{rem}\;\;(Application to Coupling Inequality).\;\;
The quantity $\Delta_n$ also appears in a coupling inequality for the maximum of a sum of random vectors. The coupling inequality refers to bounding
\[
\left|\max_{1\le j\le p}\,\frac{1}{\sqrt{n}}\sum_{i=1}^n X_i(j) - \max_{1\le j\le p}\,\frac{1}{\sqrt{n}}\sum_{i=1}^n Z_i(j)\right|,
\]
where $X_i\in\mathbb{R}^p$ are mean zero and $Z_i \sim N_p(0, \mathbb{E}\left[X_iX_i^{\top}\right])$ constructed on the same probability space as $X_i$'s. For this quantity to converge in probability to zero, Theorem 4.1 of \cite{Chern14} requires $\Delta_n$ to converge to zero, among other terms.
\end{rem}
\subsubsection{Gram Matrix to Covariance Matrix (Accounting for Centering)}\label{rem:CovarianceMaxNorm}
%\begin{rem}\label{rem:CovarianceMaxNorm}\;\;(Estimation Error of Covariance Matrix).\;\;
The quantity $\Delta_n$ only measures the difference between the sample and the population \emph{gram matrices} that involve the \emph{uncentered} $X_i$'s, and this is important in applications involving linear regression since only the gram matrix directly appears there and not the covariance matrix. In some applications, however, it is of interest to deal with the \emph{covariance matrices}
\begin{equation}\label{eq:CovarianceMatrices}
\begin{split}
\hat{\Sigma}_n^* &:= \frac{1}{n}\sum_{i=1}^n \left(X_i - \bar{X}_n\right)\left(X_i - \bar{X}_n\right)^{\top}, \quad \mbox{and} \\
\Sigma_n^* &:= \frac{1}{n}\sum_{i=1}^n \mathbb{E}\left[\left(X_i - \bar{\mu}_n\right)\left(X_i - \bar{\mu}_n\right)^{\top}\right],
\end{split}
\end{equation}
where
$\bar{X}_n := \sum_{i=1}^n X_i/n$ and $\bar{\mu}_n := \mathbb{E}\left[\bar{X}_n\right] = \sum_{i=1}^n \mathbb{E}\left[X_i\right]/n.$
%\[
%\bar{X}_n := \frac{1}{n}\sum_{i=1}^n X_i,\quad\mbox{and}\quad \bar{\mu}_n := \mathbb{E}\left[\bar{X}_n\right] = \frac{1}{n}\sum_{i=1}^n \mathbb{E}\left[X_i\right].
%\]
Note, however, that $\Sigma_n^*$ is \emph{not} the variance of $\bar{X}_n$ unless $\mu_i = \bar{\mu}_n$ for all $i$. Define the maximum elementwise norm error between the sample and population covariance matrices $\hat{\Sigma}_n$ and $\Sigma_n$, respectively, as %for the covariance matrix as
\[
\Delta_n^* := \vertiii{\hat{\Sigma}_n^* - \Sigma_n^*}_{\infty}.
\]
Theorems~\ref{cor:Covariance} and~\ref{thm:MaximalTailBound} together imply the following result (proved in Appendix~\ref{AppSec:CovarianceMaxNorm}) for $\Delta_n^*$.
\begin{thm}\label{thm:CovarianceMaxNorm}
Under the setting of Theorem~\ref{cor:Covariance}, for any $t\ge 0$, with probability at least $1 - 6e^{-t},$
\[
\Delta_n^* \le 7A_{n,p}^*\sqrt{\frac{t + 2\log p}{n}} + \frac{C_{\alpha}K_{n,p}^2(\log(2n))^{2/\alpha}(t + 2\log p)^{2/\alpha}}{n},
\]
where
\begin{align*}
A_{n,p}^* &:= \max_{1\le j\le k\le p}\left(\frac{1}{n}\sum_{i=1}^n \mathrm{Var}\left[(X_i(j) - \bar{\mu}_n(j))(X_i(k) - \bar{\mu}_n(k))\right]\right)^{1/2}.
\end{align*}
\end{thm}
In comparison to Theorem~\ref{cor:Covariance} which applied to gram matrices, the only change with covariance matrices is the replacement of $A_{n,p}$ therein with $A_{n,p}^*$ as above.
%\end{rem}
\iffalse
\begin{rem}
One of the motivations for the study of $\Delta_n$ is its importance in consistency of multiplier bootstrap. A comparison of definitions in Section \ref{sec:HDCLT} and the current section shows that the $\bar{X}_n$ part of $\hat{\Sigma}_n$ is missing. To complete the proof of consistency of multiplier bootstrap, we need to bound $$\norm{\bar{X}_n}_{\infty}^2,$$ under the assumption of mean zero random vectors $X_1, \ldots, X_n$. A bound can be derived from Theorem~\ref{thm:MaximalTailBound}. It is easily seen that the bound on $\Delta_n$ is the dominating part in comparison with the bound on $\norm{\bar{X}_n}_{\infty}^2$.
\end{rem}
\fi
\begin{rem}\label{rem:SparseCov}\;\;(Applications in Sparse Covariance Matrix Estimation).\;\;
The basic technique of sparse covariance matrix estimation is thresholding. For simplicity, consider the case of identically distributed random vectors. Recall the definition of the usual covariance matrix $\hat{\Sigma}_n^*$ from~\eqref{eq:CovarianceMatrices} and define for $\lambda > 0$, the matrix $\breve{\Sigma}_{n, \lambda}$ given by
\[
\breve{\Sigma}_{n, \lambda}(j, k) := \begin{cases}\hat{\Sigma}_n^*(j, k), &\mbox{if }|\hat{\Sigma}_n^*(j,k)| \ge \lambda,\\
0, &\mbox{otherwise,}\end{cases}
\]
for $1\le j\le k\le p$. This estimator essentially sets to zero those elements of $\hat{\Sigma}_n^*$ that are ``small''. This is referred to sometimes as universal hard thresholding since $\lambda$ does not depend on $(j,k)$. The parameter~$\lambda$ is called the thresholding parameter. It is easy to verify that
\[
\mathbb{P}\left(\Sigma_n^*(j,k) = 0\mbox{ and }\breve{\Sigma}_{n,\lambda}(j,k) \neq 0\mbox{ for some }j,k\right) \le \mathbb{P}\left(\Delta_n^* > \lambda\right).
\]
So, the right cut-off $\lambda$ for consistent support recovery would be of the same order as the rate of convergence of $\Delta_n^*$ which is $\sqrt{\log p/n}$, as shown in Theorem~\ref{thm:CovarianceMaxNorm} (under additional conditions, as in Remark~\ref{rem:ExpectationBound}). So, for a wide range of $\alpha$, the cut-off used for Gaussians works for marginally sub-Weibull random vectors too. For a more careful study of the properties of $\breve{\Sigma}_{n, \lambda}$ in terms of the operator norm and extensions to weakly sparse matrices, see \cite{Bickel08}, \cite{Cai11} and \cite{Fan16}. As can be seen from the analysis there, a result similar to Theorem~\ref{thm:CovarianceMaxNorm} plays a key role. It should be noted here that most of the literature about covariance matrix estimation is based on a joint sub-Gaussian assumption on the ingredient random vectors. Our setting above is clearly more general.
\end{rem}
\begin{rem}\label{rem:Boot}\;\;(Bootstrap Consistency).\;\;
From Remark 4.1 and Theorem 4.2 of \cite{Chern17}, it follows that the consistency of either the multiplier bootstrap or Efron's empirical bootstrap for high dimensional averages requires the convergence of $\Delta_n^*$ to zero. In fact, the multiplier bootstrap error is bounded by a multiple of $(\Delta_n^*)^{1/3}$. Hence, our results in this section prove the bootstrap consistency under weaker tail assumptions.
\end{rem}

\subsection{Covariance Matrix Estimation: Maximum k-Sub-Matrix Operator Norm}\label{subsec:SubMatrix}

{\cred This section focuses on estimation of covariance matrices of sub-Weibull random vectors under the so-called sub-matrix operator norm.} In the previous sub-section, a bound on the elementwise maximum norm {\cred for such covariance matrices} was provided. It is clear that the maximum norm only deals with the elements of the matrix. In many applications and practical data exploration, it is of much more importance to study functionals of the covariance matrix such as the eigenvalues and eigenvectors. A key ingredient in studying these functionals is consistency of the covariance matrix in the operator norm.

As expected, if the dimension of the random vectors $X_i$ is larger than the sample size $n$, then the covariance matrix is \emph{not} consistent in the operator norm. Also, in high-dimensions it is a common practice to select a subset of ``significant'' group of coordinates of $X_i$'s and explore the properties of that subset. Motivated by this discussion, we study the \emph{maximum $k$-sparse sub-matrix operator norm} of the gram matrix, for any $1 \leq k \leq p$. This norm is also of importance in high dimensional linear regression due to its connections to the {\cred {\it restricted isometry property} (RIP) \citep{Candes07} and the restricted eigenvalue (RE) condition \citep{Bickel09}}.
Define, for $k \leq p$, %the {\cred maximum $k$-sparse sub-matrix operator norm as}
\begin{equation}\label{def:RIP-def}
\RIP_n(k) := \sup_{\substack{\theta\in\mathbb{R}^p,\\\norm{\theta}_0 \le k, \norm{\theta}_2 \le 1}}|\theta^{\top}(\hat{\Sigma}_n - \Sigma_n)\theta|,
\end{equation}
with $\hat{\Sigma}_n$ and $\Sigma_n$ as defined in \eqref{eq:CovarianceDefinition}. Here, $\norm{\theta}_0$ %$\norm{\theta}_0 := \sum_{j=1}^p \mathbbm{1}\{\theta_0(j) \neq 0\}$
denotes the number of non-zero entries (i.e. the \emph{sparsity}) of $\theta$. Note further that $\RIP_n(k)$ is actually a norm for $k\ge 2$.

The quantity $\RIP_n(k)$ also plays an important role in post-Lasso linear regression asymptotics (see condition RSE(m) in \cite{Belloni13}) and more generally, in post-selection inference (see \cite{Kuch18} for details). This norm was possibly first studied (with $\Sigma_n$ being the identity matrix) in \cite{Rudelson08} under the assumption of marginally bounded random vectors or equivalently, assumption~\eqref{assump:Covariance} with $\alpha = \infty$. Also see Appendix C of \cite{Belloni13} for similar results.

\vspace{-0.1in}
\paragraph{{\cred An Easier but Sub-Optimal Bound for $\RIP_n(k)$.}}\label{rem:RIPSuboptimalBound} Our main results on tail bounds for $\RIP_n(k)$ are presented in Theorem \ref{cor:RIPBoundUnified}. However, using the results of Section \ref{sec:ElementWiseMax}, \emph{an easier but generally sub-optimal bound} on $\RIP_n(k)$ may also be obtained which we present below for the sake of completeness. Note that
\[
\RIP_n(k) ~\le ~\left(\sup_{\norm{\theta}_0 \le k, \norm{\theta}_2 \le 1}\norm{\theta}_1^2\right)\vertiii{\hat{\Sigma}_n - \Sigma_n}_{\infty}~ \le ~ k\vertiii{\hat{\Sigma}_n - \Sigma_n}_{\infty}.
\]
This is a deterministic inequality and using the bounds on $\Delta_n$ derived previously in Section \ref{sec:ElementWiseMax}, it is easy to derive bounds for $\RIP_n(k)$. For simplicity, we only present here an expectation bound instead of general tail bounds (or moment bounds) for $\RIP_n(k)$. Under the hypothesis of Theorem~\ref{cor:Covariance} in Section~\ref{sec:ElementWiseMax}, we have
\begin{equation}\label{eq:ExpectationRIP}
\mathbb{E}\left[\RIP_n(k)\right] \le C_{\alpha}\left(A_{n,p}k\sqrt{\frac{\log p}{n}} + K_{n,p}^2\frac{k(\log p \log(2n))^{2/\alpha}}{n}\right),
\end{equation}
for some constant $C_{\alpha} > 0$ depending only on $\alpha.$ This bound provides the rate of $k\sqrt{\log p/n}$ only for $\RIP_n(k)$ using the arguments of Remark \ref{rem:ExpectationBound}. Note that this is derived {\it only} under a marginal $\psi_{\alpha}$-bound, and the factor $k$ here is, in a sense, optimal under the marginal $\psi_{\alpha}$-bound hypothesis as can be seen from the pathological example discussed before Section \ref{sec:ElementWiseMax}. (For this example, the factor $\sqrt{\log p}$ disappears from the rate.)
A bound alternative to \eqref{eq:ExpectationRIP} can be derived under the hypothesis of a joint $\psi_{\alpha}$ assumption on $X_i$. Under this joint hypothesis, the dominating term becomes $\sqrt{k\log p/n}$ which is the more familiar rate.

%In what follows,
\vspace{-0.1in}
\paragraph{{\cred Main Result (Outline).}}\label{para:Sec4.2Outline} We next derive a bound on $\RIP_n(k)$ in a unified way, using a different approach, that always presents the dominating term of the (optimal) order $\sqrt{k\log p/n}$ (upto a distributional constant factor) under either of these assumptions {\cred (i.e., marginal or joint sub-Weibull)}.
This is presented in Theorem \ref{cor:RIPBoundUnified} below (proved in Appendix~\ref{AppSec:SubMatrix}), the main result of this section. %{\cred The result is presented in two parts: (a) {\it marginal} case and (b) {\it joint} case.}
Once again, we only present the result for $0 < \alpha \le 2$ and a similar result for $\alpha > 2$ can be derived using Theorem~\ref{prop:BernsteinLargerThan1}. {\cred The result is presented in two parts: (a) {\it marginal} case and (b) {\it joint} case.}  {\cred The implications, including the behavior of the bound and its rate of convergence, as well as the sample complexity requirements under either cases, are discussed in detail in Remarks \ref{rem:RIPRates}--\ref{rem:RIPMarginalSampleComplexity}, followed by a thorough comparison with the existing literature on $\RIP$ in Remark \ref{rem:RIPLiteratureComparison}. Overall, to our knowledge, Theorem~\ref{cor:RIPBoundUnified}(a) is the first result on $\RIP_n(k)$ for the {\it marginal} case, while for the {\it joint} case, Theorem \eqref{cor:RIPBoundUnified}(b) matches existing (and optimal) results %(and the optimal scaling)
for the special case of sub-Gaussians (i.e., $\alpha =2 $), and also extends these to general sub-Weibulls.}

%A bound alternative to \eqref{eq:ExpectationRIP} can be derived under the hypothesis of a joint $\psi_{\alpha}$ assumption. Under this joint hypothesis, the dominating term becomes $\sqrt{k\log p/n}$. In what follows, we derive a bound on $\RIP_n(k)$ in a unified way that always presents the dominating term of order $\sqrt{k\log p/n}$.

%The main result of this section (proved in Section~\ref{AppSec:SubMatrix} of the supplementary material) is as follows. Once again, we only present the result for $0 < \alpha \le 2$ and the similar result for $\alpha > 2$ can be derived using Theorem~\ref{prop:BernsteinLargerThan1}.

\begin{thm}[{\cred Unified Bounds for $\RIP$}]\label{cor:RIPBoundUnified}
%\begin{thm}[{Unified Bounds for $\RIP$}]\label{cor:RIPBoundUnified}
Let $X_1, \ldots, X_n$ be independent random vectors in $\mathbb{R}^p$. Define
\[
\Theta_k := \{\theta\in\mathbb{R}^p:\,\|\theta\|_0 \le k, \|\theta\|_2 \le 1\}\quad\mbox{and}\quad \Upsilon_{n,k} := \sup_{\theta\in\Theta_k}\frac{1}{n}\sum_{i=1}^n \mathrm{Var}\left[\left(X_i^{\top}\theta\right)^2\right].
\]
Fix $0 < \alpha \le 2$. Then, for every $1 \leq k \leq p$,  the following bounds hold true for $\RIP_n(k)$ {\cred as in \eqref{def:RIP-def}}:
\begin{enumerate}
\item[(a)] {\cred {({\it Marginal} Sub-Weibull Case)}.~} If $\norm{X_i}_{M,\psi_{\alpha}} \le K_{n,p}$ for all $1\le i\le n$, then for any $t > 0$, with probability at least~$1 - 3e^{-t},$
\begin{equation}\label{eq:RIPMarginal}
\begin{split}
\RIP_n(k) &\le 14\sqrt{\frac{\Upsilon_{n,k}(t + k\log(36p/k))}{n}}\\ &\qquad+ \frac{C_{\alpha}K_{n,p}^2{k}(\log(2n))^{2/\alpha}(t + k\log(36p/k))^{2/\alpha}}{n}.
\end{split}
\end{equation}
\item[(b)] {\cred ({\it Joint} Sub-Weibull Case).~} If $\norm{X_i}_{J,\psi_{\alpha}} \le K_{n,p}$ for all $1\le i\le n$, then for any $t > 0$, with probability at least~$1 - 3e^{-t},$
\begin{equation}\label{eq:RIPJoint}
\begin{split}
\RIP_n(k) &\le 14\sqrt{\frac{\Upsilon_{n,k}(t + k\log(36p/k))}{n}}\\ &\qquad + \frac{C_{\alpha}K_{n,p}^2(\log(2n))^{2/\alpha}(t + k\log(36p/k))^{2/\alpha}}{n}.
\end{split}
\end{equation}
\end{enumerate}
Here, in both cases, $C_{\alpha} > 0$ represents a constant depending only on $\alpha$.
\end{thm}
Comparing between the two bounds from parts (a) and (b) above, the only difference is an extra factor of $k$ in the second term {\cred (for part (a) -- which uses the weaker assumption of marginal sub-Weibull)} which is usually of lower order than the first term. %{\cred Their convergence rates are discussed in Remark \ref{rem:RIPRates}.}

\begin{rem}\label{rem:RIPRates}\;\;(Rate of Convergence).\;\;
%A result of similar flavor is available as Theorem 5.41 of \cite{Ver10} which however assumes $\norm{X_i}_2 \le \sqrt{p}$. Note that the result of \cite{Ver10} is not about the sub-matrix operator norm but his techniques are still applicable for this problem.
%
The bounds \eqref{eq:RIPMarginal} and \eqref{eq:RIPJoint} {\cred -- obtained for the marginal and joint sub-Weibull cases, respectively --} both provide the same rate of $(\Upsilon_{n,k}k\log p/n)^{1/2}$ for a wide range of $k$ following the arguments of Remark~\ref{rem:ExpectationBound}, and this is actually what is expected from the central limit theorem as well. {\cred More details on the sample complexity requirements for both results are provided in Remarks \ref{rem:RIPJointSampleComplexity} and \ref{rem:RIPMarginalSampleComplexity}, for the joint and marginal cases, respectively. It is also worth mentioning that to the best of our knowledge, Theorem~\ref{cor:RIPBoundUnified}(a) is the {\it first} such result in the literature for the {\it marginal} case (that too for general sub-Weibulls), while for the {\it joint} case, Theorem \eqref{cor:RIPBoundUnified}(b) {\it matches} existing results for the special case of sub-Gaussians (i.e., $\alpha =2 $), while also {\it extending} them to general sub-Weibulls. Further discussions on comparison with the existing literature are provided in Remark \ref{rem:RIPLiteratureComparison}.}%Also, it is interesting to note that bound \eqref{eq:RIPMarginal} is strictly better than \eqref{eq:RIPJoint} for $\alpha < 2$ and the reverse holds for $\alpha \ge 2$.
\end{rem}

\begin{rem}\label{rem:UpsilonGrowth}\;\;(Growth of $\Upsilon_{n,k}$).\;\;  %% Needing to replace some (\cite{...}) commands here to \citep - done in Arxiv draft. May need to do this in IMAIAI final version as well to ensure consistency of citation style -- AC (5/9/2022). %%
The leading term in the bounds of Theorem \ref{cor:RIPBoundUnified} depends on $\Upsilon_{n,k}$ which relates to the fourth moment of linear combinations. Such quantities have also appeared in several other problems, including likelihood methods with diverging number of parameters \citep{Portnoy88}, sub-Gaussian estimation of means \citep{Joly17}, tail bounds for lower eigenvalues of covariance matrices \citep{Oliveira13} and verification of so-called small-ball conditions \citep{LecueMendelson14}.
In some of these works, the fourth moment of linear combinations is assumed to be bounded by the square of the second moment. Such an assumption coupled with a bounded operator norm of $\Sigma_n$ implies that $\Upsilon_{n,k}$ is of constant order. In general, $\Upsilon_{n,k}$ can grow with $k$ and it is not clear the rate at which it can grow for arbitrary distributions. However, under a joint sub-Weibull assumption as in part (b) of Theorem \ref{cor:RIPBoundUnified}, it is at most a constant multiple of $K_{n,p}^4$.
\end{rem}

%{\cmag **ADD NEW REMARK HERE ON RESULTS FOR JOINT SUB-WEIBULL CASE (Answer to AE Comment 3 AND R1 Comment 3) -- AC.**}

{\cred
\begin{rem}\label{rem:RIPJointSampleComplexity}\;\;(Conditions for Theorem \ref{cor:RIPBoundUnified}(b) -- the \emph{Joint} Sub-Weibull Case).\;\;
For the joint sub-Weibull case, the bound \eqref{eq:RIPJoint} for $\RIP_n(k)$ converges to zero whenever
\begin{equation}
n \gg \max\{k\log(ep/k),\, k^{2/\alpha}(\log(2n))^{2/\alpha}(\log(ep/k))^{2/\alpha}\}. \label{eq:RIPJoint-implication}
\end{equation}
In particular, for the special case of joint sub-Gaussian (i.e., $\alpha = 2$), the sample complexity requirement \eqref{eq:RIPJoint-implication} simplifies to: $n \gg k\log(ep/k)\log(n)$. We clarify that the appearance of the $\log(n)$ factor here is due to our usage of Theorem \ref{prop:BernsteinLargerThan1} (and ultimately Theorem \ref{thm:MaximalTailBound}) and {\it can} be avoided if instead one directly uses Theorem \ref{prop:SumNewOrliczVex} -- this essentially %directly
relates to our earlier discussion on the optimality of Theorems~\ref{prop:SimilarBernstein}--\ref{prop:BernsteinLargerThan1} (see Section \ref{para:OptimalityVariancebounds}). It is worth noting that the sample complexity $n \gg k\log(ep/k)\log(n)$ {\it matches} (possibly upto a $\log(n)$ factor) the scaling requirements of most results known in the literature, including those of \citet{candes2005decoding,Candes07,baraniuk2008simple} and \citet[Appendix G.1]{Loh12}, among several others. %for special cases of independent Gaussian ensembles, independent bounded or Fourier ensembles, as well as more general Gaussian and sub-Gaussian vectors.
In fact, most settings considered in the existing literature are included as special cases under our {\it joint} sub-Weibull setting for the choice $\alpha =2$; see Remark \ref{rem:RIPLiteratureComparison} for further details.
\end{rem}
}

\begin{rem}\label{rem:RIPMarginalSampleComplexity}\;\;{\cred (Conditions for Theorem \ref{cor:RIPBoundUnified}(a) -- the {\it Marginal} Sub-Weibull Case).}\;\;
%\;\;({\cmag **CHANGE DESCRIPTION TO HIGHLIGHT THIS IS MARGINAL CASE -- AC**} Conditions for Theorem \ref{cor:RIPBoundUnified}(a) and Comparison with Existing Results).\;\;
Firstly, {\cred before we discuss the bound \eqref{eq:RIPMarginal},}
we note that for the initial bound of $\RIP_n(k)$ provided in~\eqref{eq:ExpectationRIP}, although the rate obtained there is generally sub-optimal, convergence to zero of the bound therein requires
\begin{equation}
n \gg \max\{k^2\log p,\, k(\log p)^{2/\alpha}\}, \label{eq:max-norm-implication}
\end{equation}
whenever $X_i\in\mathbb{R}^p$ are marginally sub-Weibull, i.e., satisfy $\|X_i\|_{M,\alpha} < \infty$. However, for Theorem~\ref{cor:RIPBoundUnified}(a), which also requires only a marginal sub-Weibull property and provides a bound with a much sharper (and optimal) rate, convergence to zero of $\RIP_n(k)$ requires
\begin{equation}
n \gg \max\{k\log(ep/k),\, k^{1 + 2/\alpha}(\log(2n))^{2/\alpha}(\log(ep/k))^{2/\alpha}\}. \label{eq:RIP-implication}
\end{equation}
For $\alpha$ considerably smaller than $1$, the requirement \eqref{eq:RIP-implication} for Theorem~\ref{cor:RIPBoundUnified}(a) thus appears more stringent in terms of $k$, compared to \eqref{eq:max-norm-implication}. This deficiency can be explained by the fact that the proof of Theorem~\ref{cor:RIPBoundUnified} uses the bound $\|\max_{\theta\in\Theta_k}\theta^{\top}X_i\|_{\psi_{\alpha}} \le (k\log(ep/k))^{1/\alpha}$ (for applying %following
Theorem~\ref{thm:MaximalTailBound}), but using $\max_{\theta\in\Theta_k}\theta^{\top}X_i \le \sqrt{k}\|X_i\|_{\infty}$, we can get a sharper bound: $\|\max_{\theta\in\Theta_k}\theta^{\top}X_i\|_{\psi_{\alpha}} \le K_{n,p}\sqrt{k}(\log(ep))^{1/\alpha}$. {\cred Formally, Lemma E.1 of~\cite{Chern17} implies that
\begin{equation}\label{eq:Chernozhukov_application}
\begin{split}
\mathbb{E}\left[\sup_{\theta\in\Theta_k}\left|\frac{1}{n}\sum_{i=1}^n \{(\theta^{\top}X_i)^2 - \mathbb{E}[(\theta^{\top}X_i)^2]\}\right|\right] &\lesssim \sqrt{\frac{k\log(ep/k)}{n}}\sup_{\theta\in\Theta_k}\left(\frac{1}{n}\sum_{i=1}^n \mathbb{E}[(\theta^{\top}X_i)^4]\right)^{1/2}\\ &\quad+ \frac{k\log(ep/k)}{n}\left(\mathbb{E}\left[\max_{1\le i\le n}\sup_{\theta\in\Theta_k}|\theta^{\top}X_i|^4\right]\right)^{1/2}\\
&\lesssim \sqrt{\frac{k\log(ep/k)}{n}}\sup_{\theta\in\Theta_k}\left(\frac{1}{n}\sum_{i=1}^n \mathbb{E}[(\theta^{\top}X_i)^4]\right)^{1/2}\\ &\quad+ \frac{k^2(\log(epn))^{1 + 2/\alpha}}{n}.
\end{split}
\end{equation}
The second inequality above follows from $\max_{\theta\in\Theta_k}\theta^{\top}X_i \le \sqrt{k}\|X_i\|_{\infty}$ and the marginal sub-Weibull assumption.
% Noting that $\sup_{\theta\in\Theta}|\theta^{\top}X_i| \le k^{1/2}\|X_i\|_{\infty}$, the marginal sub-Weibull assumption implies that
% \[
% \left(\mathbb{E}\left[\max_{1\le i\le n}\sup_{\theta\in\Theta}|\theta^{\top}X_i|^4\right]\right)^{1/2} \le k\left(\mathbb{E}\left[\max_{1\le i\le n}\|X_i\|_{\infty}^4\right]\right)^{1/2} \lesssim k(\log(epn))^{2/\alpha}.
% \]
% Therefore,
% \begin{align*}
% \mathbb{E}\left[\sup_{\theta\in\Theta}\left|\frac{1}{n}\sum_{i=1}^n \{(\theta^{\top}X_i)^2 - \mathbb{E}[(\theta^{\top}X_i)^2]\}\right|\right]
% \end{align*}
The right hand side of~\eqref{eq:Chernozhukov_application} converges to zero whenever
\begin{equation}%\[
\max\{k\log(ep/k),\, k^2(\log(epn))^{1 + 2/\alpha}\} ~=~ o(n). \label{eq:RIPMarginalBetterComplexity}  %{\cyan,} %
\end{equation}%\]
}
{\cyan %which
\eqref{eq:RIPMarginalBetterComplexity} therefore improves our original requirement \eqref{eq:RIP-implication} for Theorem \ref{cor:RIPBoundUnified}(a), and matches \eqref{eq:max-norm-implication} upto a log factor.} However, we do not follow this approach further, mostly because we want to present the statistical applications as direct corollaries of the results in Section~\ref{sec:Indep}. In any case, \eqref{eq:max-norm-implication} at least shows $n \gg k^2(\log p)^{2/\alpha}$ suffices for $\RIP_n(k)$ to go to zero {\cred under only a marginal sub-Weibull assumption}. %In any case, combining~\eqref{eq:max-norm-implication} and~\eqref{eq:RIP-implication}, we only need $n \gg k^2(\log p)^{2/\alpha}$ at worst.
\end{rem}

%{\cmag ** Start a new remark from  here -- topic being literature comparison -- and split into two sub-types indep. rows and cols. Also, can possibly add some things/facts/refs from the notepad doc -- AC**}
%{\cmag ** QUESTION -- How do we re-word the first sentence here? And which displays to refer to (the Joint case right)? -- AC**)}

\begin{rem}\label{rem:RIPLiteratureComparison}\;\;{\cred (Comparison of Theorem \ref{cor:RIPBoundUnified} with Existing Literature).}\;\;
{\cred Theorem \ref{cor:RIPBoundUnified} succinctly provides a unified set of results on $\RIP_n(k)$ as in \eqref{def:RIP-def} under very general conditions -- both in terms of the tail behavior (i.e., sub-Weibull) as well as its nature (marginal vs. joint). To the best of our knowledge, the results in Theorem \ref{cor:RIPBoundUnified}(a) for the {\it marginal} case are the first such results in the literature obtained under a (much) weaker assumption than most in the existing literature on verifying $\RIP_n(k)$, and should therefore be of substantial interest in the future. Secondly, most of the literature on $\RIP_n(k)$ has focused on specific cases of our {\it joint} sub-Weibull setting in Theorem  \ref{cor:RIPBoundUnified}(b), and our convergence rates as well as sample complexity requirements, as discussed in Remarks \ref{rem:RIPRates} and \ref{rem:RIPJointSampleComplexity}, match these results -- these include the well known works of \citet{candes2005decoding,Candes07,baraniuk2008simple,Loh12}, among many others. \citet{Rudelson13} %, whose main focus was on a different but related problem of restricted eigenvalue condition,
provides a comprehensive review of the existing literature on verification of $\RIP_n(k)$ as we consider; see in particular their discussion in Section I (pg. 3434) and the references cited therein.}
{\cred To our knowledge, apart from the obvious flexibility (and novelty) of allowing for the marginal case, our results also enjoy the benefits of extension to the general sub-weibull case (i.e., for a general $\alpha$) {\it even} in the joint case, where most of the existing literature can be summarized as special cases in some form of our joint sub-Weibull setting with the choice $\alpha = 2$.}

{\cred It is worth noting that there has certainly been some work on the joint sub-Weibull setting (for a general $\alpha$) as well, but for a {\it different} RIP problem. This includes, in particular, the works of~\cite{Adam11} and \cite{guedon2014restricted, guedon2015interval}.}
%
%{\cmag ** This part to be carefully edited  -- AC**} Both \eqref{eq:max-norm-implication} and \eqref{eq:RIP-implication} seem to compare slightly unfavorably to the works of~\cite{Adam11} and \cite{guedon2014restricted, guedon2015interval}.
%~\cite{guedon2015interval} and \cite{guedon2014restricted}.
However, there are some important differences in their setting versus ours. Their definition of $\RIP$, translated in our notation, is given by
\begin{equation}\label{eq:adamczak-RIP}
\RIP_n^*(k) ~:=~ \sup_{\substack{\alpha\in\mathbb{R}^n,\\\|\alpha\|_0 \le k, \|\alpha\|_2 \le 1}}\, \left|\frac{1}{p}\sum_{j=1}^p \left|\sum_{i=1}^n \alpha(i)X_i(j)\right|^2 - 1\right|.
\end{equation}
Here, as in Theorem~\ref{cor:RIPBoundUnified}, $X_i\in\mathbb{R}^p$ are independent random vectors. The main difference between our $\RIP_n(k)$ in~\eqref{def:RIP-def} and $\RIP_n^*(k)$ above is that for the former, the sparse linear combination involved in \eqref{def:RIP-def} is being taken over coordinates of $X_i$'s, %in the former
while for $\RIP_n^*(k)$, the linear combination involved in~\eqref{eq:adamczak-RIP} is over $X_i$'s themselves. Thus, the maximizing space in~\eqref{eq:adamczak-RIP} is \emph{not} our $\Theta_k$; in fact, it is not even a subspace of $\mathbb{R}^p$ (rather it is in $\mathbb{R}^n$). Under the definition~\eqref{eq:adamczak-RIP} of $\RIP_n^*(k)$,
\cite{Adam11} proves that $n \gg k\log^{2/\alpha}(p/k)$ suffices for controlling $\RIP_n^*(k)$ when $X_i$'s are jointly sub-Weibull$(\alpha)$ with $\alpha \in [1, 2]$ ($\|X_i\|_{J,\alpha} < \infty$) and~\cite{guedon2015interval} extends this result to the case $\alpha \in (0, 1]$. Although this condition is %much
better in terms of dependence on $k$ {\cred compared to our requirement \eqref{eq:RIPJoint-implication} under the joint sub-Weibull case when $\alpha < 2$}, it does \emph{not} apply for our definition~\eqref{def:RIP-def} of $\RIP$. %and further, we would like to point out that our requirements~\eqref{eq:max-norm-implication} and~\eqref{eq:RIP-implication} are derived only under a marginal (not joint) sub-Weibull condition on the $X_i$'s.
{\cred Further, we would also like to point out that these works require the joint sub-Weibull assumption, while we allow for the marginal case as well.}  %our requirements~\eqref{eq:max-norm-implication} and~\eqref{eq:RIP-implication} are derived only under a marginal (not joint) sub-Weibull condition on the $X_i$'s.

We remark here that the goal of~\citet{Adam11,guedon2015interval} is to derive $\RIP$ constants for the reconstruction of sparse signals, and their definition of RIP as in \eqref{eq:adamczak-RIP} works for this purpose. However, they focus on the RIP of a very different type of matrices, ones with column spaces in $\mathbb{R}^n$, not $\mathbb{R}^p$, which makes their setting fundamentally \emph{different} from ours, and their results not directly comparable to ours either. Our main motivation for studying $\RIP_n(k)$, as  in~\eqref{def:RIP-def}, %mainly
stems from considering the approximation error between the sample and population covariance matrices, which is very much needed in post-selection inference applications~\citep{Kuch18} as well as in %high dimensional inference
{\cred high dimensional linear regression \citep{Candes07,Neg12,Wainwright_Book_2019}}. %inference.
\end{rem}

\begin{rem}\;\,(Gram Matrix to Covariance Matrix).\;\,
Using the results in Section~\ref{rem:CovarianceMaxNorm}, the results of this section can be easily modified to bound $\RIP_n(k)$ when the gram matrices $\hat{\Sigma}_n$ and $\Sigma_n$ are replaced by the covariance matrices $\hat{\Sigma}_n^*$ and $\Sigma_n^*$ respectively; see Remark 4.5 of \cite{Kuch18} for more details. Similar comments also apply for the results in the next section on the restricted eigenvalue condition and will not be repeated.
\end{rem}

\begin{rem}\;\;(Applications in Adaptive Covariance Matrix Estimation).\;\;
%Apart from being useful in proving the restricted eigenvalue condition,
Concentration inequalities for $\RIP_n(k)$ are also needed in adaptive estimation of a \emph{bandable} covariance matrix. A matrix $\Sigma_n\in\mathbb{R}^{p\times p}$ is said to be \emph{$k$-bandable}, for some $k \ge 1$, if
\[
\Sigma_n(i,j) = 0\quad\mbox{for all}\;\; |i - j| \ge k, \quad \mbox{for some} \;\; k \geq 1.
\]
An adaptive estimator was proposed in \cite{Cai12} based on the idea of block thresholding. Similar to the thresholding used in sparse covariance matrix estimation (see Remark \ref{rem:SparseCov}), block thresholding sets to zero a sub-matrix if its operator norm is smaller than a threshold. The actual procedure is more complicated than this and is described in Section 2.2 of \cite{Cai12}. Theoretical study of such a block thresholding procedure requires a result similar to Theorem~\ref{cor:RIPBoundUnified}; see Theorem 3.3 in Section 3.2 of \cite{Cai12} for more details. The main difference in comparison with our result is that we do not require sub-Gaussian tails whereas the proof of Theorem 3.3 there relies heavily on the normality of the random vectors; see also \cite{Cai16} for a survey about high dimensional structured covariance matrix estimation. Using our results from this section, the performance of the adaptive estimator can be studied under much weaker assumptions of marginal sub-Weibull tail behaviors.
\end{rem}
\subsection{Restricted Eigenvalue (RE) Condition}\label{sec:RECondition}
One of the most well known estimators for high dimensional linear regression %estimators for high dimensional data
is the Lasso \citep{Tibs96}. %which is obtained by minimizing
%\[
%\frac{1}{2n}\sum_{i=1}^n\left(Y_i - X_i^{\top}\theta\right)^2 + \lambda_n\norm{\theta}_1,
%\]
%over $\theta\in\mathbb{R}^p$, for random vectors $(X_i^{\top}, Y_i)^{\top}\in\mathbb{R}^{p+1}, 1\le i\le n$.
%One of the main assumptions
A crucial assumption in the proof of the oracle inequalities for Lasso is the {\cred \it restricted eigenvalue} (RE) condition introduced by \cite{Bickel09} for the matrix $\hat{\Sigma}_n$ as defined in \eqref{eq:CovarianceDefinition}; see Section~\ref{sec:HDLinReg} for further details on its application to the theoretical analysis of Lasso. {\cred This section focuses on the RE condition and its bounds for sub-Weibulls.} For any $1 \leq k \leq p$, the \emph{RE$(k)$ condition} on $\hat{\Sigma}_n$ is given by:
\begin{equation}
\inf_{\substack{S\subseteq\{1,\ldots,p\},\\|S|\le k}}\inf_{\theta\in\mathcal{C}(S; \delta)}\frac{\theta^{\top}\hat{\Sigma}_n\theta}{\theta^{\top}\theta} \; \ge \; \gamma_n > 0, \label{eq:RestrictedEigenvalue}
\end{equation}
for some constant $\gamma_n$, where for any subset $S\subseteq\{1,2,\ldots,p\}$ and any $\delta \ge 1$,
\begin{equation}\label{eq:ConeSet}
\mathcal{C}(S; \delta) := \left\{\theta\in\mathbb{R}^p:\,\norm{\theta(S^c)}_1 \le \delta\norm{\theta(S)}_1\right\}, \;\; \mbox{where}
\end{equation}
%Here,
$\norm{v}_1$ denotes the $L_1$ norm of any vector $v \in \R^p$, and $\theta(S)$ represents the sub-vector of $\theta$ with indices in $S$; see Equation (11.10) of \cite{Hastie15}. Note, however, that for the specific application of RE conditions in the analysis of Lasso type estimators, the first infimum in \eqref{eq:RestrictedEigenvalue} over all $S$ with $|S|\le k$ is \emph{not} needed. Instead it \emph{only} needs to be verified for $S$ being the true support of the regression parameter $\beta_0$ (as in Section \ref{sec:HDLinReg})  with $\|\beta_0\|_0 \leq k$. %(assumed to be $k$-sparse) as in Section \ref{sec:HDLinReg}.
\cite{Rudelson13} verified assumption \eqref{eq:RestrictedEigenvalue} for covariance matrices of sub-Gaussian random vectors,
 %This assumption was verified for covariance matrices of sub-Gaussian random vectors by \cite{Rudelson13},
extending the work of \citet{Raskutti10} for Gaussians. It is worth mentioning that the assumption of \cite{Rudelson13} is that of jointly sub-Gaussian random vectors. Some extensions under weaker tail behavior, including sub-exponentials have also been considered in \cite{Adam11} and \cite{LecueMendelson14}, for instance, although the latter's result applies more generally (see Remark \ref{rem:REExpTails} for more discussion).
 %for other related earlier works about the verification of the RE condition and restricted isometry property.

A general result proving this assumption based on a bound on the maximum elementwise norm is given in Lemma 10.1 of \cite{Geer09}. This result, coupled with our bounds on $\Delta_n$ in Section~\ref{sec:ElementWiseMax}, implies that if the random vectors $X_i$ are (marginally) sub-Weibull as in \eqref{assump:Covariance}, then $\hat{\Sigma}_n$ satisfies the RE$(k)$ condition \eqref{eq:RestrictedEigenvalue} with probability converging to 1 as long as $\Sigma_n$ satisfies its own corresponding RE$(k)$ condition and the following holds:
\begin{equation}\label{eq:simple-RE-complexity}
kA_{n,p}\sqrt{\frac{\log p}{n}} + K_{n,p}^2\frac{k(\log n)^{2/\alpha}(\log p)^{2/\alpha}}{n} = o(1).
\end{equation}
%As noted in Section 3.2 of \cite{Raskutti10},
This result does not allow for the optimal largest size for $k$, as noted in \citet[Section 3.2]{Raskutti10} as well, but it \emph{does} relax the sub-Gaussianity assumption largely. %It is, however,
Further, it is possible to get better rates using the bounds on $\RIP_n(k)$ from Section \ref{subsec:SubMatrix}, %under assumption \eqref{assump:Covariance},
as shown below.

\vspace{-0.05in}
\paragraph{{\cred Main Result (Outline).}}\label{para:Sec4.3Outline} \hspace{-0.07in}In the following, we prove that gram matrices obtained from marginal/joint sub-Weibull random vectors satisfy the RE condition with high probability. %using Lemma 12 of \cite{Loh12} and Theorem \ref{cor:RIPBoundUnified}.
The main result %in this regard
is Theorem \ref{cor:REBoundUnified} below which is proved (in Appendix \ref{AppSec:RECondition}) using Theorem \ref{cor:RIPBoundUnified}, and Lemma 12 of \cite{Loh12}. %The main result of this section (proved in Appendix \ref{AppSec:RECondition})
%Section \ref{AppSec:RECondition} in the supplementary material)
%as follows.
%
%We
Theorem \ref{cor:REBoundUnified} actually proves a \emph{stronger} result -- a sufficient condition regarding \emph{restricted strong convexity} (RSC), a notion %that was
introduced by \cite{Neg12}. As shown later in Section~\ref{rem:RECond}, the RE condition's verification follows directly from this result.  For simplicity, we again only consider the case $0 < \alpha \le 2$ (the case $\alpha > 2$ is similar). {\cred Similar to Theorem \ref{cor:RIPBoundUnified}, the result has two cases: (a) {\it marginal} and (b) {\it joint}. For each case, the corresponding bounds \eqref{eq:REMarginal} and \eqref{eq:REJoint} in Theorem \ref{cor:REBoundUnified} establish the RSC property under appropriate conditions. Furthermore, Section \ref{rem:RECond} verifies the RE condition under {\it both} marginal and joint sub-Weibull assumptions, and also provides the {\it optimized} sample complexities required in each case; see in particular \eqref{eq:complexity-diff-cases-joint} for the joint case and \eqref{eq:RE-complexity-Marginal-Case1}--\eqref{eq:RE-complexity-Marginal-Case2} for the marginal case. In particular, these match the existing (optimal) scaling known for the joint sub-Gaussian case. More details on the implications of the results, as well as comparisons with the existing literature on the RE condition are discussed in Remarks \ref{rem:Wang-Tewari}--\ref{rem:REExpTails}. To our knowledge, a unified set of results obtained in this generality for the RE condition is not available (at least not easily) within the core statistics literature.}%can be studied similarly). {rem:RECond}
\begin{thm}[{\cred RSC: Unified Bounds for Sub-Weibulls}]\label{cor:REBoundUnified}
Under the setting of Theorem \ref{cor:RIPBoundUnified} and recalling $\Upsilon_{n,s}$ as defined therein, %and for every $1 \leq k \leq p$,
the following high probability statements hold true: for every $1 \leq s \leq p$, \par\smallskip
%\begin{enumerate}
%\item[(a)]
\noindent (a) {\cred {({\it Marginal} Sub-Weibull Case)}.~} If $\norm{X_i}_{M,\psi_{\alpha}} \le K_{n,p}$ for all $1\le i\le n$, then setting
\begin{align*}
\Xi_{n,s}^{(M)} &:= 14\sqrt{2}\sqrt{\frac{\Upsilon_{n,s}s\log(36np/s)}{n}} %\\ &\qquad
 +  \frac{C_{\alpha}K_{n,p}^2{s}(\log(2n))^{\frac{2}{\alpha}}(s\log(36np/s))^{\frac{2}{\alpha}}}{n},
\end{align*}
we have with probability at least~$1 - 3s(np)^{-1},$ simultaneously for all $\theta\in\mathbb{R}^p$,
\begin{equation}\label{eq:REMarginal}
\theta^{\top}\hat{\Sigma}_n\theta \ge \left(\vphantom{\sum_{i=1}^N}\lambda_{\min}(\Sigma_n) - 27\Xi_{n,s}^{(M)}\right)\norm{\theta}_2^2 - \frac{54\Xi_{n,s}^{(M)}}{s}\norm{\theta}_1^2.
\end{equation}
%\item[(b)]
(b) {\cred {({\it Joint} Sub-Weibull Case)}.~} If $\norm{X_i}_{J,\psi_{\alpha}} \le K_{n,p}$ for all $1\le i\le n$, then setting
\begin{align*}
\Xi_{n,s}^{(J)} &:= 14\sqrt{2}\sqrt{\frac{\Upsilon_{n,s}s\log(36np/s)}{n}} %\\ &\qquad
 +  \frac{C_{\alpha}K_{n,p}^2(\log(2n))^{\frac{2}{\alpha}}(s\log(36np/s))^{\frac{2}{\alpha}}}{n},
\end{align*}
we have with probability at least~$1 - 3s(np)^{-1},$ simultaneously for all $\theta\in\mathbb{R}^p$,
\begin{equation}\label{eq:REJoint}
\theta^{\top}\hat{\Sigma}_n\theta \ge \left(\vphantom{\sum_{i=1}^N}\lambda_{\min}(\Sigma_n) - 27\Xi_{n,s}^{(J)}\right)\norm{\theta}_2^2 - \frac{54\Xi_{n,s}^{(J)}}{s}\norm{\theta}_1^2.
\end{equation}
%\end{enumerate}
Here, in both cases, $C_{\alpha} > 0$ represents a constant depending only on $\alpha$ (but possibly different in the two cases), and $\lambda_{\min}(\Sigma_n)$ denotes the minimum eigenvalue of $\Sigma_n$.
\end{thm}
\noindent \emph{Note:}
Bounds of the type \eqref{eq:REMarginal}--\eqref{eq:REJoint} were discussed in \cite{Neg12} as sufficient conditions for verifying their general RSC condition (Definition 2); see Eqns. (20) and (31) therein. %, for general loss functions and the linear regression case respectively.
With a slight abuse of terminology, we ignore %suppress
the distinction between their original RSC condition and these sufficient conditions, and call the latter by the same name here.

\begin{rem}\label{rem:Wang-Tewari}\;\;(Implications of Theorem~\ref{cor:REBoundUnified})\;\;
The ``parameter'' $s$ in Theorem~\ref{cor:REBoundUnified} is \emph{not} directly related to the sparsity $k$ of the regression parameter $\beta_0$ (as in Section \ref{sec:HDLinReg}). It is a free parameter that can be \emph{chosen} (or optimized) suitably over $1\leq s \leq p$. %optimized over.
E.g., if we take $s = 1$, then $\Xi_{n,s}^{(M)} = \Xi_{n,s}^{(J)}$ and both these quantities converge to 0 if $n \gg (\log(np))^{2/\alpha}(\log(2n))^{2/\alpha}$. This implies that with probability at least $1 - 3/(np)$, simultaneously for all $\theta\in\mathbb{R}^p$,
\begin{equation}\label{eq:RE-Wang-Tewari}
\theta^{\top}\hat{\Sigma}_n\theta ~\ge~ \left(\lambda_{\min}(\Sigma_n) - 27\Xi_{n,1}^{(J)}\right)\|\theta\|_2^2 - 54\Xi_{n,1}^{(J)}\|\theta\|_1^2.
\end{equation}
Note that $\Xi_{n,1}^{(J)} = C_1\sqrt{\log(np)/n} + C_2(\log(n)\log(np))^{2/\alpha}/n$ (treating $K_{n,p}$ and $\Upsilon_{n,s}$ as constants). The inequality~\eqref{eq:RE-Wang-Tewari} can be compared to Proposition 8 in the recent work of~\cite{Chung17} where a similar result is derived under a joint sub-Weibull assumption only, but allowing for dependence through $\beta$-mixing of the observations. In comparison, our result's sample complexity is similar to theirs, %here matches theirs,
while having a better (faster) coefficient for $\|\theta\|_1^2$.
\end{rem}

% \begin{rem}\label{rem:RECond}\;\;(Verification of the RE$(k)$ Condition).\;\;
\subsubsection[Verification of the RE Condition]{Verification of the RE($k$) Condition}\label{rem:RECond}
As mentioned earlier, Theorem~\ref{cor:REBoundUnified} proves a stronger sufficient condition regarding RSC (as we call it; see the note below Theorem \ref{cor:REBoundUnified}). %the restricted strong convexity property.
We now show that this indeed %property
implies the RE($k$) condition \eqref{eq:RestrictedEigenvalue}, where $k$ is set to denote the true sparsity of the regression parameter $\beta_0$ (as in Section \ref{sec:HDLinReg}). In our application for Lasso, we only need the RE$(k)$ condition \eqref{eq:RestrictedEigenvalue} with $\delta = 3$. Hence, for simplicity, we only prove %RE
\eqref{eq:RestrictedEigenvalue} with $\delta = 3$. %using Theorem~\ref{cor:REBoundUnified}.
To this end, first note that for any $S\subseteq\{1,2,\ldots,p\}$ with $|S| \le k$, and for any $\theta\in\mathcal{C}(S; 3)$, we have
\[
\norm{\theta(S^c)}_1 \le 3\norm{\theta(S)}_1 \le 3\sqrt{k}\norm{\theta(S)}_2 \quad\Rightarrow\quad \norm{\theta}_1 \le 4\sqrt{k}\norm{\theta}_2.
\]
Now, let $\Xi_s$ be either $\Xi_{n,s}^{(M)}$ or $\Xi_{n,s}^{(J)}$, as in Theorem~\ref{cor:REBoundUnified}, for \emph{any} $1\le s\le p$. The inequality above and Theorem~\ref{cor:REBoundUnified} then together imply that for any given $k$ and for all $1 \leq s \leq p$, with probability at least $1 - 3s/(np)$, simultaneously for all $S$ with $|S| \le k$ and for all $\theta\in \mathcal{C}(S; 3)$,
% ${54\Xi}\norm{\theta}_1^2/{k} \le 864\Xi\norm{\theta}_2^2$ and so,
\begin{equation}
\begin{split}
\theta^{\top}\hat{\Sigma}_n\theta  ~&\ge~ \left(\lambda_{\min}(\Sigma_n) - 27\Xi_s\right)\norm{\theta}_2^2 - \frac{54\Xi_s}{s}\norm{\theta}_1^2\\
~&\ge~ \left(\lambda_{\min}(\Sigma_n) - 27\Xi_s - \frac{864k\Xi_s}{s}\right)\norm{\theta}_2^2.\label{eq:free-parameter-s}
\end{split}
\end{equation}
The inequality \eqref{eq:free-parameter-s} holds for \emph{every} $1\le s\le p$. If we choose $s = k$, then $\lambda_{\min}(\Sigma_n) \ge 1782\Xi_k$ is needed to conclude the %restricted eigenvalue
 RE$(k)$ condition~\eqref{eq:RestrictedEigenvalue}
with $\gamma_n = \lambda_{\min}(\Sigma_n)/2$; see footnote 4 of \cite{Neg12} for a related calculation. Under a joint sub-Weibull assumption, this corresponds to {requiring $n\gg (k\log(np/k)\log n)^{2/\alpha}$} (treating $\Upsilon_{n,k}$ and $K_{n,p}$ as constants).
% the sample complexity:
% \[
% %\Upsilon_{n,k}k\log(np/k) + K_{n,p}^2(\log n)^{\frac{2}{\alpha}}(k\log(np/k))^{\frac{2}{\alpha}} = o(n),
% k\log(np/k) + (\log n)^{\frac{2}{\alpha}}(k\log(np/k))^{\frac{2}{\alpha}} = o(n).
% \]
For $\alpha = 2$, this reduces to the familiar requirement of $n \gg k \log(np/k)$ upto a $\log n$ factor.

Similarly, if one chooses $s = 1$ in \eqref{eq:free-parameter-s}, then $\lambda_{\min}(\Sigma_n) \ge 27(32k + 1) \Xi_1 \gtrsim k \Xi_1$ is needed to conclude the %restricted eigenvalue
RE$(k)$ condition~\eqref{eq:RestrictedEigenvalue} with $\gamma_n = \lambda_{\min}(\Sigma_n)/2$. Under either a marginal or a joint sub-Weibull assumption, this corresponds to a sample complexity similar to \eqref{eq:simple-RE-complexity}.

In general, one can \emph{optimize} the right hand side of~\eqref{eq:free-parameter-s} over $1\le s\le p$ in order to derive a better sample complexity. %{\color{red}%Because
Note that $\Xi_s + k\Xi_s/s$ is non-random, and so, if $s^{(o)}$ minimizes $\Xi_s + k\Xi_s/s$, then~\eqref{eq:free-parameter-s} implies that with probability at least $1 - 3s^{(o)}/(np) \ge 1 - 3/n$,
\begin{equation}\label{eq:general-RE-requirement}
\theta^{\top}\hat{\Sigma}_n\theta \ge \left(\lambda_{\min}(\Sigma_n) - 864\min_{1\le s\le p}\left\{\Xi_s + \frac{k\Xi_s}{s}\right\}\right)\|\theta\|_2^2\quad\mbox{for all}\quad \theta\in\bigcup_{|S|\le k}\mathcal{C}(S; 3).%, |S| \le k.
\end{equation}
It is, however, hard to minimize $\Xi_s + k\Xi_s/s$ exactly because it involves four different powers of $s$, and hence, we find a simpler upper bound on the minimum, separately, under the joint and marginal sub-Weibull assumptions, considering different sub-cases for $k$ in each case. (In the following calculations, we treat the quantities $\Upsilon_{n,s}$, for all $1\leq s\leq p$, and $K_{n,p}$ as constants, and also ignore any multiplicative constants for rate optimization purposes.)
\vspace{-0.05in}
\paragraph{Optimized Sample Complexity in the Joint Sub-Weibull Case.}
In light of \eqref{eq:general-RE-requirement} with $\Xi_s = \Xi_{n,s}^{(J)}$, we consider minimizing $\Xi_{n,s}^{(J)} + k\Xi_{n,s}^{(J)}/s$, over $1\le s\le p$. To this end, define
\begin{equation}
s^*_J \; := \; \frac{n^{\alpha/(4-\alpha)}}{(\log n)^{4/(4-\alpha)}\log(np)}. \label{eq:s-optimal-joint}
\end{equation}
This is the value\footnote{Formally, $s^*_J$ should be defined as the smallest integer that is larger than (or equal to) the right hand side in \eqref{eq:s-optimal-joint} above. But this minor adjustment is irrelevant for the purpose of rate or sample complexity calculations, and therefore, we disregard this technicality in our calculations here.\label{footnote1}}
of $s$ obtained by minimizing $k \Xi_{n,s}^{(J)}/s$ which consists of two terms that behave antagonistically with $s$ (i.e. one increases while the other decreases). To find the best sample complexity for the RE($k$) condition to hold, we now consider three cases:
\begin{equation}\label{eq:k-diff-cases-joint}%\label{eq:k-small}
\mbox{Case (i):} \;\; \alpha = 2, \;\; \mbox{or} \;\; \mbox{Case (ii):} \;\; \alpha < 2, \; k ~\le~ s^*_J, \;\; \mbox{or} \;\; \mbox{Case (iii):}  \;\; \alpha < 2, \; k > s^*_J.  %\frac{n^{\alpha/(4-\alpha)}}{(\log n)^{4/(4-\alpha)}\log(np)},
\end{equation}
For Cases (i) and (ii), we take $s = k$ in \eqref{eq:free-parameter-s}, with $\Xi_s \equiv \Xi_{n,s}^{(J)}$, to obtain: with probability at least $1 - 3k/(np)$, simultaneously for all $\theta\in\mathcal{C}(S; 3)$ and all $S$ with $|S| \leq k$,
%\[
%\theta^{\top}\widehat{\Sigma}_n\theta \ge \left(\lambda_{\min}(\Sigma_n) - 27\Xi_{n,k}^{(J)}\right)\|\theta\|_2^2 - \frac{54\Xi_{n,k}^{(J)}}{k}\|\theta\|_1^2,
%\]
%simultaneously for all $\theta\in\mathbb{R}^p$. For $\theta\in\mathcal{C}(S; 3)$, this implies
\[
\theta^{\top}\widehat{\Sigma}_n\theta \ge \left(\lambda_{\min}(\Sigma_n) - (27 + 864)\Xi_{n,k}^{(J)}\right)\|\theta\|_2^2.
\]
%simultaneously for all $\theta\in\mathcal{C}(S; 3)$ and all $S$ with $|S| \leq k$.

\par\smallskip
\noindent Now, under Case (i) in \eqref{eq:k-diff-cases-joint}, i.e. if $\alpha = 2$, we have:
\[
\Xi_{n,k}^{(J)} \; \lesssim \; \sqrt{\frac{k\log(np)}{n}} + \frac{(\log n)(k\log(np))}{n} \;\; = \; o(1) \;\; \mbox{whenever} \;\; n \gg k\log(np)\log n.
\]
%which converges to zero if $n \gg k\log(np)\log n$. Hence
%\[
%\mbox{if }\alpha = 2\mbox{ then, we need }n\gg k\log(np)\log n.
%\]
Under Case (ii)  in \eqref{eq:k-diff-cases-joint}, i.e. if $\alpha < 2$ and $k \leq s^*_J$, and with $\Xi_{n,s}^{(J)}$ monotone in $s$,  we have:
\begin{align*}
& \Xi_{n,k}^{(J)} \; \le \; \Xi_{n,s^*_J}^{(J)} \; \lesssim \; \left(\frac{n^{\alpha/(4-\alpha)}}{n(\log n)^{4/(4-\alpha)}}\right)^{1/2} \hspace{-0.05in} + \frac{1}{n}\left(\frac{n^{\alpha/(4-\alpha)}}{(\log n)^{4/(4-\alpha)}}\right)^{2/\alpha} = \; o(1) \;\; \mbox{whenever} \;\; n \gg 1. %\\
%&= o(1).
\end{align*}
%Hence
%\[
%\mbox{if }\alpha < 2\mbox{ and $k$ satisfies~\eqref{eq:k-small}, then we need }n \gg 1.
%\]
Finally, under Case (iii)  in \eqref{eq:k-diff-cases-joint}, i.e. if $\alpha < 2$ and $k > s^*_J$, appealing to \eqref{eq:general-RE-requirement}, bounding $\min_{1\le s\le p} \{\Xi_{n,s}^{(J)} + k\Xi_{n,s}^{(J)}/s\}$ suffices. We do so using the following inequalities:
%if
%\begin{equation}\label{eq:k-large}
%k ~\ge~ \frac{n^{\alpha/(4-\alpha)}}{(\log n)^{4/(4-\alpha)}\log(np)},
%\end{equation}
%then set
%\[
%s^{*}_J ~:=~ \frac{n^{\alpha/(4-\alpha)}}{(\log n)^{4/(4-\alpha)}\log(np)},
%\]
%and use the inequality
\begin{align*}
\min_{1\le s\le p} \left\{\Xi_{n,s}^{(J)} + \frac{k}{s}\Xi_{n,s}^{(J)}\right\} & \; \le \; \Xi_{n,s^*_J}^{(J)} + \frac{k}{s^*_J}\Xi_{n,s^*_J}^{(J)} \; \le \; \frac{2k}{s^*_J}\Xi_{n,s^*_J}^{(J)} \\
& \; \lesssim \; \frac{k}{\sqrt{s^*_J}}\sqrt{\frac{\log(np)}{n}} + \frac{k(s^*_J)^{(2-\alpha)/\alpha}}{n}(\log n)^{2/\alpha}(\log (np))^{2/\alpha}\\ %\lesssim \frac{k}{\sqrt{s^*_J}}\sqrt{\frac{\log(np)}{n}}\\
& \; = \; \frac{2k\log(np)(\log n)^{2/(4-\alpha)}}{n^{2/(4-\alpha)}} \; = o(1) \;\; \mbox{if} \;\; n \gg (k\log(np))^{2-\alpha/2}\log n.
\end{align*}
%%which converges to zero if $n \gg (k\log(np))^{2-\alpha/2}\log n$. Hence
%%\[
%%\mbox{if $k$ satisfies~\eqref{eq:k-large}, then we need }n\gg (k\log(np))^{2-\alpha/2}\log n.
%%\]
%Moreover, using \eqref{eq:free-parameter-s} with $s = s_{J}^{(o)} := \arg \min_{1 \leq s \leq p}\left(\Xi_{n,s}^{(J)} + k\Xi_{n,s}^{(J)}/s\right)$, and noting that $s/(np) \leq 1/n$ for all $1 \leq s \leq p$, we get
%with probability $\geq 1 - 3s_{J}^{(o)}/(np) \geq 1 - 3/n$,
%\begin{align}
%\theta^{\top}\widehat{\Sigma}_n\theta & \; \ge \; \left(\lambda_{\min}(\Sigma_n) - 27\Xi_{n,s_{J}^{(o)}}^{(J)} - 864 \frac{k}{s_J^{(o)}} \Xi_{n,s_{J}^{(o)}}^{(J)} \right)\|\theta\|_2^2 \\
%& \; \geq \; \left(\lambda_{\min}(\Sigma_n) - 864\Xi_{n,s_{J}^{(o)}}^{(J)} - 864 \frac{k}{s_{J}^{(o)}} \Xi_{n,s_{J}^{(o)}}^{(J)} \right)\|\theta\|_2^2 \\
%& \; = \; \left(\lambda_{\min}(\Sigma_n) - 864 \min_{1\le s\le p} \left[\Xi_{n,s}^{(J)} + \frac{k}{s}\Xi_{n,s}^{(J)}\right]  \right)\|\theta\|_2^2, \label{eq-REbound-optimal-joint}
%\end{align}
%simultaneously for all $\theta\in\mathcal{C}(S; 3)$ and all $S$ with $|S| \leq k$. This therefore guarantees that the RE$(k)$ condition holds with probability going to 1 as long as $\min_{1\le s\le p} \left(\Xi_{n,s}^{(J)} + \frac{k}{s}\Xi_{n,s}^{(J)}\right)$ is $o(1)$, for which the sample complexity requirement under Case (iii) is as derived above. %as shown above for Case (iii),

%\par\smallskip
Summarizing, we conclude that under a joint sub-Weibull assumption, the RE$(k)$ condition \eqref{eq:RestrictedEigenvalue} holds (with probability converging to 1 and with $\gamma_n = \lambda_{\min}(\Sigma_n)/2$) over all the cases in \eqref{eq:k-diff-cases-joint} with the following corresponding \emph{sample complexity requirements:}
\begin{equation}\label{eq:complexity-diff-cases-joint}
\mbox{(i):} \;\;  n \gg k\log(np)\log n, \;\; \mbox{(ii):} \;\; n \gg 1, \;\; \mbox{(iii):}  \;\;  n \gg (k\log(np))^{2-\alpha/2}\log n.  %\frac{n^{\alpha/(4-\alpha)}}{(\log n)^{4/(4-\alpha)}\log(np)},
\end{equation}
Combining all cases in \eqref{eq:complexity-diff-cases-joint}, we only require $n \gg (k\log(np))^{2-\alpha/2}\log n$. \qed %for RE$(k)$ to hold.
%}

%\par\medskip
%\noindent\emph{Optimized Sample Complexity for RE($k$) Condition under the Marginal Sub-Weibull Case.}
% \par\medskip
% \noindent\emph

%\par\smallskip
\vspace{-0.05in}
\paragraph{Optimized Sample Complexity in the Marginal Sub-Weibull Case.}
Again, in light of \eqref{eq:general-RE-requirement} with $\Xi_s = \Xi_{n,s}^{(M)}$, bounding $\min_{1\leq s \leq p}\{\Xi_{n,s}^{(M)} + k\Xi_{n,s}^{(M)}/s\}$ is sufficient. Define %To this end,
%we consider minimizing $\Xi_{n,s}^{(M)} + k\Xi_{n,s}^{(M)}/s$, over $1\le s\le p$. Define
\begin{equation}\label{eq:s-optimal-marginal}
s^*_M \; := \; \frac{n^{\alpha/(4 + \alpha)}}{(\log(np))^{(4-\alpha)/(4 + \alpha)}(\log n)^{4/(4+\alpha)}}.
\end{equation}
This is the value\footnote{See Footnote \ref{footnote1} earlier regarding $s_J^*$. The same comments apply here for $s_M^*$ as well and are not repeated.\label{footnote2}}
of $s$ obtained by minimizing $k \Xi_{n,s}^{(M)}/s$ which consists of two terms that behave antagonistically with $s$. %(i.e. one increases while the other decreases).
%To find the best sample complexity for the RE($k$) condition to hold,
We now consider two different cases for $k$:
\begin{equation}\label{eq:k-diff-cases-marginal}%\label{eq:k-small}
\mbox{Case (i):} \;\; k ~\le~ s^*_M, \;\;\; \mbox{or} \;\;\; \mbox{Case (ii):}  \;\; k > s^*_M.
\end{equation}
%For the marginal case, set
Under Case (i) in \eqref{eq:k-diff-cases-marginal}, i.e. when $k \le s^*_M$, noting that $\Xi_{n,s}^{(M)}$ is monotone in $s$, we can use
\begin{align}
\min_{1\le s\le p} &\left\{\Xi_{n,s}^{(M)} + \frac{k}{s}\Xi_{n,s}^{(M)}\right\}  \\
~&\le~ \Xi_{n,k}^{(M)} + \frac{k}{k}\Xi_{n,k}^{(M)} \; = \; 2\Xi_{n,k}^{(M)} \; \le \; 2\Xi_{n,s^*_M}^{(M)}\\
&   \lesssim~ \left(\frac{n^{\alpha/(4+\alpha)}(\log(np))^{2\alpha/(4+\alpha)}}{n(\log n)^{4/(4+\alpha)}}\right)^{1/2}%\\
%&
+ \frac{1}{n}\left(\frac{n^{\alpha/(4+\alpha)}(\log(np))^{2\alpha/(4+\alpha)}}{(\log n)^{4/(4+\alpha)}}\right)^{2/\alpha} \label{eq:RE-complexity-Marginal-Case1}\\
&  =~  o(1) \quad \mbox{if} \;\; n\log n \gg (\log(np))^{\alpha/2}\;\; \mbox{and} \;\; n(\log n)^{8/(2\alpha + \alpha^2)} \gg (\log(np))^{4/(2+\alpha)}.
\end{align}
%which converges to zero if
%\[
%n\log n \gg (\log(np))^{\alpha/2}\quad\mbox{and}\quad n(\log n)^{8/(2\alpha + \alpha^2)} \gg (\log(np))^{4/(2+\alpha)}.
%\]
Similarly, under Case (ii) in \eqref{eq:k-diff-cases-marginal}, i.e. if $k > s^*_M$, we can use
\begin{align}
& \min_{1\le s\le p}\left\{\Xi_{n,s}^{(M)} + \frac{k}{s}\Xi_{n,s}^{(M)}\right\}  \; \le \; \Xi_{n,s^*_M}^{(M)} + \frac{k}{s^*_M}\Xi_{n,s^*_M}^{(M)} \; \le \; \frac{2k}{s^*_M}\Xi_{n,s^*_M}^{(M)}\\
& \quad \lesssim \; \frac{k}{\sqrt{s^*_M}}\sqrt{\frac{\log(np)}{n}} + \frac{k(s^*_M)^{2/\alpha}}{n}(\log n)^{2/\alpha}(\log (np))^{2/\alpha}\\
& \quad = \; 2 k\frac{\sqrt{\log(np)}}{\sqrt{n}}\frac{(\log(np))^{(4-\alpha)/(8+2\alpha)}(\log n)^{4/(8+2\alpha)}}{n^{\alpha/(8+2\alpha)}} \; = \; 2k \frac{(\log(np))^{4/(4+\alpha)}(\log n)^{2/(4+\alpha)}}{n^{(2+\alpha)/(4+\alpha)}} \\
& \quad = \; o(1) \; \;\; \mbox{if} \;\; n \gg k^{(4 + \alpha)/(2+\alpha)}(\log(np))^{4/(2+\alpha)}(\log n)^{2/(2+\alpha)}. \label{eq:RE-complexity-Marginal-Case2}
\end{align}
%which converges to zero if %\tcr{seems like the exponent of \log(np) below will be
%\[
%n \gg k^{(4 + \alpha)/(2+\alpha)}(\log(np))^{4/(2+\alpha)}(\log n)^{2/(2+\alpha)}.
%\]
Thus, \eqref{eq:RE-complexity-Marginal-Case1} and \eqref{eq:RE-complexity-Marginal-Case2} provide the required sample complexities under Cases (i) and (ii) in \eqref{eq:k-diff-cases-marginal}. Combining both the cases, the requirement in \eqref{eq:RE-complexity-Marginal-Case2} %above inequality
suffices for the RE$(k)$ condition to hold (with probability converging to 1) under a marginal sub-Weibull assumption. \qed

\par\smallskip
The results above, therefore, reveal several interesting sample complexity requirements, apart from the expected ones, %under different growth regimes of $k$ and
under both the joint and marginal sub-Weibull assumptions. To the best of our knowledge, a unified and general set of results like this regarding the RE condition is not (easily) available/accessible within the core statistics literature. Lastly, we also point out that some of the logarithmic factors in the requirements above could possibly be improved (or removed) based on a more refined analysis which is not pursued here. %this is the first result proving the restricted eigenvalue condition in this generality.
% \end{rem}

\begin{rem}\label{rem:REExpTails}\;\;(Requirement/Relevance of Exponential Tails for RE Condition).\;\;
Observe that the RE condition is only concerned with the minimum sparse eigenvalue and so, the assumption of exponential tails may not be required in its full strength; see \cite{Geer14} and \cite{Oliveira13} for details. In particular, for this problem, it is only required to bound (possibly exponentially), $\mbox{for some} \; \varepsilon > 0$, the probability of the event
\[
\frac{1}{n}\sum_{i=1}^n \left(X_i^{\top}\theta\right)^2 \le (1 - \varepsilon)\frac{1}{n}\sum_{i=1}^n \mathbb{E}\left[\left(X_i^{\top}\theta\right)^2\right].
\]
%, being
Because this event is related to the average of non-negative random variables, it can have an exponentially small probability even under polynomial moment conditions; see Theorem 2.19 of \cite{DeLaPena09}, for instance, for an exponential tail bound under only finite
fourth moment conditions. \cite{Oliveira13} %essentially
formalizes this to bound the probability of the event uniformly over all $\theta$, and %In an earlier version, \cite{Oliveira13}
proved a general result related to the RE condition for a \emph{normalized} covariance matrix; see Theorem 5.2 there. Some of the main differences between his result and the results in high dimensional statistics literature are listed after Theorem 5.2 therein; see also Section 6.1 of \cite{Geer14} for more comparisons.

%The
Our marginal sub-Weibull assumption (a) in Theorem \ref{cor:REBoundUnified} is equivalent to the moment growth: $\norm{X_i(j)}_{r} \leq C_{\alpha}r^{1/\alpha}$ for \emph{all} $r\ge 1$. Under an additional so-called \emph{small-ball condition}, Theorem E of \cite{LecueMendelson14} shows that the same moment growth, but \emph{only} for $1\le r\le \log(wp)$, for some constant $w\ge 1$, suffices to verify the RE condition.
Note that for $p$ diverging with $n$, this weaker assumption of \cite{LecueMendelson14} is almost equivalent to a marginal sub-Weibull requirement. It is also not clear if Theorem E of \cite{LecueMendelson14}, which is primarily aimed at the RE condition's verification, can be extended to prove more general and stronger RSC type bounds of the form \eqref{eq:REMarginal}--\eqref{eq:REJoint}, as we obtain in Theorem \ref{cor:REBoundUnified}. Nevertheless, it must also be mentioned that their result on the RE condition requires a sample complexity of $n \gtrsim \max\{k\log p, (\log p)^{4/\alpha - 1}\}$ \emph{only}. This is a weaker condition (and perhaps the weakest known) than what we can achieve here.

As we noted {\cred earlier (e.g., at the end of Remark \ref{rem:RIPMarginalSampleComplexity}, though in a different context),} %in Remark \ref{cor:RIPBoundUnified} earlier,
our main goal throughout Section \ref{sec:Applications} is to demonstrate `easy' applications of the ready-to-use inequalities from Section \ref{sec:Indep} in handling these statistical problems and obtain unified and general, yet user-friendly, results for each of them. A  more targeted problem-specific approach, possibly using different techniques (and assumptions), can perhaps lead to slightly better results or conditions for some of these problems. Given the larger focus of this paper, we do not to pursue such deeper nuanced analyses here. %a more general restricted strong convexity property as in Theorem \ref{cor:REBoundUnified}.

Finally, we also remark that although it may be possible to prove the RE condition itself under weaker tail assumptions on the covariates and allowing for an exponential growth of $p$, the theoretical analysis of Lasso and other related high dimensional estimators --- where this condition is perhaps most needed --- usually \emph{requires} (almost) exponential tails for the covariates anyway to ensure logarithmic dependence on $p$ in the bounds and in the rates. %convergence rates.
\end{rem}

\subsection{High Dimensional Linear Regression}\label{sec:HDLinReg}
In this section, we derive results related to the Lasso, a well-known high dimensional linear regression estimator introduced by \cite{Tibs96}. Let $(X_1^{\top}, Y_1)^{\top}, \ldots,$ $(X_n^{\top}, Y_n)^{\top}$ be $n$ independent random vectors in $\mathbb{R}^p\times\mathbb{R}$. Let $\beta_0\in\mathbb{R}^p$ be a vector such that
\begin{equation}\label{eq:MisSpecify}
Y_i = X_i^{\top}\beta_0 + \varepsilon_i\quad\mbox{with}\quad \frac{1}{n}\sum_{i=1}^n \mathbb{E}\left[\varepsilon_iX_i\right] = 0\in\mathbb{R}^p.
\end{equation}
Observe that such a vector $\beta_0$ \emph{always} exists (regardless of whether or not $\mathbb{E}(Y_i\big|X_i)$ is linear), as long as the population gram matrix $\sum_{i=1}^n \mathbb{E}[X_iX_i^{\top}]/n$ is invertible, and is given by
\begin{align*}
\beta_0 &= \argmin_{\theta\in\mathbb{R}^p}\,\frac{1}{n}\sum_{i=1}^n \mathbb{E}\left[\left(Y_i - X_i^{\top}\theta\right)^2\right]\\ &= \left(\frac{1}{n}\sum_{i=1}^n \mathbb{E}\left[X_iX_i^{\top}\right]\right)^{-1}\left(\frac{1}{n}\sum_{i=1}^n \mathbb{E}\left[X_iY_i\right]\right).
\end{align*}
Differentiating the objective function above implies the second condition in~\eqref{eq:MisSpecify}. A linear model is said to be \emph{well-specified} if $\mathbb{E}\left[\varepsilon_i\big|X_i\right] = 0$ in which case the second condition in~\eqref{eq:MisSpecify} holds trivially, and $\mathbb{E}(Y_i\big|X_i)$ is exactly linear and equals $X_i^{\top}\beta_0$. Thus, the specification~\eqref{eq:MisSpecify} is a much weaker condition and allows for a \emph{misspecified} linear model. Note also that in \eqref{eq:MisSpecify}, $X$ is allowed to include $1$ to account for an intercept term.

The Lasso estimator $\hat{\beta}_n(\lambda)$ of $\beta_0$, for a regularization parameter $\lambda > 0$, is given by
\begin{equation}\label{eq:Lasso}
\hat{\beta}_n(\lambda) := \argmin_{\theta\in\mathbb{R}^p}\,\frac{1}{2n}\sum_{i=1}^n \left(Y_i - X_i^{\top}\theta\right)^2 + \lambda\norm{\theta}_1.
\end{equation}
%where $\norm{\theta}_1 := \sum_{j=1}^p |\theta(j)|$ denotes the $L_1$ norm of $\theta$.
In most of the literature on Lasso, the guarantees on the estimator are usually obtained under some restrictive assumptions such as fixed or jointly sub-Gaussian covariates and/or homoscedastic Gaussian/sub-Gaussian errors, although these are not the only settings studied; see \cite{Vidaurre13} and references therein for a detailed survey of $L_1$-penalized regression methods and their computational and theoretical properties.

\vspace{-0.05in}
\paragraph{{\cred Outline of the Main Results.}}\label{para:Sec4.4Outline} In this section, we analyze the Lasso under \emph{much weaker} than usual tail assumptions on the covariates $X_i$ as well as the errors $\varepsilon_i$. Our main results in this regard are Theorems \ref{thm:LassoRate} and \ref{thm:LassoPoly} (proved in Appendix~\ref{AppSec:HDLinReg}) below. Both results only assume a \emph{marginal} sub-Weibull property for the $X_i$'s, while for the errors, Theorem \ref{thm:LassoRate} assumes $\epsilon_i$ to be sub-Weibull  and Theorem \ref{thm:LassoPoly} only assumes $\epsilon_i$ to have polynomial tails (i.e. finite moments upto some order $r \geq 2$). Moreover, our analysis throughout \emph{allows} for (a) \emph{model misspecification}, and (b) \emph{both fixed and random covariates} since we do not assume identical distributions of the random vectors. The main message of both the results here is that the Lasso estimator attains the rate of $\sqrt{k\log p/n}$ for a large range of $k, p$ if $\beta_0$ is $k$-sparse. {\cred More details on both results and their implications, including their rates of convergence and applicability in various settings, are discussed in Remarks \ref{rem:RateLasso} and \ref{rem:RateLassoPoly}. Further extensions as well as a general {\it oracle inequality} for the Lasso are given in Remark \ref{rem:LassoExtensions}.} To the best of our knowledge, these results %Theorems \ref{thm:LassoRate}-\ref{thm:LassoPoly}
are among the very few (if not the only) results proving rates of convergence of the Lasso estimator in this generality, with one notable exception being a recent result from \citet{HanWellnerAOS19} which will be discussed later in the context of Theorem \ref{thm:LassoPoly}. %Our analysis allows for \emph{both fixed and random covariates} since we do not assume identical distributions of the random vectors.

\par\smallskip
A very general result about Lasso is obtained by \cite{Neg12} that is derived based on deterministic inequalities (see Section 4.2 therein). Both our main results here are based on this general result. We present Theorem \ref{thm:LassoRate} first. %The following main result (proved in Appendix~\ref{AppSec:HDLinReg}) is based on this general result.
Recall the definitions of $\hat{\Sigma}_n$, $\Sigma_n$ from %Equation~
\eqref{eq:CovarianceDefinition}, and $\Xi_{n,s}^{(M)}$ from Theorem~\ref{cor:REBoundUnified}(a), and also that $\norm{\beta_0}_0 $ denotes the sparsity of $\beta_0$.
%Also, recall $s^{(M),*}$ from~\eqref{eq:s-optimal-marginal}.%**********
% Our main result is as follows. Recall $\Xi_{n,k}^{(M)}$ from Theorem~\ref{cor:REBoundUnified}.
\begin{thm}[Lasso with Marginally Sub-Weibull $X_i$'s and Sub-Weibull $\varepsilon_i$'s]\label{thm:LassoRate}
Consider the setting above. Suppose $\norm{\beta_0}_0 \le k$ and there exists $0 < \alpha \le 2,$ and $\vartheta, K_{n,p} > 0$ such that
\[
\max\left\{\norm{X_i}_{M,\psi_{\alpha}}, \norm{\varepsilon_i}_{\psi_{\vartheta}}\right\} \le K_{n,p}\quad\mbox{for all}\quad 1\le i\le n.
\]
Also suppose $n\ge 2$, $k \ge 1$ and the matrix $\Sigma_n$ satisfies %$\lambda_{\min}(\Sigma_n) \ge 1782\Xi_{n,k}^{(M)}$, .
%% Suppose ***********************
%\begin{equation}
%\lambda_{\min}(\Sigma_n) \ge 1728\min_{1\le s\le p}\left\{\Xi_{n,s}^{(M)} + \frac{k\Xi_{n,s}^{(M)}}{s}\right\}, \;\; \mbox{\tcm{Alt. form (need this to use \eqref{eq:general-RE-requirement} directly).}}
%\end{equation}
\begin{equation}\label{eq:Re-satisfied}
\lambda_{\min}(\Sigma_n) \ge 54\min_{1\le s\le p}\left\{\Xi_{n,s}^{(M)} + \frac{32k\Xi_{n,s}^{(M)}}{s}\right\}, %\;\; \mbox{\tcr{I stuck to this orig. version and did the proof.}}
\end{equation}
with $\Xi_{n,s}^{(M)}$ as defined in Theorem~\ref{cor:REBoundUnified}(a).
%\begin{equation}\label{eq:CovarianceBddBelowBound}
%\lambda_{\min}(\Sigma_n) \ge 1782C_{\alpha}\left[\sqrt{\frac{k\log(ep/k)}{n}} + \frac{K_{n,p}^2k\log^{\frac{2}{\alpha}}(np)\left(k\log\left(\frac{ep}{k}\right) + \log^{\frac{2}{\alpha}}\left(\frac{np}{k}\right)\right)}{n}\right].
%\end{equation}
Then, with probability at least $1 - 3(np)^{-1} - 3n^{-1}$, the regularization parameter $\lambda_n$ can be chosen to be
\begin{equation}\label{eq:LambdaChoice}
\lambda_n = 14\sqrt{2}\sigma_{n,p}\sqrt{\frac{\log(np)}{n}} + \frac{C_{\gamma}K_{n,p}^2(\log(2n))^{1/\gamma}(2\log(np))^{1/\gamma}}{n},
\end{equation}
so that the Lasso estimator $\hat{\beta}_n(\lambda_n)$ satisfies
\[
\norm{\hat{\beta}_n(\lambda_n) - \beta_0}_2 \le \frac{84\sqrt{2}}{\lambda_{\min}(\Sigma_n)}\left[\sigma_{n,p}\sqrt{\frac{k\log(np)}{n}} + \frac{C_{\gamma}K_{n,p}^2k^{1/2}(\log(np))^{2/\gamma}}{n}\right],
\]
where $C_{\gamma} > 0$ is some constant depending only on $\gamma$ and
\[
\frac{1}{\gamma} := \frac{1}{\alpha} + \frac{1}{\vartheta},\quad\mbox{and}\quad \sigma_{n,p}^2 := \max_{1\le j\le p}\frac{1}{n}\sum_{i=1}^n \mathrm{Var}\left(X_i(j)\varepsilon_i\right) \; > 0.
\]
%Here, $C_{\alpha}, C_{\gamma}$ are some constants depending only on $\alpha$ and $\gamma$, respectively.
%Here, $C_{\gamma} > 0$ is some constant depending only on $\gamma$.
\end{thm}

%\textcolor{magenta}{Is this still true? Arun.}
\begin{rem}\label{rem:RateLasso}\;\;(Rate of Convergence and a Few Other Comments on Theorem \ref{thm:LassoRate}).\;\;
%\textcolor{red}{We believe that Theorem \ref{thm:LassoRate} is the first result proving rates of convergence of the Lasso estimator in this generality.}
%*******************
It follows from the result that if \eqref{eq:Re-satisfied} holds (as verified in Section \ref{rem:RECond}) %$\lambda_{\min}(\Sigma_n) \ge 1782\Xi_{n,k}^{(M)}$ holds,
and if $$(\log(np))^{4/\gamma - 1}=o(n),\quad\mbox{as}\quad n\to\infty,$$ then the rate of convergence of the Lasso is $\sqrt{k\log p/n}$ which is also known %in the literature
to be the (near) minimax optimal rate \citep{Raskutti11}. Note further that the probability guarantee in Theorem~\ref{thm:LassoRate} is converging to 1 as $n\to\infty$, and so, the bound therein has $\log(np)$ instead of the usual $\log p$. By making the probability to be $1 - O(p^{-1})$, the usual rate of $\sqrt{k\log p/n}$ can be recovered. In the special case of conditionally homoscedastic errors $\varepsilon_i$ with $\mathbb{E}(\varepsilon_i\big|X_i) = 0$ and $\mathrm{Var}(\varepsilon_i\big|X_i) = \sigma^2$, and with $X_i$'s normalized to have marginal variances of 1, we have $\sigma_{n,p} = \sigma$ and this then leads to the familiar rate of $\sigma\sqrt{k\log p/n}$ for the Lasso estimator.

%*******************
Lastly, if a \emph{joint} (instead of marginal) sub-Weibull property is assumed on the covariates in Theorem~\ref{thm:LassoRate}, then the same result holds with $\Xi_{n,s}^{(M)}$ in \eqref{eq:Re-satisfied} replaced by $\Xi_{n,s}^{(J)}$, with $\Xi_{n,s}^{(J)}$ as %defined
in Theorem~\ref{cor:REBoundUnified}(b). (This is true for Theorem \ref{thm:LassoPoly} as well and won't be repeated there.) With $\Xi_{n,s}^{(J)} \le \Xi_{n,s}^{(M)}$, %the requirement
this version of \eqref{eq:Re-satisfied} %with $\Xi_{n,s}^{(J)}$ instead of $\Xi_{n,s}^{(M)}$ %$\lambda_{\min}(\Sigma_n) \ge 1782\Xi_{n,s}^{(J)}$
imposes weaker sample complexity related conditions on the growth of $(n,k)$, as seen from Section \ref{rem:RECond} as well. Some related results for the Lasso with jointly sub-Weibull dependent random vectors can be found in \cite{Chung17}.
\end{rem}
%\begin{rem}\,(Importance of Gram Matrices)
%It is clear from this application of Lasso, the importance of studying the gram matrix itself rather than the covariance matrix. If in using Lasso all the variables are centered by the empirical mean, then the study of covariance matrices is also of interest.
%\end{rem}

%\tcm{EDITED TILL HERE - AC (7/23; 10 pm). (Made few more changes to contents above.)} \ref{}

\vspace{-0.05in}
\paragraph*{Lasso under Polynomial Moments on Errors.}
%Theorem 4.5 above follows as a ready application of Theorem 3.3.
A careful inspection of the theoretical analysis of Lasso reveals that the assumption of sub-Weibull errors in Theorem \ref{thm:LassoRate} \emph{can be weakened} to polynomial-tailed errors. This has also been noted in the recent work of \cite{HanWellnerAOS19}; see Theorem 5 and Examples 4-5 therein, where they provide a general recipe for deriving the convergence rates of Lasso allowing for much weaker tailed errors. Their results, however, are asymptotic in nature and need the restrictive assumption of $\varepsilon_i$'s being mean 0 and independent of $X_i$, $1\le i\le n$, although they do allow for dependence among $\varepsilon_i$'s.
In Theorem \ref{thm:LassoPoly} below, we prove an analogue of Theorem \ref{thm:LassoRate} assuming only polynomial moments (upto some order $r \geq 2$) of $\varepsilon_i$.
Recall Definition \ref{def:MarginalWeibull} and $\Xi_{n,s}^{(M)}$ from Theorem  \ref{cor:REBoundUnified}, and recall that for any random variable $W$, $\norm{W}_r = \left(\mathbb{E}\left[|W|^r\right]\right)^{1/r}$ for $r > 0$. %The notation for the following theorem is the one used in Theorem 4.5.}
\begin{thm}[Lasso with Marginally Sub-Weibull $X_i$'s and Polynomial-Tailed $\epsilon_i$'s]\label{thm:LassoPoly}
Under the setting of Theorem~\ref{thm:LassoRate}, suppose $\norm{\beta_0}_0 \le k$ and there exists $0 < \alpha \le 2, r \ge 2$ so that
\[
\max_{1\le i\le n}\,\norm{X_i}_{M,\psi_{\alpha}} \le K_{n,p},\quad\mbox{and}\quad \max_{1\le i\le n}\,\norm{\varepsilon_i}_r \le K_{\varepsilon, r}.
\]
Also suppose $n\ge 2, k\ge 1$ and that %the matrix
$\Sigma_n$ satisfies~\eqref{eq:Re-satisfied}. %$\lambda_{\min}(\Sigma_n) \ge 1782\Xi_{n,k}^{(M)}$, with $\Xi_{n,k}^{(M)}$ as defined in Theorem~\ref{cor:REBoundUnified}(a).
Then for $L \ge 1$, with probability at least $1 - 3(np)^{-1} - 3n^{-1} - L^{-1}$, the regularization parameter $\lambda_n$ can be chosen to be
\begin{equation}\label{eq:LambdaChoicePoly}
\lambda_n = 14\sqrt{2}\sigma_{n,p}\sqrt{\frac{\log(np)}{n}} + \frac{C_{\alpha}K_{n,p}K_{\varepsilon, r}(\log(np))^{1/\alpha}\left[(\log(2n))^{1/\alpha} + L\right]}{n^{1 - 1/r}},
\end{equation}
so that the Lasso estimator $\hat{\beta}(\lambda_n)$ satisfies
\begin{equation}%\label{eq:L2ErrorBound}
\begin{split}
\norm{\hat{\beta}_n(\lambda_n) - \beta_0}_2 & \le \; \frac{84\sqrt{2}}{\lambda_{\min}(\Sigma_n)} \sigma_{n,p}\sqrt{\frac{k\log(np)}{n}}\\  & \qquad + C_{\alpha}K_{n,p}K_{\varepsilon, r}\frac{k^{1/2}(\log(np))^{1/\alpha}\left[(\log(2n))^{1/\alpha} + L\right]}{\lambda_{\min}(\Sigma_n) \; n^{1 - 1/r}},
\end{split}
\end{equation}
for some constant $C_{\alpha} > 0$ depending only on $\alpha$.
\end{thm}

\begin{rem}\label{rem:RateLassoPoly}\;\;(Convergence Rates and the Special Case of Fixed Designs).\;\;
Theorem~\ref{thm:LassoPoly} readily proves that the rate of convergence of the Lasso is $\sigma_{n,p}\sqrt{k\log p/n}$ if
\begin{equation}
K_{\varepsilon, r}(\log(np))^{1/\alpha - 1/2}(\log(2n))^{1/\alpha} = o(n^{1/2 - 1/r}).\label{eq:RateConstraint}
\end{equation}
In comparison to \cite{HanWellnerAOS19}, Theorem~\ref{thm:LassoPoly} provides a precise non-asymptotic extension of their (asymptotic) results under (marginally) sub-Weibull covariates, without the assumption regarding the errors being independent of the covariates. Since our result allows for (a) non-identically distributed observations, (b) both fixed and random designs, as well as (c) possibly misspecified models, it serves as a  generalization (under sub-Weibull covariates) of Theorem 5 (and Example 5) of~\cite{HanWellnerAOS19}. Moreover, a careful inspection of their sample complexity requirement, as given in Equation (4.4) of their result, implies the condition $(\log p)^{4/\alpha +1} = O(n^{2-4/r})$ when translated into our setup and notation. This is a far stronger condition (e.g., if $\log p$ is polynomial in $n$) than our requirement \eqref{eq:RateConstraint}.

Finally, note that if we are under a \emph{fixed design}, i.e. if $X_i, 1\le i\le n$ are $n$ fixed vectors, then $X_i$'s simply are marginally sub-Weibull $(\infty)$ and %$\alpha$ in Theorem \ref{thm:LassoPoly} can be taken to be $\infty$ since
\[
\max_{1\le i\le n}\norm{X_i}_{M,\psi_{2}} \; \leq \; \max_{1\le i\le n}\norm{X_i}_{M,\psi_{\infty}} \; = \; \max_{1\le i\le n} \max_{1\le j\le p}|X_i(j)|.
\]
Hence, applying Theorem~\ref{thm:LassoPoly} with $\alpha = 2$ in this case, we observe that a rate of $\sqrt{k \log p / n}$ can be achieved %In this case,
under the (almost trivial) rate constraint %the rate constraint~\eqref{eq:RateConstraint} is equivalent to
\[
K_{\varepsilon, r}(\log(np))^{-1/2}(\log(2n))^{1/2} = o(n^{1/2 - 1/r}),
\]
which is satisfied as long as $n$ is large enough and $r > 2$. Similarly, for Theorem~\ref{thm:LassoRate}, the constraint becomes: $(\log(np))^{4/\vartheta - 1}=o(n)$. It should be noted that for fixed designs, the RE condition is simply an explicit assumption.
\end{rem}

\begin{rem}\label{rem:LassoExtensions}\;\;(Extensions and Other Estimators).\;\;
Using the probability tools from Section~\ref{sec:Indep} and the method of proof in this section, it is possible to prove very general results extending Theorem~\ref{thm:LassoRate} in several directions (similar extensions also apply to Theorem~\ref{thm:LassoPoly}, although %even though
we only illustrate them for Theorem~\ref{thm:LassoRate}). We briefly discuss some of these below.

Theorem~\ref{thm:LassoRate} is proved under the assumption of `hard' sparsity in the sense that no more than $k$ entries of $\beta_0$ are non-zero. One can actually derive a more \emph{general oracle inequality} (without any hard sparsity condition) for $\hat{\beta}_n(\lambda_n)$ using Theorem 1 of \cite{Neg12}.

Under the assumptions of Theorem~\ref{thm:LassoRate} (except the hard sparsity), an \emph{oracle inequality for the Lasso} is as follows. For a choice of $\lambda_n$ as in~\eqref{eq:LambdaChoice}, with probability converging to 1,
\begin{equation}%\label{eq:OracleIneq}
\begin{split}
&\norm{\hat{\beta}_n(\lambda_n) - \beta_0}_2^2 \label{eq:OracleIneq} \\
&\quad\le \min_{S:\,\Xi_{n,|S|}^{(M)} = o(1)}\left[\frac{18\lambda_n^2{|S|}}{\Gamma^2_n(S)}~+~\frac{8\lambda_n\norm{\beta_0(S^c)}_1}{\Gamma_n(S)} + \frac{3456\Xi_{n,|S|}^{(M)}\norm{\beta_0(S^c)}_1^2}{|S|\Gamma_n(S)}\right],
\end{split}
\end{equation}
where
%\[
$\Gamma_n(S) := \lambda_{\min}(\Sigma_n) - 1755\Xi_{n,|S|}^{(M)}.$
%\]
 Under the condition $\Xi_{n,|S|}^{(M)} = o(1)$, for large enough $n$, $\Gamma_n(S) \ge \lambda_{\min}(\Sigma_n)/2$. (The constants could possibly be improved here.) This is an oracle inequality because there is \emph{no} assumption on $\beta_0$ and the bound \emph{adapts} to the true sparsity of $\beta_0$. The proof is %can be found
 in Appendix~\ref{AppSec:HDLinReg} (Proposition \ref{prop:Lassooracineq}).
 As shown in Section 4.3 of \cite{Neg10}, inequality~\eqref{eq:OracleIneq} implies a rate of convergence if $\beta_0$ is \emph{weakly sparse}.

Following the proof of Proposition 2 of \cite{Neg10}, and using the proof of Theorem~\ref{cor:REBoundUnified}, it is easy to prove the restricted strong convexity property for generalized linear models when the covariates are marginally sub-Weibull. Hence, Theorem~\ref{thm:LassoRate} can be easily extended to %the case of
$L_1$-penalized estimation methods for \emph{generalized linear models} as well.

%% Needing to replace some (\cite{...}) commands here to \citep - done in Arxiv draft. May need to do this in IMAIAI final version as well to ensure consistency of citation style -- AC (5/9/2022). %% 
Finally, we mention that apart from the Lasso, there are many \emph{other estimators} available for high dimensional linear regression, including, for instance, the Dantzig selector \citep{Candes07} and the square-root Lasso \citep{Belloni11}, among others. The key ingredients in the analysis of all these estimators are the restricted eigenvalue condition and the gradient's control, as shown in \cite{Geer16}. Hence, the rate of convergence of these estimators can also be derived under weaker tail assumptions based on our results.
\end{rem}

\section{Conclusions and Future Work}\label{sec:Conclusions}
In this paper, we proposed a new Orlicz norm that extracts a part sub-Gaussian tail behavior for sums of independent random variables. Various concentration inequalities related to sub-Weibull random variables and processes are then studied in a unified way. We hope that the exposition here amplifies the use of sub-Weibull random variables, especially the heavy-tailed ones, in the theoretical analysis of statistical methods. To illustrate this, we studied four fundamental statistical problems in high-dimensions and extended many of the by-now standard results in the literature. As mentioned earlier (e.g., in Section \ref{rem:RECond}), our main goal here was to demonstrate the applications of our user-friendly concentration inequalities in handling these statistical problems under (much) weaker than usual tail assumptions, and obtaining fairly general results that still compare favorably to existing ones under stronger conditions. For some of the problems (e.g., RE condition), a more nuanced problem specific analysis can possibly lead to slightly better results or conditions than ours. But we refrain from such refined analyses given our main focus in this paper. Nevertheless, we do believe our results in Section~\ref{sec:RECondition} provide a much needed unified analysis on the RE condition that is not easily accessible in the statistics literature. Moreover, our results  on the Lasso in Section ~\ref{sec:HDLinReg} are possibly the first results in the literature that are obtained in such generality. %For example, our results in Sections~\ref{sec:RECondition} and~\ref{sec:HDLinReg} are possibly the first set of unified results on the RE condition and the Lasso under marginally sub-Weibull covariates and sub-Weibull or polynomial-tailed errors.

%% Needing to replace some (\cite{...}) commands here to \citep - done in Arxiv draft. May need to do this in IMAIAI final version as well to ensure consistency of citation style -- AC (5/9/2022). %%

Throughout the paper, we have restricted the random variables/vectors to be independent to keep the presentation simple. The independence assumption, however, may not be appropriate for many econometric applications. The extensions of the results in Section~\ref{sec:Indep} are available in \cite{Merv11} for strong mixing random variables, and in Appendix B of \cite{Kuch18} for functionally dependent random variables \citep{Wu05}. Unfortunately, many useful processes are not strongly mixing and the results of \cite{Kuch18} do not reduce to those in Section~\ref{sec:Indep} under independence. Extensions to the case of martingales are also not fully understood. A recent progress in this direction is \cite{Fan17Gaussian} that provides the result for martingales with $\alpha = 2$; see also \cite{Fan17} for related results. Tail bounds for martingales matching their asymptotic normality under sub-Weibull martingale differences have important implications for concentration results related to functions of independent random variables, which in turn are useful for dependent data \citep{Wu05}; see \cite{Bouch05} for more applications in this regard. %These results also have useful implications for dependent data. \cite{Wu05} introduced a notion of functional and physical dependence based on a coupling idea and the random variables in the process are defined as functions of a semi-infinite sequence of i.i.d. random variables.
Thus, it is %therefore
worth considering possible extensions of our results in Section \ref{sec:Indep} to martingales.

%In relation to the probabilistic part of the paper, we believe that there are some open questions as discussed in Remark~\ref{rem:Martingales} and some interesting problems mentioned in Remark~\ref{rem:EmpProcess}.
In terms of further statistical applications of our results, an important problem worth considering is a complete study of the problem in Remark~\ref{rem:LinearKernel}, including consistency of the LKAEs in terms of the supremum norm and/or uniform-in-bandwidth consistency. These problems have been considered under an asymptotic setting by \citet{Ein00, Einmahl05} using empirical process techniques. Their basic framework can indeed be adopted and combined with our results on suprema of empirical processes in Appendix %Section
\ref{sec:EmpProcess} to obtain a sequence of widely applicable non-asymptotic results for LKAEs involving sub-Weibulls.

Further, it is also of interest to study the version of these problems involving the so-called ``generated regressors'', wherein the kernel smoothing is only performed over (possibly) lower dimensional and/or estimated (if unknown) transformations of the original covariates. Such methods are of considerable importance in econometrics and in the sufficient dimension reduction literature. The latter
can be particularly useful in high dimensional settings, where a fully non-parametric smoothing may be undesirable due to the curse of dimensionality;
%especially when the transformations on the covariates are either known or consistently estimated;
see \cite{Mammen12, Mammen13} for some results and literature review on non-parametric regression over generated regressors. Using our empirical process results from Appendix  %Section
\ref{sec:EmpProcess} again, it would be of interest to obtain non-asymptotic tail bounds and rates of convergence for such LKAEs over generated regressors, especially in ``truly'' high dimensional settings where the dimension of the original covariates could be much larger than the sample size. While all these problems are interesting, a detailed analysis is far too involved for the scope of the current paper. We do hope to explore some of these problems separately in the future.
%One main application of our tools in non-parametric function estimation involves kernel estimation with estimated covariates; see \cite{Mammen12, Mammen13} for some results. The description of the problem is as follows. Suppose $(X_1, Y_1), \ldots, (X_n, Y_n)$ are independent random vectors in $\mathbb{R}^{p+1}$ and let $\hat{\beta}$ be an estimator consistent to some vector $\beta_0$. If $p$ is very large, then non-parametric estimation of conditional expectation of $Y$ given $X$ is not recommended due to the curse of dimensionality. With some structural assumptions, estimation of the conditional expectation of $Y$ given $X^{\top}\beta_0$ is useful for prediction. The study of the estimator constructed based on $(X_1^{\top}\hat{\beta}, Y_1), \ldots, (X_n^{\top}\hat{\beta}_n, Y_n)$ requires a careful understanding of dimension dependence in covering number results. Given the length of the present article, we could not present these results here and they will appear elsewhere.
%%%%%%%%%%%%%%%%%%%%%%%%%%%%%%%%%%%%%%%

%\tcm{FULLY EDITED NOW (Lasso proofs slightly edited as well). See Discussion too (made changes). Take a look (make any changes, if needed, and let me know) - AC (8/3, 3 am).}

\appendix
%\appendix
%\section{Properties of the Generalized Bernstein-Orlicz Norm}\label{AppSec:PropGBO}
\section{Properties of the GBO Norm}\label{AppSec:PropGBO}
In this section, we provide a collection of some useful basic properties of the GBO norm. Since it does not have a closed form, it is hard to directly see the part sub-Gaussian behavior captured by the GBO norm for sub-Weibulls, as shown in ~\eqref{eq:SeparateBernstein} for sub-exponentials. To resolve this issue, we first provide in Proposition~\ref{prop:EquivalentNorm} an equivalent norm that is based on a closed form $g$. (The proofs of all Propositions in this Appendix are given in Appendix \ref{AppSec:Definition}.) %The proof is given in Appendix~\ref{AppSec:Definition}.
\begin{prop}\label{prop:EquivalentNorm}
Fix $\alpha, L > 0$. Define $\phi_{\alpha, L}:[0, \infty) \to [0, \infty)$ as
\[
\phi_{\alpha, L}(x) = \exp\left(\min\left\{x^2, \left(\frac{x}{L}\right)^{\alpha}\right\}\right) - 1.
\]
Then for any random variable $X$,
%\begin{equation}\label{eq:EquivalenceNorms}
$\norm{X}_{\Psi_{\alpha, L}} \le \norm{X}_{\phi_{\alpha, L}} \le 2\norm{X}_{\Psi_{\alpha, L}}.$
%\end{equation}
\end{prop}
In the remaining part of this section, we derive various properties of $\norm{\cdot}_{\Psi_{\alpha, L}}$, the proofs of which are all in %deferred to
Appendix~\ref{AppSec:Definition}.
We start with simple monotonicity properties of $\norm{\cdot}_{\Psi_{\alpha, L}}$.
\begin{prop}[Monotonicity Properties]\label{prop:Monotone}
The following monotonicity properties hold for the GBO norm:
\begin{enumerate}
\item[(a)] If $|X| \le |Y|$ almost surely, then $\norm{X}_{\Psi_{\alpha, L}} \le \norm{Y}_{\Psi_{\alpha, L}}$ for all $\alpha, L > 0$.
\item[(b)] For any random variable $X$, $\norm{X}_{\Psi_{\alpha, L}} \le \norm{X}_{\Psi_{\alpha, K}}$ for $0 \leq L \le K.$
\end{enumerate}
\end{prop}

The following sequence of propositions prove the equivalence of finite $\Psi_{\alpha, L}$-norm with a tail bound and a moment growth. The proofs are similar to those of \cite{Geer13}. It is worth mentioning
here that although we present some of the results with explicit constants, our goal is not to provide optimal constants and they could possibly be improved.
\begin{prop}[Equivalence of Tail and Norm Bounds]\label{prop:EquivalenceTailMoment}
For any random variable $X$ with $\delta := \norm{X}_{\Psi_{\alpha, L}}$, we have
\begin{equation}\label{eq:TailForm}
\mathbb{P}\left(|X| \ge \delta\left\{\sqrt{t} + Lt^{1/\alpha}\right\}\right) \le 2\exp(-t),\quad\mbox{for all}\quad t\ge 0.
\end{equation}
Conversely, if the tail bound~\eqref{eq:TailForm} holds for some constants $\delta, L > 0$, then $$\norm{X}_{\Psi_{\alpha, c(\alpha)L}} \le \sqrt{3}\delta,\quad\mbox{where}\quad c(\alpha) := 3^{1/\alpha}/\sqrt{3}.$$
\end{prop}

\begin{prop}[Equivalence of Moment Growth and Norm Bound]\label{prop:EquivalenceNormMoment}
For any random variable $X$,
\begin{equation}\label{eq:NormToMoment}
{C}_*(\alpha)\sup_{p\ge 1}\frac{\norm{X}_p}{\sqrt{p} + Lp^{1/\alpha}} \; \le \; \norm{X}_{\Psi_{\alpha, L}} \le \; {C}^*(\alpha)\sup_{p\ge 1}\frac{\norm{X}_p}{\sqrt{p} + Lp^{1/\alpha}},
\end{equation}
where
%\[
${C}_*(\alpha) := \frac{1}{2}\min\{1, \alpha^{1/\alpha}\}\quad\mbox{and}\quad {C}^*(\alpha) := e\max\left\{2, 4^{1/\alpha}\right\}.$
%\]
\end{prop}
%\begin{rem}
%The equivalence statement of Proposition \ref{prop:EquivalenceNormMoment} can be alternatively written as
%
%The quantity $\sup_{p\ge 1}\norm{X}_p/(\sqrt{p} + Lp^{1/\alpha})$ is a norm and so $\norm{X}_{\Psi_{\alpha, L}}$ is a quasi-norm. This fact is formalized in the following result.
%\end{rem}
%Equivalence of norm and moment growth allows a simple derivation for the quasi-norm property of $\norm{\cdot}_{\Psi_{\alpha, L}}$.
\begin{prop}[Quasi-Norm Property]\label{prop:QuasiNorm}
For any sequence of (possibly dependent) random variables $X_i, 1\le i\le k$,
\[
\norm{\sum_{i=1}^k X_i}_{\Psi_{\alpha, L}} \le \; Q_{\alpha}\sum_{i=1}^k\norm{X_i}_{\Psi_{\alpha, L}},
\]
where
\begin{equation*}
Q_{\alpha} \; := \;
\begin{cases}
%e\max\{2,4^{1/\alpha}\}\max\{1, (1/\alpha)^{1/\alpha}\},&\mbox{if }\alpha < 1,\\
2e(4/\alpha)^{1/\alpha},&\mbox{if }\alpha < 1,\\
1,&\mbox{if }\alpha \ge 1.
\end{cases}
\end{equation*}
\end{prop}

One of the main advantages of Orlicz norms of the exponential type lies in their usefulness to derive maximal inequalities. The following result proves one such for the GBO norm~$\norm{\cdot}_{\Psi_{\alpha, L}}.$
\begin{prop}[Maximal Inequality]\label{prop:MaximalPsi}
Let $X_1, \ldots, X_N$ be random variables (possibly dependent) such that
%\[
$\max_{1\le j\le N}\norm{X_j}_{\Psi_{\alpha, L}} \le \Delta < \infty$
%\]
for some $\alpha, L, \Delta > 0$.
Set
%\[
$X_N^* := \max_{1\le j\le N}\left|X_j\right|$, and recall $c(\alpha)$ and $Q_{\alpha}$ from Propositions \ref{prop:EquivalenceTailMoment} and \ref{prop:QuasiNorm}.
%\]
Then for all $t \ge 0$,
\[
\mathbb{P}\left(X_N^* \ge \Delta\left\{\sqrt{t + \log N } + L\left(t + \log N\right)^{1/\alpha}\right\}\right) \le 2\exp(-t),
\]
%\begin{equation}\label{eq:PositivePartBound}
%\norm{\left(\max_{1\le j\le N}|X_j| - \Delta\left\{\sqrt{\log(1 + N)} + M(\alpha)L\left(\log(1 + N)\right)^{1/\alpha}\right\}\right)_+}_{\Psi_{\alpha, K(\alpha)}} \le \sqrt{3}\Delta,
%\end{equation}
and
\begin{equation}\label{eq:MaxOrliczBound}
\norm{X_N^*}_{\Psi_{\alpha, K(\alpha)L}} \le \Delta Q_{\alpha}\left\{\sqrt{3} + \sqrt{\log N} + M(\alpha)L\left(\log N \right)^{\frac{1}{\alpha}}\right\},
\end{equation}
where $K(\alpha) := c(\alpha)M(\alpha)$ with $M(\alpha) := \max\{1, 2^{(1-\alpha)/\alpha}\}$. %Recall $c(\alpha)$ and $Q_{\alpha}$ from Propositions \ref{prop:EquivalenceTailMoment} and \ref{prop:QuasiNorm}.
\end{prop}

\begin{rem}\;\;(Bound on the Expectation of the Maximum).\;\;
From Proposition~\ref{prop:MaximalPsi} %inequality~\eqref{eq:MaxOrliczBound}
it follows that
\begin{align}
\norm{X_N^*}_1
%&\le \delta\left\{\sqrt{\log (1 + N)} + M(\alpha)L(\log (1 + N))^{1/\alpha}\right\}\\ &\qquad+ 2\sqrt{3}\delta\left[\frac{1}{\sqrt{2}} + c(\alpha)M(\alpha)L(1/\alpha)^{1/\alpha}\right]\\
&\le \max_{1\le j\le N}\norm{X_j}_{\Psi_{\alpha, L}}C_{\alpha}\left\{\sqrt{\log N } + L\left(\log N \right)^{1/\alpha}\right\},\label{eq:MaximalInequalityFirstMoment}
\end{align}
for some constant $C_{\alpha}$ depending only on $\alpha.$ Note that if the random variables are sub-Gaussian ($\alpha = 2$), then the rate becomes $\sqrt{\log N}$. The main implication of the GBO norm is that it shows the rate can \emph{still} be $\sqrt{\log N}$ even if $\alpha \neq 2$ as long as~$L(\log N)^{1/\alpha - 1/2} = o(1)$.
\end{rem}
\par\smallskip
The next proposition provides an alternative to, and a generalization of, Proposition \ref{prop:MaximalPsi}. This is similar to Proposition 4.3.1 of \cite{DeLaPena99}. Note that for infinitely many random variables $(N = \infty)$, Proposition~\ref{prop:MaximalPsi} does not lead to useful bounds; see the discussion following Proposition 4.3.1 of \cite{DeLaPena99} for the importance of considering an alternative result as presented below. %for importance of this alternative.
\begin{prop}[A Sharper Maximal Inequality]\label{prop:RatioMaximal}
Let $X_1, X_2, \ldots$ be any sequence of random variables (possibly dependent) such that for all $i = 1,2,\hdots$, $\norm{X_i}_{\Psi_{\alpha, L}} < \infty$ for some $\alpha, L > 0$, and recall $c(\alpha)$, $Q_{\alpha}$ and $M(\alpha)$ from Propositions \ref{prop:EquivalenceTailMoment}, \ref{prop:QuasiNorm} and \ref{prop:MaximalPsi}. %For any sequence of random variables $X_1, X_2, \ldots$ satisfying
%\[
%\norm{X_i}_{\Psi_{\alpha, L}} < \infty \quad i = 1,2,\hdots, \quad \mbox{for some}\quad \alpha, L > 0.
%\]
Then
\[
\norm{\sup_{k \ge 1}\frac{|X_k|}{\sqrt{2}\norm{X_k}_{\Psi_{\alpha, L}}\Psi_{\alpha, S(\alpha)L}^{-1}(k)}}_{\Psi_{\alpha, c(\alpha)M(\alpha) L}} \le \; 2.5Q_{\alpha},
\]
where $S(\alpha) := 2^{1/\alpha}M(\alpha)/2$. %Recall $c(\alpha)$, $Q_{\alpha}$ and $M(\alpha)$ from Propositions \ref{prop:EquivalenceTailMoment}, \ref{prop:QuasiNorm} and \ref{prop:MaximalPsi}.
\end{prop}
\subsection{Extensions to Tail Behaviors Involving Multiple Regimes}
The GBO norm $\norm{\cdot}_{\Psi_{\alpha, L}}$ introduced in Section \ref{sec:Definition} is designed to exploit two regimes in the tail of a random variable, namely, Gaussian and Weibull of order $\alpha$. It is of interest to extend the theory to exploit more than two regimes in the tail of a random variable. Many examples exist where this is relevant, including in particular $U$-statistics based on independent variables; see, for example, \cite{Lat99}, \cite{Gine00} and \cite{Bouch05} for results on $U$-statistics and Rademacher Chaos.

For vectors $\boldsymbol{\alpha} = (\alpha_1, \ldots, \alpha_k)\in(\mathbb{R}^+)^k$ and $\boldsymbol{L} = (L_1, \ldots, L_k)\in(\mathbb{R}^+)^k$, for some $k$, define the function $\Psi_{\boldsymbol{\alpha}, \boldsymbol{L}}(\cdot)$ based on the inverse function
\[
\Psi_{\boldsymbol{\alpha}, \boldsymbol{L}}^{-1}\left(t\right) := \sum_{j = 1}^k L_j\left(\log(1 + t)\right)^{1/\alpha_j}\quad\mbox{for}\quad t\ge 0.
\]
The extended multiple regime GBO norm is defined by setting $g(\cdot) = \Psi_{\boldsymbol{\alpha}, \boldsymbol{L}}(\cdot)$ in Definition \ref{def:IncOrliczNorm}. The GBO norm $\norm{\cdot}_{\Psi_{\alpha, L}}$ corresponds to $\boldsymbol{\alpha} = (1/2, \alpha)$ and $\boldsymbol{L} = (1, L)$. Similar to $\Psi_{{\alpha}, {L}}(\cdot)$, there is no closed form expression for $\Psi_{\boldsymbol{\alpha}, \boldsymbol{L}}(\cdot)$, and a function $\phi_{\boldsymbol{\alpha}, \boldsymbol{L}}(\cdot)$ closely related to $\Psi_{\boldsymbol{\alpha}, \boldsymbol{L}}(\cdot)$  is given by:
\[
\phi_{\boldsymbol{\alpha}, \boldsymbol{L}}^{-1}(t) \; := \; \max\left\{L_j\left(\log(1 + t)\right)^{1/\alpha_j}:\,1\le j\le k\right\}.
\]
It is easy to check that
%\[
$\norm{X}_{\Psi_{\boldsymbol{\alpha}, \boldsymbol{L}}} \le \norm{X}_{\phi_{\boldsymbol{\alpha}, \boldsymbol{L}}} \le k\norm{X}_{\Psi_{\boldsymbol{\alpha}, \boldsymbol{L}}}.$
%\]
All the properties stated in this section also hold for the extended GBO norm $\norm{\cdot}_{\Psi_{\boldsymbol{\alpha}, \boldsymbol{L}}}$. Their proofs are similar and hence omitted to avoid repetition.

%{\cmag ** ADD A NEW REMARK/SUBSECTION HERE REGARDING THE CONVEX ETC. STUFF (R1 -- Other remark 1) -- AC**}

\phantomsection
\addcontentsline{toc}{section}{Supplementary Material}

%\section*{Supplementary Material}\label{supp_mat}         % Should not use a \center inside! Doesn't work with \section at all, and even with \section*, the supp file doesn't run! %
%\textbf{Supplement to ``Moving Beyond Sub-Gaussianity in High Dimensional Statistics: Applications in Covariance Estimation and Linear Regression''} (.pdf file).
%The supplementary material (Appendices~\ref{sec:EmpProcess}--\ref{AppSec:EmpProcess}) contains additional results and technical materials that could not be accommodated in the main article. In Appendix~\ref{sec:EmpProcess}, we extend the study of sub-Weibulls to tail bounds for the suprema of empirical processes. In Appendices~\ref{AppSec:Definition}--\ref{AppSec:EmpProcess}, we present the proofs of all our results in the main article and the supplement.

\section*{Supplementary Material}\label{supp_mat}         % Should not use a \center inside! Doesn't work with \section at all, and even with \section*, the supp file doesn't run! %
\textbf{Supplement to ``Moving Beyond Sub-Gaussianity in High Dimensional Statistics: Applications in Covariance Estimation and Linear Regression''}. %(pdf).
The supplementary material (Appendices~\ref{sec:EmpProcess}--\ref{AppSec:EmpProcess}) contains additional results and technical materials that could not be accommodated in the main article. In Appendix~\ref{sec:EmpProcess}, we extend the study of sub-Weibulls to tail bounds for the suprema of empirical processes. In Appendices~\ref{AppSec:Definition}--\ref{AppSec:EmpProcess}, we present the proofs of all our results in the main article and the supplement.

\phantomsection
\addcontentsline{toc}{section}{Acknowledgements}

\section*{Acknowledgements}\label{acknowledge}

We would like to thank the Editor, the anonymous Associate Editor and the two Reviewers for their constructive comments and useful suggestions that helped significantly improve the article. We would also like to thank Dr. Edward George for helpful initial discussions that improved the article's presentation.

\phantomsection
\addcontentsline{toc}{section}{References}

\bibliographystyle{apalike}
\bibliography{References}  %% Needed to create this alternative .bib file to clean up some MathSci based citations that were becoming a nuisance under JMLR style citations.%%

%\begin{abstract}
%This supplementary document (Appendices~\ref{sec:EmpProcess}--\ref{AppSec:EmpProcess}) contains the proofs of all the results in the main paper, as well as additional results (and their proofs) on tail bounds for suprema of empirical processes with sub-Weibull functions. In Appendix~\ref{sec:EmpProcess}, we present these supplementary results on suprema of empirical processes. Proofs of all the results in Section~\ref{sec:Definition} (along with those in Appendix~\ref{AppSec:PropGBO}) and Section~\ref{sec:Indep} of the main paper are presented in Appendices~\ref{AppSec:Definition} and \ref{AppSec:Indep}, respectively.  The results of Section~\ref{sec:Applications}, as well as those in Appendix~\ref{sec:EmpProcess} here, are proved in Appendices~\ref{AppSec:Applications} and \ref{AppSec:EmpProcess}, respectively.
%
%\end{abstract}

\appendix
\setcounter{section}{1}

\renewcommand\theequation{\thesection.\arabic{equation}}  %% NOTE: Needed this to take care of a major eqn. numbering issue in the Supp -- this was WRONG in the first draft! - 2/12/2022.

\section{Norms of Supremum of Empirical Processes}\label{sec:EmpProcess}
%\phantomsection
%\addcontentsline{toc}{section}{Norms of Supremum of Empirical Processes}

In this section, we present tail and norm bounds for the supremum of empirical processes with certain tail bounds on the envelope function. To avoid any issues about measurability, we follow the convention of \cite{Tala14} and define
\begin{equation}\label{eq:Convention}
\mathbb{E}\left[\sup_{t\in T}X_t\right] := \sup\left\{\mathbb{E}\left[\sup_{t\in S}X_t\right]:\,S\subseteq T\mbox{ is finite}\right\},
\end{equation}
for any stochastic process $\{X_t\}$ indexed by $t\in T$ for some set $T$; see Equation (2.2) of \cite{Tala14}. Using this convention, we can define the $g$-Orlicz norm of the supremum as
\begin{equation}\label{eq:OrliczNormConvention}
\norm{\sup_{t\in T} X_t}_{g} := \inf\left\{C > 0:\,\mathbb{E}\left[g\left(\left|\sup_{t\in S} \frac{X_t}{C}\right|\right)\right] \le 1\quad\mbox{for all}\quad S\subseteq T\mbox{ finite}\right\}.
\end{equation}

The setting for all the results in this section is as follows. Let $X_1, X_2, \ldots, X_n$ be independent random variables with values in a measurable space $(\mathcal{X}, \mathcal{B})$ and $\mathcal{F}$ is a class of measurable functions $f:\mathcal{X}\to\mathbb{R}$ such that $\mathbb{E}f(X_i) = 0$ for all $f\in\mathcal{F}$. Define
\begin{equation}\label{eq:FirstNotation}
Z := \sup_{f\in\mathcal{F}}\left|\sum_{i=1}^n f(X_i)\right|\quad\mbox{and}\quad\Sigma_n(\mathcal{F}) := \sup_{f\in\mathcal{F}}\sum_{i=1}^n \mathbb{E}\left[f^2(X_i)\right].
\end{equation}
Without loss of generality, we can assume that $\mathcal{F}$ is finite, using \eqref{eq:OrliczNormConvention}. The final result will not depend on the cardinality of $\mathcal{F}$ implying the result by \eqref{eq:OrliczNormConvention}.
%\tcm{
Based on the Generalized Bernstein-Orlicz norm and the generic chaining proof techniques in Section 10.2 of \cite{Tala14} and Section 5 of \cite{Dirk15}, one can obtain ``optimal" tail bounds on the supremum of the empirical processes under a sub-Weibull envelope assumption in terms of the $\gamma$-functionals of \cite{Tala14}. These bounds, however, require computation of the complexity of $\mathcal{F}$ in terms of two distances and this can be hard in many examples of interest. For this reason, we first provide deviation bounds, and then bounds on the expectation (maximal inequalities), in terms of uniform covering and bracketing numbers.
%}
The proofs of all results in this section are given in Appendix~\ref{AppSec:EmpProcess}. %Section~\ref{AppSec:EmpProcess}.

Before proceeding to unbounded function classes, we first state a result that provides a moment bound for the supremum of a bounded empirical process. This is essentially the Talagrand's inequality for empirical processes. The result is based on Theorem 3.3.16 of \cite{GINE16} and is given with explicit constants to resemble the Bernstein's inequality for real-valued random variables; see also Theorem 1.1 and Lemma 3.4 of \cite{Klein05}.
\begin{prop}\label{prop:BddProc}
Suppose $\mathcal{F}$ is a class of uniformly bounded measurable functions~$f:\mathcal{X}\to [-U, U]$ for some $U < \infty$.
% and $X_1, X_2, \ldots, X_n$ are independent $\mathcal{X}$-valued random variables such that $\mathbb{E}\left[f(X_i)\right] = 0$ for all $f\in\mathcal{F}$. Define
%\[
%Z := \sup_{f\in\mathcal{F}}\left|\sum_{i=1}^n f(X_i)\right|,\quad \Sigma_n(\mathcal{F}) := \sup_{f\in\mathcal{F}}\sum_{i=1}^n \mathbb{E}\left[f^2(X_i)\right].
%\]
Then, under the setting above,
% \begin{equation}\label{eq:NormBound}
% \norm{\left(Z - \mathbb{E}[Z]\right)_+}_{\Psi_{1, L_n}} \le 2\left(\Sigma_n(\mathcal{F}) + 2U\mathbb{E}[Z]\right)^{1/2},
% \end{equation}
% where
% \[
% L_n := 1.5U\left(\Sigma_n(\mathcal{F}) + 2U\mathbb{E}[Z]\right)^{-1/2}.
% \]
% Furthermore,
for $p\ge 1$,
\begin{equation}\label{eq:MomentBound}
\norm{Z}_p \le \mathbb{E}\left[Z\right] + p^{1/2}\left(2\Sigma_n(\mathcal{F}) + 4U\mathbb{E}[Z]\right)^{1/2} + 6Up.
\end{equation}
\end{prop}

Proposition \ref{prop:BddProc} can now be extended to possibly unbounded empirical processes using the proof of Theorem 4 of \cite{Adam08} and this is in lines with our use of the technique in the proofs of Theorems \ref{prop:SimilarBernstein} and \ref{prop:BernsteinLargerThan1}. Set
\[
F(X_i) := \sup_{f\in\mathcal{F}}|f(X_i)|\quad\mbox{for}\quad 1\le i\le n\quad\mbox{and}\quad \rho := 8\mathbb{E}\left[\max_{1\le i\le n}|F(X_i)|\right].
\]
The function $F(\cdot)$ is called the envelope function of $\mathcal{F}$. Define the truncated part and the remaining unbounded part of $Z$ as
\begin{equation}\label{eq:SplitZ1Z2}
\begin{split}
Z_1 &:= \sup_{f\in\mathcal{F}}\left|\sum_{i=1}^n \Big(f(X_i)\mathbbm{1}\{|f(X_i)| \le \rho\} - \mathbb{E}\left[f(X_i)\mathbbm{1}\{|f(X_i)| \le \rho\}\right]\Big)\right|, \\
Z_2 &:= \sup_{f\in\mathcal{F}}\left|\sum_{i=1}^n \Big(f(X_i)\mathbbm{1}\{|f(X_i)| > \rho\} - \mathbb{E}\left[f(X_i)\mathbbm{1}\{|f(X_i)| > \rho\}\right]\Big)\right|.
\end{split}
\end{equation}
\begin{thm}\label{thm:UnBddProc}
Suppose, for some $\alpha, K > 0$,
\[
\max_{1\le i\le n}\norm{\sup_{f\in\mathcal{F}} |f(X_i)|}_{\psi_{\alpha}} \le K < \infty.
\]
Then, under the notation outlined above, for $\alpha_* = \min\{\alpha, 1\}$ and $p\ge 2$,
\begin{equation}\label{eq:MomentEmpBound}
\norm{Z}_p \le 2\mathbb{E}\left[Z_1\right] + \sqrt{2}p^{1/2}\Sigma_n^{1/2}(\mathcal{F}) + C_{\alpha}p^{1/\alpha_*}\norm{\max_{1\le i\le n} F(X_i)}_{\psi_{\alpha}},
\end{equation}
and
\begin{equation}\label{eq:NormEmpBound}
\norm{(Z - 2e\mathbb{E}[Z_1])_+}_{\Psi_{\alpha_*, L_n(\alpha)}} \le 3\sqrt{2}e\Sigma_n^{1/2}(\mathcal{F}),
\end{equation}
where
\begin{align*}
C_{\alpha} &:= 3\sqrt{2\pi}(1/\alpha_*)^{1/\alpha_*}K_{\alpha_*}\left[8 + (\log 2)^{1/\alpha - 1/\alpha_*}\right],\\
L_n(\alpha) &:= \frac{9^{1/\alpha_*}C_{\alpha}}{3\sqrt{2}}\norm{\max_{1\le i\le n}F(X_i)}_{\psi_{\alpha}}\Sigma_n^{-1/2}(\mathcal{F}).
\end{align*}
Here the constant $K_{\alpha_*}$ is the one used in Theorem 6.21 of \cite{LED91}.
\end{thm}

\begin{rem}
It is clear that this result reduces to Theorems \ref{prop:SimilarBernstein} and \ref{prop:BernsteinLargerThan1} if the function class $\mathcal{F}$ contains only one function. Note that in this case, $\mathbb{E}[Z_1]$ is bounded by $\Sigma_n^{1/2}(\mathcal{F})$. There are two differences of Theorem \ref{thm:UnBddProc} in comparison with Theorem 4 of \cite{Adam08}. Firstly, our result allows for the full range $\alpha \in (0, \infty)$ instead of just $\alpha \in (0, 1]$. Secondly, our result \emph{only} involves $\mathbb{E}[Z_1]$, that is, the expectation of the supremum of \emph{bounded} empirical processes instead of $\mathbb{E}[Z]$. This allows us to use many of the existing maximal inequalities for supremum of bounded empirical processes for the study of unbounded empirical processes as well. Also, it is interesting to note that using the bound on $\mathbb{E}[Z_1]$, and the moment bound \eqref{eq:MomentEmpBound}, we can bound $\mathbb{E}[Z]$. This is similar to the results in Section 5 of~\cite{Chern14}.
\end{rem}
\begin{rem}
The proof technique as mentioned above is truncation and using the Talagrand's inequality for the truncated part. We have taken this proof technique from \cite{Adam08}. Even if the envelope function does not satisfy a $\psi_{\alpha}$-norm bound, this part of the proof works. The moment bounds for the remaining unbounded part have to be obtained under whatever moment assumption the envelope function satisfies. This was done in \cite{Lederer14} under polynomial tails of the envelope function. The dominating term even in their bounds resemble the asymptotic Gaussian behavior as do ours.
\end{rem}
The application of Theorem \ref{thm:UnBddProc} only requires bounding $\mathbb{E}\left[Z_1\right]$, the expectation of the supremum of a bounded empirical process. Most of the maximal inequalities available in the literature apply to this case. The following two results provide such inequalities based on uniform entropy and bracketing entropy (defined below). There are many classes for which uniform covering and bracketing numbers are available and these can be found in \cite{VdvW96}. We only give these inequalities for bounded classes and explicitly show the dependence on the bound (which in our case may increase with the sample size). In the following, we use the classical empirical processes notation. For any function $f$, define the linear operator
\[
\mathbb{G}_n(f) := \frac{1}{\sqrt{n}}\sum_{i=1}^n \left\{f(X_i) - \mathbb{E}\left[f(X_i)\right]\right\}.
\]
Note here that we allow for non-identically distributed random variables $X_1, X_2, \ldots, X_n$.

Given a metric or a pseudo-metric space $(T, d)$ with metric $d$, for any $\epsilon > 0$, its covering number $N(\epsilon, T, d)$ is defined as the smallest number of balls of $d$-radius $\epsilon$ needed to cover $T$. More precisely, $N(\epsilon, T, d)$ is the smallest $m$ such that there exists $t_1, t_2, \ldots, t_m\in T$ satisfying
\[
\sup_{t\in T}\inf_{1\le j\le m}\,d(t, t_j) \le \epsilon.
\]
For any function class $\mathcal{F}$ with envelope function $F$, the uniform entropy integral is defined for $\delta > 0$ as
\[
J(\delta, \mathcal{F}, \norm{\cdot}_2) := \sup_{Q}\int_0^{\delta}\sqrt{\log (2N(x\norm{F}_{2,Q}, \mathcal{F}, \norm{\cdot}_{2,Q}))}dx,
\]
where the supremum is taken over all discrete probability measures $Q$ and $\norm{h}_{2,Q}$ denotes the $\norm{\cdot}_2$-norm of $h$ with respect to the probability measure $Q$, that is, $\norm{h}_{2,Q}^2 := \mathbb{E}_Q\left[h^2\right]$. To provide explicit constants we use Theorem 3.5.1 of \cite{GINE16} along with Theorem 2.1 of \cite{vdV11}.
\begin{prop}\label{prop:MaximalUniform}
Suppose $\mathcal{F}$ is a class of measurable functions with envelope function $F$ satisfying $\norm{F}_{\infty} \le U < \infty$. Assume that $\mathcal{F}$ contains the zero function. Then
\[
\mathbb{E}\left[\sup_{f\in\mathcal{F}}\left|\mathbb{G}_n(f)\right|\right] \le 16\sqrt{2}\norm{F}_{2,P}J\left(\delta_n(\mathcal{F}), \mathcal{F}, \norm{\cdot}_2\right)\left[1 + \frac{128\sqrt{2}UJ\left(\delta_n(\mathcal{F}), \mathcal{F}, \norm{\cdot}_2\right)}{\sqrt{n}\delta_n^2(\mathcal{F})\norm{F}_{2,P}}\right],
\]
where $\Sigma_n(\mathcal{F})$ is as defined in \eqref{eq:FirstNotation},
\[
\norm{F}_{2,P}^2 := \frac{1}{n}\sum_{i=1}^n \mathbb{E}\left[F^2(X_i)\right],\quad\mbox{and}\quad \delta_n^2(\mathcal{F}) := \frac{\Sigma_n(\mathcal{F})}{n\norm{F}_{2,P}^2}.
\]
\end{prop}

The following proposition proves an alternative to Proposition \ref{prop:MaximalUniform} using bracketing numbers. For $\epsilon > 0$, let the set $\{[{f}_j^{L}, {f}_{j}^{U}]:\,1\le j\le {N}_{\epsilon}\}$ represents the minimal $\epsilon$-bracketing set of $\mathcal{F}$ with respect to $\norm{\cdot}_{2,P}$-norm if for any $f\in\mathcal{F}$, there exists an $1\le I\le {N}_{\epsilon}$ such that for all $x$,
\[
{f}_{I}^{L}(x) \le f(x) \le {f}_I^{U}(x)\quad\mbox{and}\quad \frac{1}{n}\sum_{i=1}^n \mathbb{E}\left[|{f}_I^{U}(X_i) - {f}_I^{L}(X_i)|^2\right] \le \epsilon^2.
\]
The number ${N}_{\epsilon}$ is the $\epsilon$-bracketing number, usually denoted by $N_{[\,]}(\epsilon, \mathcal{F}, \norm{\cdot}_{2,P}).$ Define the bracketing entropy integral as
\[
J_{[\,]}\left(\eta, \mathcal{F}, \norm{\cdot}_{2,P}\right) := \int_0^{\eta} \sqrt{\log\left(2N_{[\,]}\left(x, \mathcal{F}, \norm{\cdot}_{2,P}\right)\right)}dx\quad\mbox{for}\quad \eta > 0.
\]
The following proposition is very similar to Proposition 3.4.2 of \cite{VdvW96} and we provide it here with explicit constants allowing for non-identically distributed random variables. The proof follows that of Theorem 3.5.13 and Proposition 3.5.15 of \cite{GINE16} and we do not repeat the proof except for necessary changes. Also, see Theorem 6 of \cite{Poll02}.
\begin{prop}\label{prop:BracketingEntropy}
Suppose $\mathcal{F}$ is a class of measurable functions with envelope function $F$ satisfying $\norm{F}_{\infty} \le U < \infty$. Then
\begin{align*}
\mathbb{E}\left[\sup_{f\in\mathcal{F}}\left|\mathbb{G}_n(f)\right|\right] \le 2J_{[\,]}\left(\delta_n(\mathcal{F}), \mathcal{F}, \norm{\cdot}_{2,P}\right)\left[58 + \frac{J_{[\,]}\left(\delta_n(\mathcal{F}), \mathcal{F}, \norm{\cdot}_{2,P}\right)U}{\sqrt{n}\delta^2_n(\mathcal{F})}\right],
\end{align*}
for any $\delta_n(\mathcal{F})$ satisfying $\delta_{n}(\mathcal{F}) \ge \Sigma_{n}^{1/2}(\mathcal{F})/\sqrt{n}$ with $\Sigma_n(\mathcal{F})$ as in~\eqref{eq:FirstNotation}.
\end{prop}
For the sake of completeness, we provide one last result relating the expectation of the unbounded supremum $Z$ in terms of the expectation of the supremum $Z_1$ of a bounded empirical process. Theorem~\ref{thm:UnBddProc} provides such a result under a sub-Weibull envelope assumption, while the following result applies in general.
\begin{prop}\label{prop:ExpectationUnbddBdd}
Under the notation outlined before Theorem~\ref{thm:UnBddProc}, we have
\[
\mathbb{E}\left[Z\right] \le \mathbb{E}\left[Z_1\right] + 8\mathbb{E}\left[\max_{1\le i\le n}\, F(X_i)\right].
\]
\end{prop}
\begin{rem}\label{rem:EmpProcess}
We note that only a sample of empirical process results are presented here. For many applications, the results on the statistic $Z$ in \eqref{eq:FirstNotation} are not sufficient. The main reason for this is that these results do not allow for function dependent scaling. For example, if the variance of $\sum f(X_i)$ varies too much as $f$ varies over $\mathcal{F}$, then it is desirable to obtain bounds for
\[
\sup_{f\in\mathcal{F}}\left(\sum_{i=1}^n \mathbb{E}\left[f^2(X_i)\right]\right)^{-1/2}{\left|\sum_{i=1}^n f(X_i)\right|}.
\]
This arises in uniform-in-bandwidth results related to linear kernel averages; see Theorem 1 of \cite{Einmahl05} for a precise problem. The derivation there is based on a well-known technique called the peeling device introduced by \cite{Alex85}; see \citet[page 70]{Geer00} for more details. More general function dependent scalings in empirical processes are considered in \cite{Gine06}. In both these works, the functions are taken to be uniformly bounded and extensions to sub-Weibull random variables are desirable. The problems above have a non-random function dependent scaling and there are also some interesting problems involving a random function dependent scaling. One ready example of this is related to the Nadaraya-Watson kernel smoothing estimator of the conditional expectation which is a ratio of two linear kernel averages. We hope to explore some of these in the future.
%Another example concerns self-normalized empirical process given by
%\[
%\sup_{f\in\mathcal{F}}\left(\sum_{i=1}^n f^2(X_i)\right)^{-1/2}{\left|\sum_{i=1}^n f(X_i)\right|}.
%\]
%This a supremum of $t$-statistics and it is interesting to note that the ingredients of the supremum here are sub-Gaussian just under the assumption of finite third moment of $f(X_i)$; see \cite{Shao13} and \cite{Bercu02} for details.
\end{rem}
\section{Proofs of All Results in Section~\ref{sec:Definition} and Appendix~\ref{AppSec:PropGBO}}\label{AppSec:Definition}

%%\par\medskip   %% NOTE: Having to add this due to strange spacing issues in the IAI template for the Supp! -- 2/12/2022.
\begin{proof}[Proof of Proposition~\ref{prop:EquivalentNorm}]
%\begin{subsection}{Proof of Proposition~\ref{prop:EquivalentNorm}}
It is clear from the definition of $\phi_{\alpha, L}(\cdot)$ that
\[
\phi_{\alpha, L}^{-1}(t) = \max\left\{\sqrt{\log(1 + t)}, L\left(\log(1 + t)\right)^{1/\alpha}\right\}\quad\mbox{for all}\quad t\ge 0.
\]
It follows that for all $t\ge 0$,
\begin{equation}\label{eq:InverseEquiv}
\phi_{\alpha, L}^{-1}(t) \le \Psi_{\alpha, L}^{-1}(t) \le 2\phi_{\alpha, L}^{-1}(t).
\end{equation}
Hence for all $x \ge 0$,
\begin{equation}\label{eq:EquivalenceFunctions}
\phi_{\alpha, L}(x/2) \le \Psi_{\alpha, L}(x) \le \phi_{\alpha, L}(x).
\end{equation}
The result now follows by Definition~\ref{def:IncOrliczNorm}.
\end{proof}
%\end{subsection}

%\par\medskip   %% NOTE: Having to add this due to strange spacing issues in the IAI template for the Supp! -- 2/12/2022.
\begin{proof}[Proof of Proposition \ref{prop:Monotone}]
%\begin{subsection}{Proof of Proposition \ref{prop:Monotone}}
\begin{enumerate}
\item[(a)] If $\norm{Y}_{\Psi_{\alpha, L}} = \infty$, then the result is trivially true. If $\delta = \norm{Y}_{\Psi_{\alpha, L}} < \infty$, then for $\eta > \delta$,
\[
\mathbb{E}\left[\Psi_{\alpha, L}\left(\frac{|Y|}{\eta}\right)\right] \le 1\quad\Rightarrow\quad\mathbb{E}\left[\Psi_{\alpha, L}\left(\frac{|X|}{\eta}\right)\right] \le 1.
\]
Letting $\eta\downarrow\delta$ implies the result.
\item[(b)] The result trivially holds if $\norm{X}_{\Psi_{\alpha, L}} = \infty$. Assume $\norm{X}_{\Psi_{\alpha, L}} < \infty$. It is clear from the definition \eqref{eq:InversePsi} of $\Psi_{\alpha, L}^{-1}(t)$,
\[
\Psi_{\alpha, L}^{-1}(t) \le \Psi_{\alpha, K}^{-1}(t)\quad\mbox{for all}\quad t \ge 0.
\]
Observe that for $\eta > \norm{X}_{\Psi_{\alpha, L}}$,
\begin{align*}
\mathbb{E}\left[\Psi_{\alpha, K}\left(|X|/\eta\right)\right] &= \int_0^{\infty} \mathbb{P}\left(|X| \ge \eta\Psi_{\alpha, K}^{-1}\left(t\right)\right)dt\\
&\le \int_0^{\infty} \mathbb{P}\left(|X| \ge \eta\Psi_{\alpha, L}^{-1}(t)\right)dt = \mathbb{E}\left[\Psi_{\alpha, L}\left(\frac{|X|}{\eta}\right)\right] \le 1.
\end{align*}
Letting $\eta \downarrow \norm{X}_{\Psi_{\alpha, L}}$ implies the result.
\end{enumerate}
\end{proof}
%\end{subsection}

%%\par\medskip   %% NOTE: Having to add this due to strange spacing issues in the IAI template for the Supp! -- 2/12/2022.
\begin{proof}[Proof of Proposition \ref{prop:EquivalenceTailMoment}]
%\begin{subsection}{Proof of Proposition \ref{prop:EquivalenceTailMoment}}
From definitions \eqref{eq:AlphaOrliczNorm} and \eqref{eq:InversePsi}, for $\eta > \delta$,
\begin{align*}
\mathbb{P}\left(|X| \ge \eta\left[\sqrt{t} + Lt^{1/\alpha}\right]\right) &= \mathbb{P}\left(\frac{|X|}{\eta} \ge \Psi_{\alpha, L}^{-1}(e^t - 1)\right)\\
&= \mathbb{P}\left(\Psi_{\alpha, L}\left(|X|/\eta\right) + 1 \ge e^t\right)\\
&\le \Big(\mathbb{E}\left[\Psi_{\alpha, L}(|X|/\eta)\right] + 1\Big)\exp(-t) \le 2\exp(-t).
\end{align*}
Now taking limit as $\eta\downarrow \delta$ implies the first part of the result.

For the converse result, set $c(\alpha) = 3^{1/\alpha - 1/2}$. Observe that% from \eqref{eq:EquivalenceFunctions},
\begin{align*}
\mathbb{E}\left[\Psi_{\alpha, c(\alpha)L}\left(\frac{|X|}{\sqrt{3}\delta}\right)\right] %&\le \mathbb{E}\left[\phi_{\alpha, c(\alpha)L}\left(\frac{|X|}{2\sqrt{3}\delta}\right)\right]\\
%&= \int_0^{\infty} \mathbb{P}\left(|X| \ge 2\sqrt{3}\delta\phi_{\alpha, c(\alpha)L}^{-1}(t)\right)dt\\
&= \int_0^{\infty} \mathbb{P}\left(|X| \ge \sqrt{3}\delta\Psi_{\alpha, c(\alpha)L}^{-1}(t)\right)dt\\
&= \int_0^{\infty} \mathbb{P}\left(|X| \ge \sqrt{3}\delta\left\{\sqrt{\log(1 + t)} + c(\alpha)L\left(\log(1 + t)\right)^{1/\alpha}\right\}\right)dt\\
&= \int_0^{\infty} \mathbb{P}\left(|X| \ge \delta\left\{\sqrt{\log(1 + t)^3} + L\left(\log(1 + t)^3\right)^{1/\alpha}\right\}\right)dt\\
&\le 2\int_0^{\infty}\frac{1}{(1 + t)^3}dt \; \le 1.
\end{align*}
This implies $\norm{X}_{\alpha, c(\alpha)L} \le \sqrt{3}\delta$ and completes the proof of the proposition.
\end{proof}
%\end{subsection}

%\par\medskip   %% NOTE: Having to add this due to strange spacing issues in the IAI template for the Supp! -- 2/12/2022.
\begin{proof}[Proof of Proposition \ref{prop:EquivalenceNormMoment}]
%\begin{subsection}{Proof of Proposition \ref{prop:EquivalenceNormMoment}}
For a proof of the first inequality in Proposition~\ref{prop:EquivalenceNormMoment}, %in~\eqref{eq:NormToMoment},
note that it holds trivially if $\norm{X}_{\Psi_{\alpha, L}} = \infty$. Assume $\delta := \norm{X}_{\Psi_{\alpha, L}} < \infty$. Fix $\eta > \delta$. From the hypothesis and inequality~\eqref{eq:EquivalenceFunctions},
\[
\mathbb{E}\left[\Psi_{\alpha, L}(|X|/\eta)\right] \le 1\quad\Rightarrow\quad \mathbb{E}\left[\exp\left(\min\left\{\left(\frac{|X|}{2\eta}\right)^2, \left(\frac{|X|}{2\eta L}\right)^{\alpha}\right\}\right) - 1\right] \le 1.
\]
Thus, for $p\ge 1$, (using the inequalities $x^p/p! \le \exp(x) - 1$ and $(p!)^{1/p} \le p$)
\begin{equation}\label{eq:MinMomentBound}
\norm{\min\left\{\left(\frac{|X|}{2\eta}\right)^2, \left(\frac{|X|}{2\eta L}\right)^{\alpha}\right\}}_p \le p.
\end{equation}
Now observe by the equivalence of inverse functions \eqref{eq:InverseEquiv}, for any $x \ge 0$
\begin{equation}\label{eq:FirstMomentBound}
x \le \Psi_{\alpha, L}^{-1}(\phi_{\alpha, L}(x)) = \left(\min\left\{x^2, \left(\frac{x}{L}\right)^{\alpha}\right\}\right)^{1/2} + L\left(\min\left\{x^2, \left(\frac{x}{L}\right)^{\alpha}\right\}\right)^{1/\alpha}.
\end{equation}
Taking $x = |X|/(2\eta)$ in \eqref{eq:FirstMomentBound} and using triangle inequality of $\norm{\cdot}_p$-norm,
\begin{align}
\qquad \norm{\frac{X}{2\eta}}_p &\le \norm{\min\left\{\left(\frac{|X|}{2\eta}\right)^2, \left(\frac{|X|}{2\eta L}\right)^{\alpha}\right\}}_{\frac{p}{2}}^{\frac{1}{2}} + L\norm{\min\left\{\left(\frac{|X|}{2\eta}\right)^2, \left(\frac{|X|}{2\eta L}\right)^{\alpha}\right\}}_{\frac{p}{\alpha}}^{\frac{1}{\alpha}}.\label{eq:FirstMomentEquiIneq}
% \nonumber\\
% &\le \sqrt{p/2} + L\left(p/\alpha\right)^{1/\alpha}.\label{eq:PGe2Bound}
\end{align}
If $p\ge \alpha$, then from~\eqref{eq:MinMomentBound}
\[
\norm{\min\left\{\left(\frac{|X|}{2\eta}\right)^2, \left(\frac{|X|}{2\eta L}\right)^{\alpha}\right\}}_{{p}/{\alpha}} \le p/\alpha,
\]
and for $1 \le p \le \alpha$,
\[
\norm{\min\left\{\left(\frac{|X|}{2\eta}\right)^2, \left(\frac{|X|}{2\eta L}\right)^{\alpha}\right\}}_{{p}/{\alpha}}^{1/\alpha} \le \norm{\min\left\{\left(\frac{|X|}{2\eta}\right)^2, \left(\frac{|X|}{2\eta L}\right)^{\alpha}\right\}}_{1} \le 1 \le p^{1/\alpha}.
\]
Combining these two inequalities, we get for $p\ge 1$,
\begin{equation}\label{eq:MinMomentWeibullPart}
\norm{\min\left\{\left(\frac{|X|}{2\eta}\right)^2, \left(\frac{|X|}{2\eta L}\right)^{\alpha}\right\}}_{{p}/{\alpha}}^{1/\alpha} \le p^{1/\alpha}\max\left\{1, (1/\alpha)^{1/\alpha}\right\}.
\end{equation}
% If $p\ge 2$,
% \[
% \norm{\min\left\{\left(\frac{|X|}{2\eta}\right)^2, \left(\frac{|X|}{2\eta L}\right)^{\alpha}\right\}}_{{p}/{2}} \le p/2,
% \]
% and for $p = 1$,
% \[
% \norm{\min\left\{\left(\frac{|X|}{2\eta}\right)^2, \left(\frac{|X|}{2\eta L}\right)^{\alpha}\right\}}_{{p}/{2}}^{1/2} \le \norm{\min\left\{\left(\frac{|X|}{2\eta}\right)^2, \left(\frac{|X|}{2\eta L}\right)^{\alpha}\right\}}_{1} \le 1 = \sqrt{p}.
% \]
% Combining these two inequalities we get for $p\ge 1$,
% \begin{equation}\label{eq:MinMomentGaussianPart}
% \norm{\min\left\{\left(\frac{|X|}{2\eta}\right)^2, \left(\frac{|X|}{2\eta L}\right)^{\alpha}\right\}}_{{p}/{2}}^{1/2} \le \sqrt{p}.
% \end{equation}
A similar inequality holds with $(p/\alpha, 1/\alpha)$ replaced by $(p/2, 1/2)$. Substituting inequality~\eqref{eq:MinMomentWeibullPart} in~\eqref{eq:FirstMomentEquiIneq}, it follows for $p\ge 1$ that
\[
\norm{X}_p \le (2\eta)\left[\sqrt{p} + Lp^{1/\alpha}\max\{1, (1/\alpha)^{1/\alpha}\}\right].
\]
Therefore by letting $\eta \downarrow \delta$, for $p\ge 1$,
\[
\norm{X}_p \le 2\norm{X}_{\Psi_{\alpha, L}}\sqrt{p} + 2L\norm{X}_{\Psi_{\alpha, L}}p^{1/\alpha}\max\{1, (1/\alpha)^{1/\alpha}\},
\]
or equivalently,
\[
\frac{1}{2}\min\{1, \alpha^{1/\alpha}\}\sup_{p\ge 1}\,\frac{\norm{X}_p}{\sqrt{p} + Lp^{1/\alpha}} \le \norm{X}_{\Psi_{\alpha, L}}.
\]

\textit{Converse:}\;For a proof of the second inequality in Proposition~\ref{prop:EquivalenceNormMoment}, %in~\eqref{eq:NormToMoment},
set
\[
\Delta := \sup_{p\ge 1}\,\frac{\norm{X}_p}{\sqrt{p} + Lp^{1/\alpha}},
\]
so that
\[
\norm{X}_p \le \Delta\sqrt{p} + L\Delta p^{1/\alpha}\quad\mbox{for all}\quad p\ge1.
\]
Note by Markov's inequality and these moment bounds that for any $t \ge 1$,
\[
\mathbb{P}\left(|X| \ge e\Delta\sqrt{t} + eL\Delta t^{1/\alpha}\right) \le \exp(-t),
\]
and for $0 < t < 1$ (trivially),
\[
\mathbb{P}\left(|X| \ge e\Delta\sqrt{t} + eL\Delta t^{1/\alpha}\right) \le 1.
\]
Hence, for any $t > 0$,
\begin{equation}\label{eq:MomentToTail}
\mathbb{P}\left(|X| \ge e\Delta\sqrt{t} + eL\Delta t^{1/\alpha}\right) \le e\exp(-t).
\end{equation}
Take $K = e\max\{2, 4^{1/\alpha}\}$. Observe that,
\begin{align*}
\mathbb{E}\left[\Psi_{\alpha, L}\left(\frac{|X|}{K\Delta}\right)\right] %&\le
%\mathbb{E}\left[\phi_{\alpha, K}\left(\frac{|X|}{2\sqrt{4}\delta}\right)\right]\\
%&= \int_0^{\infty} \mathbb{P}\left(|X| \ge 2\sqrt{4}\delta\phi^{-1}_{\alpha, K}(t)\right)dt\\
&= \int_0^{\infty}\mathbb{P}\left(|X| \ge K\Delta\Psi_{\alpha, L}^{-1}(t)\right)dt\\
&= \int_0^{\infty} \mathbb{P}\left(|X| \ge K\Delta\left\{\sqrt{\log(1 + t)} + L(\log(1 + t))^{1/\alpha}\right\}\right)dt\\
&\le \int_0^{\infty} \mathbb{P}\left(|X| \ge e\Delta\sqrt{\log (1 + t)^4} + eL\Delta(\log (1 + t)^4)^{1/\alpha}\right)dt\\
&\le e\int_0^{\infty}\frac{1}{(1 + t)^4}dt = e/3 < 1.
\end{align*}
Therefore, $\norm{X}_{\Psi_{\alpha, L}} \le K\Delta.$
\end{proof}
%\end{subsection}

%\par\medskip   %% NOTE: Having to add this due to strange spacing issues in the IAI template for the Supp! -- 2/12/2022.
For the proofs of the results in Section~\ref{sec:Indep}, we use the following alternative result regarding inversion of moment bounds to get bounds on the GBO norm. %norm bound from the moment bound.
\begin{prop}\label{prop:NormMomentEquiSupp}
If $\norm{X}_p \le C_1\sqrt{p} + C_2p^{1/\alpha},$ holds for $p\ge 1$ and some constants $C_1, C_2$, then $\norm{X}_{\Psi_{\alpha, K}} \le 2eC_1,$ where $K := 4^{1/\alpha}C_2/(2C_1)$.
\end{prop}

%%\par\medskip   %% NOTE: Having to add this due to strange spacing issues in the IAI template for the Supp! -- 2/12/2022.
\begin{proof}
From the proof of Proposition~\ref{prop:EquivalenceNormMoment} (or, in particular~\eqref{eq:MomentToTail}), we get
\begin{equation}\label{eq:MomentToTailSupp}
\mathbb{P}\left(|X| \ge eC_1\sqrt{t} + eC_2t^{1/\alpha}\right) \le e\exp(-t),\quad\mbox{for all}\quad t\ge0.
\end{equation}
Take $K = 4^{1/\alpha}C_2/(2C_1)$ as in the statement of the result. Observe that with $\delta := eC_1$,
\begin{align*}
\mathbb{E}\left[\Psi_{\alpha, K}\left(\frac{|X|}{\sqrt{4}\delta}\right)\right] %&\le
%\mathbb{E}\left[\phi_{\alpha, K}\left(\frac{|X|}{2\sqrt{4}\delta}\right)\right]\\
%&= \int_0^{\infty} \mathbb{P}\left(|X| \ge 2\sqrt{4}\delta\phi^{-1}_{\alpha, K}(t)\right)dt\\
&= \int_0^{\infty}\mathbb{P}\left(|X| \ge \sqrt{4}\delta\Psi_{\alpha, K}^{-1}(t)\right)dt\\
&= \int_0^{\infty} \mathbb{P}\left(|X| \ge \sqrt{4}\delta\left\{\sqrt{\log(1 + t)} + K(\log(1 + t))^{1/\alpha}\right\}\right)dt\\
&= \int_0^{\infty} \mathbb{P}\left(|X| \ge eC_1\sqrt{\log (1 + t)^4} + eC_2L(\log (1 + t)^4)^{1/\alpha}\right)dt\\
&\le e\int_0^{\infty}\frac{1}{(1 + t)^4}dt = e/3 < 1.
\end{align*}
Therefore, $\norm{X}_{\Psi_{\alpha, K}} \le 2eC_1.$
\end{proof}

%\par\medskip   %% NOTE: Having to add this due to strange spacing issues in the IAI template for the Supp! -- 2/12/2022.
\begin{proof}[Proof of Proposition \ref{prop:QuasiNorm}]
%\begin{subsection}{Proof of Proposition \ref{prop:QuasiNorm}}
Assume without loss of generality that $\norm{X_i}_{\Psi_{\alpha, L}} < \infty$ for all $1\le i\le n$, as otherwise the result is trivially true. If $\alpha > 1$, then $\Psi_{\alpha, L}^{-1}(\cdot)$ is a concave function and hence $\norm{\cdot}_{\Psi_{\alpha, L}}$ is a proper norm proving the result. For $\alpha < 1$, the result follows trivially by noting that both sides of the inequality in Proposition~\ref{prop:EquivalenceNormMoment} %inequality~\eqref{eq:NormToMoment}
are norms. This completes the proof.
%,
%\[
%\norm{X_1 + X_2}_{\Psi_{\alpha, L}} \le Q_{\alpha}\left[\norm{X_1}_{\Psi_{\alpha, L}} + \norm{X_2}_{\Psi_{\alpha, L}}\right].
%\]
\end{proof}
%\end{subsection}

%\par\medskip   %% NOTE: Having to add this due to strange spacing issues in the IAI template for the Supp! -- 2/12/2022.
\begin{proof}[Proof of Proposition \ref{prop:MaximalPsi}]
%\begin{subsection}{Proof of Proposition \ref{prop:MaximalPsi}}
By union bound and Proposition \ref{prop:EquivalenceTailMoment},
\begin{align*}
\mathbb{P}&\left(\max_{1\le j\le N}|X_j| \ge \Delta\left\{\sqrt{t + \log N} + L\left(t + \log N\right)^{1/\alpha}\right\}\right)\\
&\le \sum_{j = 1}^N \mathbb{P}\left(|X_j| \ge \Delta\left\{\sqrt{t + \log N} + L\left(t + \log N\right)^{1/\alpha}\right\}\right)\\
&\le 2N\exp\left(-t - \log N\right)\le \frac{2N}{N}\exp(-t).
\end{align*}
Hence the tail bound follows. To bound the norm note that for all $\alpha > 0$,
\[
(t + \log N)^{1/\alpha} \le M(\alpha)\left(t^{1/\alpha} + (\log N)^{1/\alpha}\right),
\]
and from the tail bound of the maximum,
\begin{align*}
\mathbb{P}&\left(Z \ge \delta\left\{\sqrt{t} + M(\alpha)Lt^{1/\alpha}\right\}\right)\\
&\le \mathbb{P}\left(\max_{1\le j\le N}|X_j| \ge \Delta\left\{\sqrt{t + \log N} + L\left(t + \log N\right)^{1/\alpha}\right\}\right) \le 2\exp(-t),
\end{align*}
where
\[
Z := \left(\max_{1\le j\le N}|X_j| - \Delta\left\{\sqrt{\log N} + M(\alpha)L\left(\log N\right)^{1/\alpha}\right\}\right)_+.
\]
Hence by Proposition \ref{prop:EquivalenceTailMoment}, $\norm{Z}_{\Psi_{\alpha, K(\alpha)}} \le \sqrt{3}\Delta$. The result follows by Proposition \ref{prop:QuasiNorm} along with the fact
\[
\max_{1\le j\le N}|X_j| \le Z + \Delta\left\{\sqrt{\log N} + M(\alpha)L\left(\log N\right)^{1/\alpha}\right\},
\]
and by noting that the random variables on both sides are non-negative.
\end{proof}
%\end{subsection}

%\par\medskip   %% NOTE: Having to add this due to strange spacing issues in the IAI template for the Supp! -- 2/12/2022.
\begin{proof}[Proof of Proposition \ref{prop:RatioMaximal}]
By homogeneity, we can without loss of generality assume that
\[
\norm{X_k}_{\Psi_{\alpha, L}} = 1\quad\mbox{for all}\quad k \ge 1.
\]
Note that by the union bound, for $t\ge 0$,
\begin{align*}
\mathbb{P}\left(\sup_{k \ge 1}\right.&\left.\left(|X_k| - \sqrt{2\log(1 + k)} - M(\alpha)L(2\log(1 + k))^{1/\alpha}\right)_+ \ge \sqrt{t} + M(\alpha)Lt^{1/\alpha}\right)\\
&\le \sum_{k \ge 1}\mathbb{P}\left(|X_k| \ge \sqrt{t + 2\log(1 + k)} + L(t + 2\log(1 + k))^{1/\alpha}\right)\\
&\le \sum_{k \ge 1} \frac{2}{(1 + k)^2}\exp(-t)\\
&\le \frac{2(\pi^2 - 6)}{6}\exp(-t) < 2\exp(-t).
\end{align*}
Hence by Proposition \ref{prop:EquivalenceTailMoment},
\begin{equation}\label{eq:PositivePartSupBound}
\norm{\sup_{k\ge 1}\left(|X_k| - \sqrt{2\log(1 + k)} - M(\alpha)L(2\log(1 + k))^{1/\alpha}\right)_+}_{\Psi_{\alpha, c(\alpha)M(\alpha)L}} \le \sqrt{3}.
\end{equation}
Recall $c(\alpha) = 3^{1/\alpha}/\sqrt{3}$. Since $\sqrt{2\log(1 + k)} \ge 1$ for $k \ge 1$, using \eqref{eq:PositivePartSupBound}, it follows that
\[
\norm{\sup_{k\ge 1}\left(\frac{|X_k|}{\sqrt{2\log(1 + k)} + M(\alpha)L(2\log(1 + k))^{1/\alpha}} - 1\right)_+}_{\Psi_{\alpha, c(\alpha)M(\alpha)L}}\le 1.5.
\]
The result now follows by an application of Proposition \ref{prop:QuasiNorm}.
%\mathbb{P}\left(\vphantom{\sup_{k\ge 1}\frac{|X_k|}{a\Psi_{\alpha, \sqrt{L}}^{-1}(k)}}\Psi_{\alpha, \sqrt{L}}\right.&\left.\left(\sup_{k\ge 1}\frac{|X_k|}{a\Psi_{\alpha, \sqrt{L}}^{-1}(k)}\right) \ge t\right)\\ &\le \sum_{k = 1}^{N}\mathbb{P}\left(|X_k| \ge a\Psi_{\alpha, \sqrt{L}}^{-1}(k)\Psi_{\alpha, \sqrt{L}}^{-1}(t)\right)\\
%&\le \sum_{k = 1}^N \mathbb{P}\left(|X_k| \ge a\sqrt{\log(1 + t)\log(1 + k)} + aL\left(\log(1 + t)\log(1 + k)\right)^{1/\alpha}\right)\\
%&\le \sum_{k = 1}^N \mathbb{P}\left(|X_k| \ge \sqrt{8\log(1 + t)\log(1 + k)} + L\left(8\log(1 + t)\log(1 + k)\right)^{1/\alpha}\right)\\
%&\le \sum_{k = 1}^N \frac{1}{\exp\left(2\log(1 + k^2)\log(1 + t^2)\right) - 1}\\
%&\le \sum_{k = 1}^N \frac{1}{k^2t^2} \le \frac{\pi^2}{6t^2}.
%Here the following inequality is used.
%\[
%(e^x - 1)(e^y - 1) \le e^{2xy} - 1\quad\mbox{for}\quad x, y \ge 1.
%\]
%Thus,
%\[
%\mathbb{E}\left[\Psi_{\alpha, \sqrt{L}}\left(\sup_{k \ge 1}\frac{|X_k|}{a\Psi_{\alpha, \sqrt{L}}^{-1}(k)}\right)\right] \le 1 + \int_1^{\infty}\frac{\pi^2}{6t^2}dt = 1 + \pi^2/6 < e.
%\]
%Following the proof of Proposition \ref{prop:EquivalenceNormMoment}, it can be seen that
%\[
%\norm{\sup_{k \ge 1}\frac{|X_k|}{ab\Psi_{\alpha, \sqrt{L}}^{-1}(k)}}_{\Psi_{\alpha, \sqrt{L}}} \le 1,\quad\mbox{where}\quad b = \max\{4^{1/\alpha}, 2\}.
%\]
\end{proof}

\section{Proofs of All Results in Section \ref{sec:Indep}}\label{AppSec:Indep}

\begin{proof}[Proof of Theorem \ref{prop:SumNewOrliczVex}]
Since $a_iX_i = (a_i\norm{X_i}_{\psi_{\alpha}})(X_i/\norm{X_i}_{\psi_{\alpha}})$, we can without loss of generality assume $\norm{X_i}_{\psi_{\alpha}} = 1$. Define $Y_i = (|X_i| - \eta)_+$ with $\eta = (\log 2)^{1/\alpha}$. This implies that
\begin{equation}\label{eq:Translation}
\mathbb{P}\left(|X_i| \ge t\right) \le 2\exp(-t^{\alpha})\quad\Rightarrow\quad \mathbb{P}\left(Y_i \ge t\right) \le \exp(-t^{\alpha}).
\end{equation}
By symmetrization inequality (Proposition 6.3 of \cite{LED91}),
\[
\norm{\sum_{i=1}^n a_iX_i}_p \le 2\norm{\sum_{i=1}^n \varepsilon_ia_iX_i}_p,
\]
for independent Rademacher random variables $\varepsilon_i, 1\le i\le n$ independent of $X_i, 1\le i\le n$. Using the fact that $\varepsilon_iX_i $ is identically distributed as $\varepsilon_i|X_i|$ and by Theorem 1.3.1 of \cite{DeLaPena99}, it follows that
\begin{align}
\norm{\sum_{i=1}^n a_iX_i}_p \le 2\norm{\sum_{i=1}^n \varepsilon_ia_i|X_i|}_p &\le 2\norm{\sum_{i=1}^n \varepsilon_ia_i(\eta + Y_i)}_p\\ &\le 2\norm{\sum_{i=1}^n \varepsilon_ia_iY_i}_p + 2\eta\norm{\sum_{i=1}^n \varepsilon_ia_i}_p\nonumber\\
&\le 2\norm{\sum_{i=1}^n \varepsilon_ia_iY_i}_p + 2\eta\sqrt{p}\norm{a}_2.\label{eq:FirstInequalityWeibull}
\end{align}
By inequality \eqref{eq:Translation},
\[
\norm{\sum_{i=1}^n a_i\varepsilon_iY_i}_p \le \norm{\sum_{i=1}^n a_iZ_i}_p,
\]
for symmetric independent random variables $Z_i, 1\le i\le n$ satisfying $\mathbb{P}\left(|Z_i| \ge t\right) = \exp(-t^{\alpha})$ for all $t\ge 0.$ Now we apply the bound in Examples 3.2 and 3.3 of \cite{LAT97} in combination with Theorem 2 there.
%It is easy to check that for $p\ge 1$,
%\[
%\mathbb{E}[|Z_i|^p] \le \int_0^{\infty} t^p\alpha t^{\alpha - 1}\exp(-t^{\alpha})dt = \int_0^{\infty} u^{p/\alpha}e^{-u}du = \Gamma\left(1 + \frac{p}{\alpha}\right)
%\]
%Thus, for $p\ge 1$ and $0 < \alpha \le 1$,
%\[
%\mathbb{E}[|X_i|^p] \le \sqrt{2\pi}\exp(1/12)\left(p/\alpha\right)^{p/\alpha}.
%\]
%See, for example, Theorem 1.1 of \cite{JAMES15}.

\textit{Case $\alpha \le 1$:} Example 3.3 of \cite{LAT97} shows that for $p\ge 2$,
\begin{equation}\label{eq:LogConvex}
\norm{\sum_{i=1}^n a_iZ_i}_p \le \max\left\{p^{1/2}\sqrt{2}\norm{a}_22^{1/\alpha}, \frac{p^{1/\alpha}\norm{a}_p}{\exp(1/(2e))}\right\}\frac{e^3(2\pi)^{1/4}e^{1/24}}{\alpha^{1/\alpha}}.
\end{equation}
Using the proof of Corollary 1.2 of \cite{Bog15} and substituting the resulting bound in \eqref{eq:FirstInequalityWeibull}, it follows that for $p\ge2$,
\begin{equation}\label{eq:MomentBoundConvex}
\norm{\sum_{i=1}^n a_iX_i}_p \le \sqrt{8}e^3(2\pi)^{1/4}e^{1/24}(e^{2/e}/\alpha)^{1/\alpha}\left[\sqrt{p}\norm{a}_2 + p^{1/\alpha}\norm{a}_{\infty}\right].
\end{equation}
Corollary 1.2 of \cite{Bog15} uses the inequality $p^{1/p} \le e$ but using $p^{1/p} \le e^{1/e}$ gives the bound above; see also Remark 3, Equation (3) of \cite{Kole15}. Set $C'(\alpha) := \sqrt{8}e^3(2\pi)^{1/4}e^{1/24}(e^{2/e}/\alpha)^{1/\alpha}$. For $p = 1$ note that
\[
\norm{\sum_{i=1}^n a_iX_i}_1 \le \norm{\sum_{i=1}^n a_iX_i}_2 \le C'(\alpha)\left[\sqrt{2}\norm{a}_1 + 2^{1/\alpha}\norm{a}_{\infty}\right]%\left(\sum_{i=1}^n a_i^2\mathbb{E}\left[X_i^2\right]\right)^{1/2} \le \sqrt{8}e^3(2\pi)^{1/4}e^{1/24}(e^{2/e}/\alpha)^{1/\alpha}\norm{a}_2.
\]
Thus for $p\ge 1$,
\[
\norm{\sum_{i=1}^n a_iX_i}_p \le C'(\alpha)\max\{\sqrt{2}, 2^{1/\alpha}\}\left[\sqrt{p}\norm{a}_2 + p^{1/\alpha}\norm{a}_{\infty}\right].
\]
% using $\norm{\cdot}_1 \le \norm{\cdot}_2$ for $\sum a_iX_i$.
\iffalse
, for any $x\in\mathbb{R}^k (k\ge 1)$, $p\ge 2$ and $\gamma > 0$,
\[
\norm{x}_p \le e^{6\gamma}\left(e^{-\gamma p}\norm{x}_2 + \norm{x}_{\infty}\right).
\]
This implies by taking $x_i = |a_i|p^{1/\alpha}$,
\begin{align*}
\left(\sum_{i=1}^n |a_i|^pp^{p/\alpha}\right)^{1/p} &\le e^{6\gamma}\left(e^{-\gamma p}p^{1/\alpha}\left(\sum_{i=1}^n a_i^2\right)^{1/2} + p^{1/\alpha}\max_{1\le i\le n} |a_i|\right).
\end{align*}
Taking $\gamma = 1/\alpha$ and using the fact $p^{1/\alpha} \le \exp(p/\alpha)$,
\[
p^{1/\alpha}\norm{a}_p \le e^{6/\alpha}\left(\norm{a}_2 + p^{1/\alpha}\norm{a}_{\infty}\right),
\]
for $p\ge 1$. Combining this inequality with \eqref{eq:LogConvex} gives
\begin{equation}\label{eq:MomentBoundConvex}
\norm{\sum_{i=1}^n a_iX_i}_p \le \left(\sqrt{p}\norm{a}_2 + p^{1/\alpha}\norm{a}_{\infty}\right)\frac{2\sqrt{2}e^3(2\pi)^{1/4}e^{1/24}e^{6/\alpha}}{\alpha^{1/\alpha}}\quad\mbox{for}\quad p\ge 1.
\end{equation}
\fi
Hence the result follows by Proposition \ref{prop:NormMomentEquiSupp}.
\iffalse
\end{proof}

%% NOTE: This proof (a part of it at least) is not included (actually commented out and merged with the previous proof). Hence not trying to include any vspace here - 2/12/2022.
\begin{proof}[Proof of Proposition \ref{prop:SumNewOrliczCave}]
The proof closely follows the proof of Proposition \ref{prop:SumNewOrliczVex} expect for the usage of example 3.2 of \cite{LAT97}. As in the proof there, assume without loss of generality that $\norm{X_i}_{\psi_{\alpha}} = 1$ for all $1\le i\le n$. Following the argument there, it follows that
\[
\norm{\sum_{i=1}^n a_iX_i}_p \le 2\norm{\sum_{i=1}^n a_iZ_i}_p + 2\eta\sqrt{p}\norm{a}_2,
\]
where $\eta = (\log 2)^{1/\alpha}$ and $Z_i, 1\le i\le n$ are independent symmetric random variables satisfying $\mathbb{P}\left(|Z_i| \ge t\right) = \exp(-t^{\alpha})$ for all $t\ge 0$. Now
\fi

\textit{Case $\alpha \ge 1$:} it follows from (13) of Example 3.2 of \cite{LAT97} that for $p\ge 2$,
\[
\norm{\sum_{i=1}^n a_iZ_i}_p \le 4e\left[\sqrt{p}\left(\sum_{i=1}^n a_i^2\right)^{1/2} + p^{1/\alpha}\left(\sum_{i=1}^n |a_i|^{\beta}\right)^{1/\beta}\right],
\]
with $\beta = \beta(\alpha)$ as mentioned in the statement. Therefore, for $p\ge 2$,
\begin{equation}\label{eq:MomentBoundConcave}
\norm{\sum_{i=1}^n a_iX_i}_p \le (4e + 2\eta)\sqrt{p}\norm{a}_2 + 4ep^{1/\alpha}\norm{a}_{\beta}.
\end{equation}
For $p = 1$, note that
\[
\norm{\sum_{i=1}^n a_iX_i}_1 \le \norm{\sum_{i=1}^na_iX_i}_2 \le \max\{\sqrt{2}, 2^{1/\alpha}\}\left[(4e + 2\eta)\norm{a}_2 + 4e\norm{a}_{\beta}\right],
\]
and so, for $p\ge 1$,
\[
\norm{\sum_{i=1}^n a_iX_i}_p \le \max\{\sqrt{2}, 2^{1/\alpha}\}\left[(4e + 2\eta)\sqrt{p}\norm{a}_2 + 4ep^{1/\alpha}\norm{a}_{\beta}\right].
\]
The result now follows by an application of Proposition \ref{prop:NormMomentEquiSupp}. The tail bound follows from Proposition~\ref{prop:EquivalenceTailMoment}.
\end{proof}

%\par\medskip   %% NOTE: Having to add this due to strange spacing issues in the IAI template for the Supp! -- 2/12/2022.
Before proving moment inequalities with unbounded variables, we first provide the Bernstein moment bounds for bounded random variables since this forms an integral part of our proofs.

\begin{prop}(Bernstein's Inequality for Bounded Random Variables)\label{prop:BernsteinMoment}
Suppose $Z_1, Z_2,$ $\ldots, Z_n$ are independent random variables with mean zero and uniformly bounded by $U$ in absolute value. Then for $p\ge 1$,
\[
\norm{\sum_{i=1}^n Z_i}_p \le \sqrt{6p}\left(\sum_{i=1}^n \mathbb{E}\left[Z_i^2\right]\right)^{1/2} + 10pU.
\]
\end{prop}

%%\par\medskip   %% NOTE: Having to add this due to strange spacing issues in the IAI template for the Supp! -- 2/12/2022.
\begin{proof}[Proof of Proposition \ref{prop:BernsteinMoment}]
By Theorem 3.1.7 of \cite{GINE16},
\[
\mathbb{P}\left(|S_n| \ge t\right) \le 2\exp\left(-\frac{t^2}{2\sigma_n^2 + 2Ut/3}\right),\quad\mbox{for all}\quad t\ge 0.
\]
where $$S_n := \sum_{i=1}^n Z_i,\quad\mbox{and}\quad\sigma_n^2 := \sum_{i=1}^n \mathbb{E}\left[Z_i^2\right].$$
To bound the moments of $|S_n|$, note that
\begin{enumerate}
  \item[(a)] If $2Ut/3 \le 2\delta\sigma_n^2$ (or equivalently, $t \le 3\delta\sigma_n^2/U$), then
  \[
  \exp\left(-\frac{t^2}{2\sigma_n^2 + 2Ut/3}\right) \le \exp\left(-\frac{t^2}{2(1 + \delta)\sigma_n^2}\right).
  \]
  \item[(b)] If $2Ut/3 \ge 2\delta\sigma_n^2$, then
  \[
  \exp\left(-\frac{t^2}{2\sigma_n^2 + 2Ut/3}\right) \le \exp\left(-t\frac{3\delta}{2U(1 + \delta)}\right).
  \]
\end{enumerate}
Set $t_0 := 3\delta\sigma_n^2/U$. Now observe that for $p\ge 2$,
\begin{align*}
\mathbb{E}&\left[|S_n|^p\right]\\ &= \int_0^{\infty} pt^{p-1}\mathbb{P}\left(|S_n| \ge t\right)dt\\
&\le 2\int_0^{\infty} pt^{p-1}\exp\left(-\frac{t^2}{2\sigma_n^2 + 2Ut/3}\right)dt\\
&= 2\int_0^{t_0} pt^{p-1}\exp\left(-\frac{t^2}{2(1 + \delta)\sigma_n^2}\right)dt + 2\int_{t_0}^{\infty}pt^{p-1}\exp\left(-\frac{3\delta t}{2U(1 + \delta)}\right)dt\\
&\le 2\int_0^{\infty} pt^{p-1}\exp\left(-\frac{t^2}{2(1 + \delta)\sigma_n^2}\right)dt + 2\int_{0}^{\infty}pt^{p-1}\exp\left(-\frac{3\delta t}{2U(1 + \delta)}\right)dt\\
&=: \mathbf{I} + \mathbf{II}.
\end{align*}
By a change of variable for $\mathbf{I}$, we have
\begin{align*}
\mathbf{I} &= 2\int_0^{\infty} pt^{p-1}\exp\left(-\frac{t^2}{2(1 + \delta)\sigma_n^2}\right)dt\\
% &= 2\int_0^{\infty} p\left(z\sqrt{(1 + \delta)\sigma_n^2}\right)^{p-1}\exp\left(-\frac{z^2}{2}\right)dz\sqrt{(1 + \delta)\sigma_n^2}\\
&= 2\left(\sqrt{(1 + \delta)\sigma_n^2}\right)^{p}\int_0^{\infty} pz^{p-1}\exp\left(-\frac{z^2}{2}\right)dz\\
&\overset{(1)}{=} 2\left(\sqrt{(1 + \delta)\sigma_n^2}\right)^{p}2^{p/2}\Gamma\left(1 + \frac{p}{2}\right)\\
&\overset{(2)}{\le} 2\left(\sqrt{2(1 + \delta)\sigma_n^2}\right)^{p}\sqrt{2\pi}\exp\left(-1 - \frac{p}{2}\right)\left(1 + \frac{p}{2}\right)^{(p+1)/2}\exp\left(\frac{1}{12(1 + p/2)}\right).
\end{align*}
Inequality (1) above can be found in Exercise 3.3.4(a) of \cite{GINE16} and inequality (2) follows from Theorem 1.1 of \cite{JAMES15}. Simplifying the above bound for $p \ge 2$, we get
\begin{align*}
\mathbf{I}^{1/p} &\le \sqrt{\frac{2(1 + \delta)\sigma_n^2}{e}}\left(\frac{2\sqrt{2\pi}}{e}\right)^{1/p}\left(1 + \frac{p}{2}\right)^{\frac{1}{2} + \frac{1}{2p}}\exp\left(\frac{1}{12p(1 + p/2)}\right)\\
% &\le \sqrt{\frac{2(1 + \delta)\sigma_n^2}{e}}\left(\frac{2\sqrt{2\pi}}{e}\right)^{1/2}\left(1 + \frac{p}{2}\right)^{\frac{1}{2} + \frac{1}{2p}}\exp\left(\frac{1}{48}\right)\\
% &\le \sqrt{\frac{2(1 + \delta)\sigma_n^2}{e}}\left(\frac{2\sqrt{2\pi}}{e}\right)^{1/2}\left(1 + \frac{p}{2}\right)^{1/2}\left(1 + \frac{p}{2}\right)^{\frac{1}{2p}}\exp\left(\frac{1}{48}\right)\\
% &\le \sqrt{\frac{2(1 + \delta)\sigma_n^2}{e}}\left(\frac{2\sqrt{2\pi}}{e}\right)^{1/2}\left(1 + \frac{p}{2}\right)^{1/2}2^{1/4}\exp\left(\frac{1}{48}\right)\\
&\le \sqrt{\frac{2(1 + \delta)}{e}}\left(\frac{2\sqrt{2\pi}}{e}\right)^{1/2}2^{1/4}\exp\left(\frac{1}{48}\right){\sigma_n\sqrt{p}} \le {2.1\sigma_n\sqrt{p}},\quad\mbox{(taking }\delta = 1).
\end{align*}
To bound $\mathbf{II}$, note that by change of variable
\begin{align*}
\mathbf{II} &= 2\int_0^{\infty} pt^{p-1}\exp\left(-\frac{3\delta t}{2U(1 + \delta)}\right)dt = 2\left(\frac{2U(1 + \delta)}{3\delta}\right)^p\int_0^{\infty} pz^{p-1}\exp(-z)dz\\
&\overset{(1')}{=} 2\left(\frac{2U(1 + \delta)}{3\delta}\right)^p\Gamma(1 + p)\\
&\overset{(2')}{\le} 2\left(\frac{2U(1 + \delta)}{3\delta}\right)^p\sqrt{2\pi}(1 + p)^{p + \frac{1}{2}}\exp(-p - 1)\exp\left(\frac{1}{12(p+1)}\right).
\end{align*}
Here again inequalities (1$'$) and (2$'$) follows from Exercise 3.3.4 (a) of \cite{GINE16} and Theorem 1.1 of \cite{JAMES15}, respectively. This implies for $p\ge 2$, that
\begin{align*}
\mathbf{II}^{1/p} &\le \left(\frac{2eU(1 + \delta)}{3\delta}\right)\left(\frac{2\sqrt{2\pi}}{e}\right)^{1/p}(1 + p)^{1 + \frac{1}{2p}}\exp\left(\frac{1}{12p(p+1)}\right)\\
&\le \frac{2eU(1 + \delta)}{3\delta}\left(\frac{2\sqrt{2\pi}}{e}\right)^{1/2}\left(\frac{3p}{2}\right)(1 + p)^{1/(2p)}\exp\left(\frac{1}{72}\right)\\
&\le \frac{2e(1 + \delta)}{3\delta}\left(\frac{2\sqrt{2\pi}}{e}\right)^{1/2}\left(\frac{3}{2}\right)3^{1/4}\exp\left(\frac{1}{72}\right){ Up} \le 10{Up},
\end{align*}
also for $\delta = 1$. Therefore, for $p\ge 1$,
\[
\left(\mathbb{E}\left[\left|S_n\right|^p\right]\right)^{1/p} \le 2.1\sigma_n\sqrt{p} + 10pU \le \sqrt{6p\sigma_n^2} + 10pU.
\]
\iffalse
Combining Theorem 3.1.5 and Proposition 3.1.6 of \cite{GINE16}, it follows that for any $x\ge 0$,
\[
\mathbb{P}\left(\left|\sum_{i=1}^n Z_i\right| \ge \sqrt{2x}\left(\sum_{i=1}^n \mathbb{E}\left[Z_i^2\right]\right)^{1/2} + \frac{Ux}{3}\right) \le 2\exp(-x).
\]
This implies by Proposition \ref{prop:EquivalenceTailMoment},
\[
\norm{\sum_{i=1}^n Z_i}_{\Psi_{1, L_n}} \le \sqrt{6}\left(\sum_{i=1}^n \mathbb{E}\left[Z_i^2\right]\right)^{1/2},\quad\mbox{for}\quad L_n := \frac{\sqrt{3}}{\sqrt{2}}U\left(\sum_{i=1}^n \mathbb{E}\left[Z_i^2\right]\right)^{-1/2}.
\]
By an application of Proposition \ref{prop:EquivalenceNormMoment}, for $p \ge 1$,
\[
\norm{\sum_{i=1}^n Z_i}_p \le \sqrt{6p}\left(\sum_{i=1}^n \mathbb{E}\left[Z_i^2\right]\right)^{1/2} + 3pU,
\]
proving the result.
\fi
\end{proof}

%\par\medskip   %% NOTE: Having to add this due to strange spacing issues in the IAI template for the Supp! -- 2/12/2022.
\begin{proof}[Proof of Theorem \ref{prop:SimilarBernstein}]
The method of proof is a combination of truncation and Hoffmann-Jorgensen's inequality. Define $$Z = \max_{1\le i\le n}|X_i|,\quad \rho = 8\mathbb{E}\left[Z\right],\quad K = \max_{1\le i\le n}\norm{X_i}_{\psi_{\alpha}},$$
\[
X_{i,1} = X_i\mathbbm{1}\{|X_i| \le \rho\} - \mathbb{E}\left[X_i\mathbbm{1}\{|X_i| \le \rho\}\right],\quad\mbox{and}\quad X_{i,2} = X_i - X_{i,1}.
\]
It is clear that $X_i = X_{i,1} + X_{i,2}$ and $|X_{i,1}| \le 2\rho$ for $1\le i\le n$. Also by triangle inequality, for $p\ge 1$,
\[
\norm{\sum_{i=1}^n X_i}_p \le \norm{\sum_{i=1}^n X_{i,1}}_p + \norm{\sum_{i=1}^n X_{i,2}}_p.
\]
Now note that for $1\le i\le n$,
\[
\mathbb{E}[X_{i,1}^2] = \mbox{Var}\left(X_{i,1}\right) = \mbox{Var}\left(X_{i}\mathbbm{1}_{\{|X_i| \le \rho\}}\right)\le \mathbb{E}[X_i^2].
\]
Thus, applying Bernstein's inequality (Proposition \ref{prop:BernsteinMoment}), for $p\ge 1$,
\begin{equation*}%\label{eq:BoundPart0}
\norm{\sum_{i=1}^n X_{i,1}}_p \le \sqrt{6p}\left(\sum_{i=1}^n \mathbb{E}\left[X_i^2\right]\right)^{1/2} + 20p\rho.
\end{equation*}
By Hoffmann-Jorgensen's inequality (Proposition 6.8 of~\cite{LED91}) and by the choice of $\rho$,
\[
\norm{\sum_{i=1}^n X_{i,2}}_{1} \le 2\norm{\sum_{i=1}^n |X_i|\mathbbm{1}\{|X_i| \ge \rho\}}_1 \le 16\norm{Z}_1,
\]
since
\[
\mathbb{P}\left(\max_{1 \leq k \leq n}  \sum_{i=1}^k |X_i|\mathbbm{1}\{|X_i| \ge \rho\} > 0\right) ~\le~ \mathbb{P}\left(Z \ge \rho\right) \le 1/8.
\]
Therefore, by Theorem 6.21 of \cite{LED91},
\[
\norm{\sum_{i=1}^n X_{i,2}}_{\psi_{\alpha}} \le 17K_{\alpha}\norm{Z}_{\psi_{\alpha}},
\]
where the constant $K_{\alpha}$ is given in Theorem 6.21 of \cite{LED91}. Hence, for $p\ge 1$,
\[
\norm{\sum_{i=1}^n X_{i,2}}_p \le C_{\alpha}K_{\alpha}K(\log(n + 1))^{1/\alpha}p^{1/\alpha},
\]
for some constant $C_{\alpha}$ depending on $\alpha$. Therefore, for $p\ge 1$,
\begin{equation}\label{eq:MomentBoundAlphaLess1}
\norm{\sum_{i=1}^n X_i}_p \le \sqrt{6p}\left(\sum_{i=1}^n \mathbb{E}\left[X_i^2\right]\right)^{1/2} + C_{\alpha}K_{\alpha}K(\log (n+1))^{1/\alpha}p^{1/\alpha},
\end{equation}
for some constant $C_{\alpha} > 0$ (possibly different from the previous line). Hence the result follows by Proposition \ref{prop:NormMomentEquiSupp}.
\end{proof}

%\par\medskip   %% NOTE: Having to add this due to strange spacing issues in the IAI template for the Supp! -- 2/12/2022.
\begin{proof}[Proof of Theorem \ref{prop:BernsteinLargerThan1}]
The proof follows the same technique as that of Theorem \ref{prop:SimilarBernstein}. Define $Z = \max_{1\le i\le n}|X_i|$, $\rho = 8\mathbb{E}\left[Z\right]$, $K = \max_{1\le i\le n}\norm{X_i}_{\psi_{\alpha}},$
\[
X_{i,1} = X_i\mathbbm{1}\{|X_i| \le \rho\} - \mathbb{E}\left[X_i\mathbbm{1}\{|X_i| \le \rho\}\right],\quad\mbox{and}\quad X_{i,2} = X_i - X_{i,1}.
\]
Following the same argument as in the proof of Proposition~\ref{prop:SimilarBernstein}, for $p\ge 1$,
\begin{equation}\label{eq:BoundPart1}
\norm{\sum_{i=1}^n X_{i,1}}_p \le \sqrt{6p}\left(\sum_{i=1}^n \mathbb{E}\left[X_i^2\right]\right)^{1/2} + 20p\rho.
\end{equation}
By Hoffmann-Jorgensen's inequality (Proposition 6.8 of~\cite{LED91}) and by the choice of $\rho$,
\[
\norm{\sum_{i=1}^n X_{i,2}}_{1} \le 2\norm{\sum_{i=1}^n |X_i|\mathbbm{1}\{|X_i| \ge \rho\}}_1 \le 16\norm{Z}_1.
\]
Since $\norm{X_i}_{\psi_{\alpha}} < \infty$ for $\alpha > 1$, $\norm{X_i}_{\psi_1} < \infty$. Hence applying Theorem 6.21 of \cite{LED91}, with $\alpha = 1$,
\[
\norm{\sum_{i=1}^n X_{i,2}}_{\psi_1} \le K_{1}\left[16\norm{Z}_1 + \norm{Z}_{\psi_1}\right] \le 17K_{1}\norm{Z}_{\psi_1}.
\]
By Problem 5 of Chapter 2.2 of \cite{VdvW96}, $$\norm{Z}_{\psi_1} \le \norm{Z}_{\psi_{\alpha}}(\log 2)^{1/\alpha - 1}\quad\mbox{for}\quad\alpha \ge 1,$$ and so,
\begin{equation}\label{eq:BoundPart2}
\norm{\sum_{i=1}^n X_{i,2}}_{\psi_1} \le 17K_{1}(\log 2)^{1/\alpha - 1}\norm{Z}_{\psi_{\alpha}} \le C_{\alpha}(\log(n+1))^{1/\alpha}\max_{1\le i\le n}\norm{X_i}_{\psi_{\alpha}},%17K_{1}(\log 2)^{1/\alpha - 1}C_{\alpha}(\log (n + 1))^{1/\alpha},
\end{equation}
for some constant $C_{\alpha} > 0$ depending only on $\alpha$. Therefore, combining inequalities \eqref{eq:BoundPart1} and \eqref{eq:BoundPart2} with $\rho \le 8C_{\alpha}(\log (n + 1))^{1/\alpha}$, for $p\ge 1$
\begin{equation}\label{eq:MomentBoundAlphaLarger1}
\norm{\sum_{i=1}^n X_i}_p \le \sqrt{6p}\left(\sum_{i=1}^n \mathbb{E}\left[X_i^2\right]\right)^{1/2} + C_{\alpha}p(\log(n + 1))^{1/\alpha},
\end{equation}
for some constant $C_{\alpha} > 0$ (possibly different from that in \eqref{eq:BoundPart2}) depending only on $\alpha$. Now the result follows by Proposition \ref{prop:NormMomentEquiSupp} with $\alpha = 1$.
\end{proof}

%\par\medskip   %% NOTE: Having to add this due to strange spacing issues in the IAI template for the Supp! -- 2/12/2022.
\begin{proof}[Proof of Theorem \ref{thm:MaximalTailBound}]
\textit{Case $\alpha \le 1$:} Using the moment bound \eqref{eq:MomentBoundAlphaLess1} in the proof of Theorem \ref{prop:SimilarBernstein}, it follows that for all $1\le j\le q$ and $t \ge 0,$
\[
\mathbb{P}\left(\left|\frac{1}{n}\sum_{i=1}^n X_i(j)\right| \ge e\sqrt{\frac{6\Gamma_{n,q}t}{n}} + \frac{C_{\alpha}K_{n,q}(\log(2n))^{1/\alpha}t^{1/\alpha}}{n}\right) \le ee^{-t},
\]
for some constant $C_{\alpha}$ depending only on $\alpha$ (see, for example, the proof of~\eqref{eq:MomentToTail} for inversion of moment bounds to tail bounds). Hence by the union bound,
\begin{align*}
\mathbb{P}\left(\norm{\frac{1}{n}\sum_{i=1}^n X_i}_{\infty} \right.&\ge\left. 7\sqrt{\frac{\Gamma_{n,q}(t + \log q)}{n}} + \frac{C_{\alpha}K_{n,q}(\log(2n))^{\frac{1}{\alpha}}(t + \log q)^{\frac{1}{\alpha}}}{n}\right)\\
&\le \sum_{j=1}^q \frac{ee^{-t}}{q} \le 3e^{-t}.
\end{align*}

\textit{Case $\alpha \geq 1$:} Using the moment bound \eqref{eq:MomentBoundAlphaLarger1} in the proof of Theorem \ref{prop:BernsteinLargerThan1}, it follows that for all $1\le j\le q$ and $t \ge 0,$
\[
\mathbb{P}\left(\left|\frac{1}{n}\sum_{i=1}^n X_i(j)\right| \ge e\sqrt{\frac{6\Gamma_{n,q}t}{n}} + \frac{C_{\alpha}K_{n,q}(\log(2n))^{1/\alpha}t}{n}\right) \le ee^{-t},
\]
for some constant $C_{\alpha}$ depending only on $\alpha$. Hence by the union bound,
\begin{align*}
\mathbb{P}\left(\norm{\frac{1}{n}\sum_{i=1}^n X_i}_{\infty} \right.&\ge\left. 7\sqrt{\frac{\Gamma_{n,q}(t + \log q)}{n}} + \frac{C_{\alpha}K_{n,q}(\log(2n))^{\frac{1}{\alpha}}(t + \log q)}{n}\right)
% \\
% &\le \sum_{j=1}^q \frac{ee^{-t}}{q}
\le 3e^{-t}.
\end{align*}
This completes the proof.
\end{proof}

%%\par\medskip   %% NOTE: Having to add this due to strange spacing issues in the IAI template for the Supp! -- 2/12/2022.
\paragraph*{Proof of the Claim \eqref{eq:Product} Regarding the Orlicz Norm of Products.}
%{Proof of Remark \ref{rem:ProductVariables}}
The following result establishes the bound \eqref{eq:Product} on the Orlicz norm of a product of random variables. %, and also proves the claim in Remark~\ref{rem:ProductVariables}.

\begin{prop}\label{prop:Product}
If $W_i,$ $1\le i\le k$ are (possibly dependent) random variables satisfying $\norm{W_i}_{\psi_{\alpha_i}} < \infty$ for some $\alpha_i > 0$, then
\begin{equation}\label{eq:ProductSupp}
\norm{\prod_{i=1}^k W_i}_{\psi_{\beta}} \le \prod_{i=1}^k \norm{W_i}_{\psi_{\alpha_i}}\quad\mbox{where}\quad \frac{1}{\beta} := \sum_{i=1}^k \frac{1}{\alpha_i}.
\end{equation}
\end{prop}

%%\par\medskip   %% NOTE: Having to add this due to strange spacing issues in the IAI template for the Supp! -- 2/12/2022.
\begin{proof}
The bound is trivial for $k = 1$ and it holds for $k > 2$ if it holds for $k = 2$ by recursion. For $k = 2$, set $\delta_i = \norm{W_i}_{\psi_{\alpha_i}}$ for $i = 1,2$. Fix $\eta_i > \delta_i$ for $i = 1,2$. By definition of $\delta_i$, this implies
\begin{equation}\label{eq:EtaDef}
\mathbb{E}\left[\exp\left(\left|\frac{W_i}{\eta_i}\right|^{\alpha_i}\right)\right] \le 2\quad\mbox{for}\quad i = 1,2.
\end{equation}
Observe that
\begin{align*}
\mathbb{E}\left[\exp\left(\left|\frac{W_1}{\eta_1}\cdot\frac{W_2}{\eta_2}\right|^{\beta}\right)\right] &\overset{(1)}{\le} \mathbb{E}\left[\exp\left(\left|\frac{W_1}{\alpha_1}\right|^{\alpha_1}\frac{\beta}{\alpha_1} + \left|\frac{W_2}{\eta_2}\right|^{\alpha_1}\frac{\beta}{\alpha_2}\right)\right]\\
&\overset{(2)}{\le} \left(\mathbb{E}\left[\exp\left(\left|\frac{W_1}{\eta_1}\right|^{\alpha_1}\right)\right]\right)^{\beta/\alpha_1}\left(\mathbb{E}\left[\exp\left(\left|\frac{W_2}{\eta_2}\right|^{\alpha_2}\right)\right]\right)^{\beta/\alpha_2}\\
&\overset{(3)}{\le} 2.
\end{align*}
Here, (1) and (2) are applications of Young's inequality and H\"{o}lder's inequality respectively, while (3) follows by \eqref{eq:EtaDef} and the definition of $\beta$. By taking limit as $\eta_i\downarrow \delta_i$, the result \eqref{eq:Product} now follows for $k = 2$. For $k > 2$, the result then follows by recursion, as noted earlier. %proves \eqref{eq:Product} for $k = 2$.
\end{proof}
\def\R{\mathbb{R}}
\def\P{\mathbb{P}}
\def\E{\mathbb{E}}
\def\psihat{\widehat{\psi}}
\def\psialpha{\psi_{\alpha}}
%
%%\par\medskip   %% NOTE: Having to add this due to strange spacing issues in the IAI template for the Supp! -- 2/12/2022.
\paragraph*{Proof of the Bound \eqref{bound:LinearKernel} in Remark \ref{rem:LinearKernel}.} For each fixed $x \in \R^p$, let $T_h(Z; x) := h^{-p} Y K ( (X - x)/h )$, where $Z := (Y,X)$. Then under our assumed conditions, using the quasi-norm property and moment bounds for the $\norm{\cdot}_{\psialpha}$ norm (see, for instance, Chapter 2.2 of \citet{VdvW96}) along with Proposition \ref{prop:Product}, we have: for all $x \in \R^p$, %Then under our assumed conditions, using a simple symmetrization based argument along with Proposition \ref{prop:Product}, we have: for all $x \in \R^p$,
\begin{align}
&  \qquad \qquad \norm{T_h(Z;x) - \E\{ T_h(Z;x) \}}_{\psialpha} & \leq & \;\; A_{\alpha}\left[ \norm{T_h(Z;x)}_{\psialpha} + \E\{ |T_h(Z;x)| \} \right] \\
&  &\leq&  \;\; A_{\alpha}\left[ \norm{T_h(Z;x)}_{\psialpha} + B_{\alpha} \norm{T_h(Z;x)}_{\psialpha} \right]  \\
&  &=&  \;\; D_{\alpha} \norm{T_h(Z;x)}_{\psialpha} \;\; \leq \; D_{\alpha} h^{-p} C_Y C_K, \label{eq1:kernelproof}
\end{align}
where $A_{\alpha}, B_{\alpha} > 0 $ are some constants depending only on $\alpha$, and $D_{\alpha} := A_{\alpha}(1+B_{\alpha}) > 0$.

Further, $\mbox{Var}\{ T_h(Z;x)\} \leq \E \{ T_h^2(Z;x) \}$ and $\E \{ T_h^2(Z;x) \}$ satisfies: for all $x \in \R^p$,
\begin{align}\label{eq2:kernelproof}
& \E \{ T_h^2(Z;x) \}  & = &  \;\; \E\left[ \E \left\{ T_h^2(Z;x) | X  \right\}\right] \; = \; \frac{1}{h^{2p}}\int_{\R^p} \left\{\E\left(Y^2 | X = u \right)\right\} K^2\left( \frac{u- x}{h} \right) f(u) du \nonumber \\
& &= & \;\; \frac{1}{h^{2p}}\int_{\R^p} \left\{ \E\left(Y^2 | X = x + h \varphi \right)\right\} K^2\left(\varphi\right) f(\varphi) \; h^p d\varphi  \;\; \leq \; \frac{R_K M_Y}{h^p},
\end{align}
where the final bound is due to our assumptions. The result \eqref{bound:LinearKernel} now follows by simply applying Theorem \ref{thm:MaximalTailBound} to the random variables $ T_h(Z_i;x) - \E\{ T_h(Z_i;x) \}$, $1 \leq i \leq n$, and by using the bounds \eqref{eq1:kernelproof} and \eqref{eq2:kernelproof}. This completes the proof. \qed

%%%%%%%%%%%%%%%%%%%%%%%%%%%%%%%%%%%%%%%%%%%%%%%%%
%%%%%%%%%%%%%%%%%%%%%%%%%%%%%%%%%%%%%%%%%%%%%%%%%
\section{Proofs of All Results in Section~\ref{sec:Applications}}\label{AppSec:Applications}

%\subsection{Proof of Results in Section~\ref{rem:CovarianceMaxNorm}}\label{AppSec:CovarianceMaxNorm}
\subsection{Proofs of the Results in Section~\ref{sec:ElementWiseMax}}\label{AppSec:CovarianceMaxNorm}

%%\par\medskip   %% NOTE: Having to add this due to strange spacing issues in the IAI template for the Supp! -- 2/12/2022.
\begin{proof}[Proof of Theorem~\ref{cor:Covariance}]
Under assumption \eqref{assump:Covariance}, it follows from Proposition~\ref{prop:Product} that
\[
\max_{1\le i\le n}\max_{1\le j\le k\le p}\norm{X_i(j)X_i(k)}_{\psi_{\alpha/2}} \le K_{n,p}^2,
\]
and so, Theorem~\ref{thm:MaximalTailBound} with $q = p^2$ implies the result.
%By Theorem \ref{prop:SimilarBernstein}, we get for $q\ge 1$, and fixed $1\le j \le k\le p$,
%\begin{align*}
%\norm{\Delta_n(j, k)}_q &\le \sqrt{6q}\norm{\Delta_n(j,k)}_2 + C_{\alpha}q^{2/\alpha}K_{n,p}^2n^{-1/2}\left[1 + (\log (1 + n))^{2/\alpha}\right]\\
%&\le \sqrt{6q}A_{n, p} + C_{\alpha}q^{2/\alpha}B_{n,p}(\alpha),
%\end{align*}
%for some constant $C_{\alpha}$ depending only on $\alpha$. Notice that Theorem \ref{prop:SimilarBernstein} is stated for sums of independent random variables but here they are applied for $n^{-1/2}$-scaled sums and this is the reason for a factor of $n^{-1/2}$ in $B_{n,p}(\alpha).$ Therefore for any $t\ge 0$,
%\[
%\mathbb{P}\left(\Delta_n \ge eA_{n,p}\sqrt{6t + 6\log(1 + p^2)} + eC_{\alpha}B_{n,p}(\alpha)(t + \log(1 + p^2))^{2/\alpha}\right) \le e\exp(-t),
%\]
%proving the result.
\end{proof}

%\par\medskip   %% NOTE: Having to add this due to strange spacing issues in the IAI template for the Supp! -- 2/12/2022.
\begin{proof}[Proof of Theorem~\ref{thm:CovarianceMaxNorm}]
It is easy to verify that
\begin{equation}\label{eq:DecompositionDelta_nS}
\begin{split}
\hat{\Sigma}_n^* &= \frac{1}{n}\sum_{i=1}^n \left(X_i - \bar{\mu}_n\right)\left(X_i - \bar{\mu}_n\right)^{\top} - \left(\bar{X}_n - \bar{\mu}_n\right)\left(\bar{X}_n - \bar{\mu}_n\right)^{\top}\\
&=: \tilde{\Sigma}_n^* - \left(\bar{X}_n - \bar{\mu}_n\right)\left(\bar{X}_n - \bar{\mu}_n\right)^{\top}.
\end{split}
\end{equation}
Clearly,
\begin{equation}\label{eq:BoundCovariance}
\Delta_n^* \le \vertiii{\tilde{\Sigma}_n^* - \Sigma_n^*}_{\infty} + \norm{\bar{X}_n - \bar{\mu}_n}_{\infty}^2,
\end{equation}
where $\norm{x}_{\infty}$ represents the maximum absolute element of $x$. Since $\tilde{\Sigma}_n^*$ is the gram matrix corresponding to the random vectors $X_i - \bar{\mu}_n$, Theorem~\ref{cor:Covariance} applies for the first term on the right hand side of~\eqref{eq:BoundCovariance}. For the second term, Theorem~\ref{thm:MaximalTailBound} applies. Combining these two bounds, we get that for any $t\ge 0$, with probability at least $1 - 6e^{-t}$,
\begin{equation}\label{eq:FirstCovariance}
\begin{split}
\Delta_n^* &\le 7A_{n,p}^*\sqrt{\frac{t + 2\log p}{n}} + \frac{C_{\alpha}K_{n,p}^2(\log(2n))^{2/\alpha}(t + 2\log p)^{2/\alpha}}{n}\\
&\quad+ 98\left(\frac{B_{n,p}(t + \log p)}{n}\right) + \frac{C_{\alpha}K_{n,p}^2(\log(2n))^{2/\alpha}(t + \log p)^{2/\alpha}}{n^2},
\end{split}
\end{equation}
where
% \begin{align*}\frac{1}{n}
$B_{n,p} := \max_{1\le j\le p}\,\sum_{i=1}^n \mbox{Var}\left(X_i(j)\right)/n.$
% \end{align*}
As before, it is easy to show that $B_{n,p} \le C_{\alpha}K_{n,p}^2$ and so, the last two terms of inequality~\eqref{eq:FirstCovariance} are of lower order than the second term and hence, we obtain that with probability at least $1 - 6e^{-t}$,
\begin{equation}\label{eq:FinalCovariance}
\Delta_n^* \le 7A_{n,p}^*\sqrt{\frac{t + 2\log p}{n}} + \frac{C_{\alpha}K_{n,p}^2(\log(2n))^{2/\alpha}(t + 2\log p)^{2/\alpha}}{n},
\end{equation}
with a possibly increased constant $C_{\alpha} > 0$.
\end{proof}

\subsection{Proofs of the Results in Secion~\ref{subsec:SubMatrix}}\label{AppSec:SubMatrix}

%%\par\medskip   %% NOTE: Having to add this due to strange spacing issues in the IAI template for the Supp! -- 2/12/2022.
\begin{proof}[Proof of Theorem~\ref{cor:RIPBoundUnified}]
To prove the result, note that
\[
\RIP_n(k) = \sup_{\substack{\theta\in\mathbb{R}^p,\\\norm{\theta}_0 \le k, \norm{\theta}_2 \le 1}}\left|\frac{1}{n}\sum_{i=1}^n \left\{\left(X_i^{\top}\theta\right)^2 - \mathbb{E}\left[\left(X_i^{\top}\theta\right)^2\right]\right\}\right|,
\]
and define the set
\[
\Theta_k := \left\{\theta\in\mathbb{R}^p:\,\norm{\theta}_0 \le k,\norm{\theta}_2 = 1\right\}\subseteq\mathbb{R}^p.
\]
For every $\varepsilon > 0$, let $\mathcal{N}_{\varepsilon}$ denote the $\varepsilon$-net of $\Theta_k$, that is, every $\theta\in\Theta_k$ can be written as $\theta = x_{\theta} + z_{\theta}$ where $\norm{x_{\theta}}_2 \le 1, x_{\theta}\in\mathcal{N}_{\varepsilon}$ and $\norm{z_{\theta}}_2 \le \varepsilon$. In this representation $x_{\theta}$ and $z_{\theta}$ can be taken to have the same support as that of $\theta$. By Lemma 3.3 of \cite{Plan13}, it follows that
\begin{equation}\label{eq:CoveringBound}
\left|\mathcal{N}_{1/4}\right| \le \left(\frac{36p}{k}\right)^k.
\end{equation}
By Proposition 2.2 of \cite{Ver12}, it is easy to see that $\RIP_n(k)$ can be bounded by a finite maximum as
\begin{equation}\label{eq:FiniteSupremumOperator}
\RIP_n(k) \le 2\sup_{\theta\in\mathcal{N}_{1/4}}\left|\frac{1}{n}\sum_{i=1}^n \left\{\left(X_i^{\top}\theta\right)^2 - \mathbb{E}\left[\left(X_i^{\top}\theta\right)^2\right]\right\}\right|.
\end{equation}
This implies that $\RIP_n(k)$ can be controlled by controlling a finite maximum of averages. Set
\[
\Lambda_n(k) := \sup_{\theta\in\mathcal{N}_{1/4}}\left|\frac{1}{n}\sum_{i=1}^n \left\{\left(X_i^{\top}\theta\right)^2 - \mathbb{E}\left[\left(X_i^{\top}\theta\right)^2\right]\right\}\right|.
\]
\begin{enumerate}
\item[(a)] Under the marginal $\psi_{\alpha}$-bound, it is easy to see that for $\theta\in\Theta_k$ with support $S\subseteq\{1,\ldots,p\}$ of size $k$,
\[
\norm{\left(X_i^{\top}\theta\right)^2}_{\psi_{\alpha/2}} \le \norm{\sum_{j\in S} X_i^2(j)}_{\psi_{\alpha/2}} \le C_{\alpha}\sum_{j\in S}\norm{X_i(j)}_{\psi_{\alpha}}^2 \le C_{\alpha}K_{n,p}^2k,
\]
for some constant $C_{\alpha}$ depending only on $\alpha$. Hence by Theorem \ref{thm:MaximalTailBound}, it follows that for any $t > 0$, with probability at least~$1 - 3e^{-t}$,
\begin{align*}
\Lambda_n(k) &\le 7\sqrt{\frac{\Upsilon_{n,k}(t + k\log(36p/k))}{n}}\\ &\qquad+ \frac{C_{\alpha}K_{n,p}^2k(\log(2n))^{2/\alpha}(t + k\log(36p/k))^{2/\alpha}}{n}.
\end{align*}
\item[(b)] Under the joint $\psi_{\alpha}$-bound, it readily follows that
\[
\sup_{\theta\in\Theta_k}\,\norm{\left(X_i^{\top}\theta\right)^2}_{\psi_{\alpha/2}} \le K_{n,p}^2.
\]
Hence by Theorem \ref{thm:MaximalTailBound}, we get that for any $t > 0$, with probability at least~$1 - 3e^{-t}$,
\begin{align*}
\Lambda_n(k) &\le 7\sqrt{\frac{\Upsilon_{n,k}(t + k\log(36p/k))}{n}}\\ &\qquad+ \frac{C_{\alpha}K_{n,p}^2(\log(2n))^{2/\alpha}(t + k\log(36p/k))^{2/\alpha}}{n}.
\end{align*}
\end{enumerate}
The result now follows since $\RIP_n(k) \le 2\Lambda_n(k)$.
\end{proof}

\subsection{Proofs of the Results in Section~\ref{sec:RECondition}}\label{AppSec:RECondition}

%%\par\medskip   %% NOTE: Having to add this due to strange spacing issues in the IAI template for the Supp! -- 2/12/2022.%
\begin{proof}[Proof of Theorem~\ref{cor:REBoundUnified}]
The proof follows using Theorem \ref{cor:RIPBoundUnified} and Lemma 12 of \cite{Loh12}.
\begin{enumerate}
\item[(a)] From part (a) of Theorem \ref{cor:RIPBoundUnified}, we have with probability at least~$1 - 3s(np)^{-1}$,
\[
\RIP_n(s) \le \Xi_{n,s}^{(M)}.
\]
On the event where this inequality holds, applying Lemma 12 of \cite{Loh12} with $\Gamma = \hat{\Sigma}_n - \Sigma_n$ and $\delta = \Xi_{n,s}^{(M)}$ proves that with probability at least~$1 - 3s(np)^{-1}$,
\[
\theta^{\top}\hat{\Sigma}_n\theta \ge \left(\vphantom{\sum_{i=1}^N}\lambda_{\min}(\Sigma_n) - 27\Xi_{n,s}^{(M)}\right)\norm{\theta}_2^2 - \frac{54\Xi_{n,s}^{(M)}}{s}\norm{\theta}_1^2\quad\mbox{for all}\quad\theta\in\mathbb{R}^p.
\]
\item[(b)] From part (b) of Theorem \ref{cor:RIPBoundUnified}, we get with probability at least~$1 - 3s(np)^{-1}$,
\[
\RIP_n(s) \le \Xi_{n,s}^{(J)}.
\]
By a similar argument as above, the result follows.
\end{enumerate}
This completes the proof of Theorem~\ref{cor:REBoundUnified}.
\end{proof}

\subsection{Proofs of the Results in Section~\ref{sec:HDLinReg}}\label{AppSec:HDLinReg}
The following is a general result of \cite{Neg12} and this particular form is taken from \cite{Hastie15}.

\begin{lem}[Theorem 11.1 of \cite{Hastie15}]\label{lem:Negahbhan12}
Assume that the matrix $\hat{\Sigma}_n$ satisfies the restricted eigenvalue bound \eqref{eq:RestrictedEigenvalue} with $\delta = 3$. Fix any vector $\beta\in\mathbb{R}^p$ with $\norm{\beta}_0 \le k$. Given a regularization parameter $\lambda_n$ satisfying
\[
\lambda_n \ge 2\norm{\frac{1}{n}\sum_{i=1}^n X_i\left(Y_i - X_i^{\top}\beta\right)}_{\infty} > 0,
\]
any estimator $\hat{\beta}_n(\lambda_n)$ from the Lasso \eqref{eq:Lasso} satisfies the bound
\[
\norm{\hat{\beta}_n(\lambda_n) - \beta}_2 \le \frac{3}{\gamma_n}\sqrt{k}\,\lambda_n.
\]
\end{lem}
Lemma \ref{lem:Negahbhan12} holds for any of the minimizers $\hat{\beta}_n(\lambda_n)$ in case of non-uniqueness.

%\par\medskip   %% NOTE: Having to add this due to strange spacing issues in the IAI template for the Supp! -- 2/12/2022.
\begin{proof}[Proof of Theorem~\ref{thm:LassoRate}]
Using Proposition~\ref{prop:Product}, it follows that
\[
\max_{1\le i\le n}\norm{X_i(j)\varepsilon_j}_{\psi_{\gamma}} \le K_{n,p}^2.
\]
By Theorem~\ref{thm:MaximalTailBound}, it follows that with probability at least $1 - 3(np)^{-1}$,
\begin{equation}\label{eq:LambdaRate}
\norm{\frac{1}{n}\sum_{i=1}^n X_i\varepsilon_i}_{\infty} \le 7\sqrt{2}\sigma_{n,p}\sqrt{\frac{\log(np)}{n}} + \frac{C_{\gamma}K_{n,p}^2(\log(2n))^{1/\gamma}(2\log(np))^{1/\gamma}}{n}.
\end{equation}
Next, under assumption~\eqref{eq:Re-satisfied}, and using similar arguments as those used to prove \eqref{eq:general-RE-requirement} in the analysis in Section~\ref{rem:RECond},  we have: if $\bar{s}_M$ minimizes $\Xi_{n,s}^{(M)} + 32 k\Xi_{n,s}^{(M)}/s$ over $1 \leq s \leq p$, then~\eqref{eq:free-parameter-s} implies that with probability at least $1 - 3\bar{s}_{M}/(np) \ge 1 - 3/n$,
\begin{equation}\label{eq:Lasso-RE-verification}
\theta^{\top}\hat{\Sigma}_n\theta \; \ge \left(\lambda_{\min}(\Sigma_n) - 27\min_{1\le s\le p}\left\{\Xi_s + 32 \frac{k\Xi_s}{s}\right\}\right)\|\theta\|_2^2 \; \geq  \frac{\lambda_{\min}(\Sigma_n)}{2}\|\theta\|_2^2  \;\; \forall %\quad \mbox{for all}\;\;
\theta\in\bigcup_{|S|\le k}\mathcal{C}(S; 3),
\end{equation}
%we obtain that the restricted eigenvalue condition holds
%simultaneously for all $\theta\in\mathcal{C}(S; 3)$ and for all $S$ with $|S| \leq k$. Hence,
so that the RE$(k)$ condition \eqref{eq:RestrictedEigenvalue} holds with probability at least $1 - 3n^{-1}$ and with $\gamma_n = \lambda_{\min}(\Sigma_n)/2$ and $\delta = 3$. Therefore, applying Lemma~\ref{lem:Negahbhan12}, along with the bound \eqref{eq:LambdaRate}, it follows that with probability at least $1 - 3(np)^{-1} - 3n^{-1}$, there exists a $\lambda_n$ (given by twice the upper bound in \eqref{eq:LambdaRate}) and hence, a Lasso estimator $\hat{\beta}_n(\lambda_n)$, satisfying:
\[
\norm{\hat{\beta}_n(\lambda_n) - \beta_0}_2 \le \frac{84\sqrt{2}}{\lambda_{\min}(\Sigma_n)}\left[\sigma_{n,p}\sqrt{\frac{k\log(np)}{n}} + 2^{1/\gamma}C_{\gamma}K_{n,p}^2\frac{k^{1/2}(\log(np))^{2/\gamma}}{n}\right].
\]
Note that the choice of $\lambda_n$ above is as claimed in the result. This completes the proof.
\end{proof}

%\par\medskip   %% NOTE: Having to add this due to strange spacing issues in the IAI template for the Supp! -- 2/12/2022.
\begin{proof}[Proof of Theorem~\ref{thm:LassoPoly}]
We already showed in the proof of Theorem~\ref{thm:LassoRate} above that under assumption~\eqref{eq:Re-satisfied}, the RE$(k)$ condition \eqref{eq:RestrictedEigenvalue} holds with probability at least $1 - 3n^{-1}$ with $\gamma_n = \lambda_{\min}(\Sigma_n)/2$ and $\delta = 3$. %the RE condition holds with probability at least $1 - 3s^{(M),*}(np)^{-1}$.
Hence, to apply Lemma~\ref{lem:Negahbhan12}, it is enough to show that the $\lambda_n$ in the statement of Theorem \ref{thm:LassoPoly} is a valid choice for Lemma~\ref{lem:Negahbhan12}. For this, we prove that with probability at least $1 - 3(np)^{-1} - L^{-1}$, for any $L \geq 1$,
\begin{equation}\label{eq:HighProbGradient}
\begin{split}
\max_{1\le j\le p}\left|\frac{1}{n}\sum_{i=1}^n \varepsilon_iX_i(j)\right| & \le \; 7\sqrt{2}\sigma_{n,p}\sqrt{\frac{\log(np)}{n}} \\  & \qquad + \frac{C_{\alpha}K_{n,p}K_{\varepsilon, r}(\log(np))^{1/\alpha}\left[(\log(2n))^{1/\alpha} + L\right]}{n^{1 - 1/r}}.
\end{split}
\end{equation}
We follow the proof technique of Theorem~\ref{prop:SimilarBernstein} to reduce the assumption on $\varepsilon_i$ to polynomial moments, as follows. Define
\[
C_{n,\varepsilon} := 8\mathbb{E}\left[\max_{1\le i\le n}\left|\varepsilon_i\right|\right] \le 8n^{1/r}\max_{1\le i\le n}\norm{\varepsilon_i}_r \le 8n^{1/r}K_{\varepsilon, r}.
\]
Note that under the setting of Theorem~\ref{thm:LassoRate}, for $1\le j\le p$,
\[
\frac{1}{n}\sum_{i=1}^n \varepsilon_iX_i(j) = \frac{1}{n}\sum_{i=1}^n \left\{\varepsilon_iX_i(j) - \mathbb{E}[\varepsilon_iX_i(j)]\right\}.
\]
Set for $1\le i\le n$, $S_i := \varepsilon_iX_i - \mathbb{E}\left[\varepsilon_iX_i\right]\in\mathbb{R}^p,$ and for $1\le j\le p$,
\begin{align*}
S_i^{(1)}(j) &:= S_i(j)\mathbbm{1}_{\{|\varepsilon_i| \le C_{n,\varepsilon}\}} - \mathbb{E}\left[S_i(j)\mathbbm{1}_{\{|\varepsilon_i| \le C_{n,\varepsilon}\}}\right],\\
 S_i^{(2)}(j) &:= S_i(j)\mathbbm{1}_{\{|\varepsilon_i| > C_{n,\varepsilon}\}} - \mathbb{E}\left[S_i(j)\mathbbm{1}_{\{|\varepsilon_i| > C_{n,\varepsilon}\}}\right].
\end{align*}
Therefore, by triangle inequality,
\begin{equation}\label{eq:TwoTerms}
\begin{split}
\max_{1\le j\le p}\left|\frac{1}{n}\sum_{i=1}^n \varepsilon_iX_i(j)\right| &= \max_{1\le j\le p}\left|\frac{1}{n}\sum_{i=1}^n S_i(j)\right|\\
&\le \max_{1\le j\le p}\left|
\frac{1}{n}\sum_{i=1}^n S_i^{(1)}(j)\right| + \max_{1\le j\le p}\left|\frac{1}{n}\sum_{i=1}^n S_i^{(2)}(j)\right|.
\end{split}
\end{equation}
For the summands of the first term, note that
\begin{equation*}
\begin{split}
\mbox{Var}(S_i^{(1)}(j)) & \le \; \mathbb{E}\left[S_i^2(j)\mathbbm{1}_{\{|\varepsilon_i| \le C_{n,\varepsilon}\}}\right] \\ & \le \; \mathbb{E}\left[S_i^2(j)\right] = \mbox{Var}(S_i(j)) = \mbox{Var}(\varepsilon_iX_i(j)),
\end{split}
\end{equation*}
and for some constant $B_{\alpha}$ (depending only on $\alpha$),
\begin{align*}
\norm{S_i^{(1)}(j)}_{\psi_{\alpha}} &\le 2\norm{S_i(j)\mathbbm{1}_{\{|\varepsilon_i| \le C_{n,\varepsilon}\}}}_{\psi_{\alpha}}\\ &\le 2B_{\alpha}\norm{\varepsilon_iX_i(j)\mathbbm{1}_{\{|\varepsilon_i| \le C_{n,\varepsilon}\}}}_{\psi_{\alpha}} + 2B_{\alpha}\left|\mathbb{E}[\varepsilon_iX_i(j)]\right|\\
&\le 2B_{\alpha}C_{n,\varepsilon}K_{n,p} + 2B_{\alpha}\norm{\varepsilon_i}_2\norm{X_i(j)}_2\\
&\le 2B_{\alpha}C_{n,\varepsilon}K_{n,p} + 2B_{\alpha}\norm{\varepsilon_i}_2K_{n,p} = 2B_{\alpha}K_{n,p}\left[C_{n,\varepsilon} + \norm{\varepsilon_i}_2\right]\\
&\le 4B_{\alpha}K_{n,p}C_{n,\varepsilon} \le 32n^{1/r}B_{\alpha}K_{n,p}K_{\varepsilon, r}.
\end{align*}
Therefore, by Theorem~\ref{thm:MaximalTailBound}, it follows that with probability at least $1 - 3(np)^{-1}$,
\begin{equation}\label{eq:FirstTermBoundLasso}
\begin{split}
\max_{1\le j\le p}\left|\frac{1}{n}\sum_{i=1}^n S_i^{(1)}(j)\right| & \le \; 7\sqrt{2}\sigma_{n,p}\sqrt{\frac{\log(np)}{n}} \\ & \qquad + \frac{C_{\alpha}K_{n,p}K_{\varepsilon, r}(\log(2n))^{1/\alpha}(\log(np))^{1/\alpha}}{n^{1 - 1/r}}.
\end{split}
\end{equation}
For the second term in~\eqref{eq:TwoTerms}, note that
\begin{equation*}
\norm{\max_{1\le j\le p}\left|\frac{1}{n}\sum_{i=1}^n S_i^{(2)}(j)\right|}_1 \le 2\norm{\max_{1\le j\le p}\frac{1}{n}\sum_{i=1}^n |\varepsilon_iX_i(j)|\mathbbm{1}_{\{|\varepsilon_i| > C_{n,\varepsilon}\}}}_1.
\end{equation*}
By the definition of $C_{n,\varepsilon}$, we have
\[
\mathbb{P}\left(\max_{1\le j\le p}\frac{1}{n}\sum_{i=1}^n |\varepsilon_iX_i(j)|\mathbbm{1}_{\{|\varepsilon_i| > C_{n,\varepsilon}\}} > 0\right) \le \mathbb{P}\left(\max_{1\le i\le n}|\varepsilon_i| > C_{n,\varepsilon}\right) \le 1/8.
\]
Thus by Hoffmann-Jorgensen's inequality, we have
\begin{equation*}
\begin{split}
\norm{\max_{1\le j\le p}\left|\frac{1}{n}\sum_{i=1}^n S_i^{(2)}(j)\right|}_1 &\le \frac{2}{n}\norm{\max_{1\le j\le p}\max_{1\le i\le n}|\varepsilon_iX_i(j)|}_1\\
&\le \frac{2}{n}\norm{\max_{1\le i\le n}|\varepsilon_i|}_2\norm{\max_{\substack{1\le i\le n,\\1\le j\le p}}|X_i(j)|}_2\\
&\le \frac{2C_{\alpha}(\log(np))^{1/\alpha}}{n^{1 - 1/r}}K_{\varepsilon, r}K_{n,p},
\end{split}
\end{equation*}
for some constant $C_{\alpha} > 0$. So, for any $L \ge 1$, with probability at least $1 - L^{-1}$,
\begin{equation}\label{eq:SecondTermBoundLasso}
\max_{1\le j\le p}\left|\frac{1}{n}\sum_{i=1}^n S_i^{(2)}(j)\right| \le \frac{2LC_{\alpha}(\log(np))^{1/\alpha}K_{\varepsilon, r}K_{n,p}}{n^{1 - 1/r}}.
\end{equation}
From inequalities~\eqref{eq:FirstTermBoundLasso} and~\eqref{eq:SecondTermBoundLasso}, we get with probability at least $1 - 3(np)^{-1} - L^{-1}$,
\begin{equation}\label{eq:TwoTermsBoundLasso}
\begin{split}
\max_{1\le j\le p}\left|\frac{1}{n}\sum_{i=1}^n S_i(j)\right| & \le \; 7\sqrt{2}\sigma_{n,p}\sqrt{\frac{\log(np)}{n}} \\ & \qquad + \frac{C_{\alpha}K_{n,p}K_{\varepsilon, r}(\log(np))^{1/\alpha}\left[(\log(2n))^{1/\alpha} + L\right]}{n^{1 - 1/r}}.
\end{split}
\end{equation}
Taking together the events on which the RE condition and the inequality~\eqref{eq:TwoTermsBoundLasso} hold, we have: with probability at least $1 - 3(np)^{-1} - 3n^{-1} - L^{-1}$, the RE$(k)$ condition \eqref{eq:RestrictedEigenvalue} is satisfied with $\gamma_n = \lambda_{\min}(\Sigma_n)/2$ and $\delta = 3$, and $\lambda_n$ can be chosen as
\[
\lambda_n = 14\sqrt{2}\sigma_{n,p}\sqrt{\frac{\log(np)}{n}} + \frac{C_{\alpha}K_{n,p}K_{\varepsilon, r}(\log(np))^{1/\alpha}\left[(\log(2n))^{1/\alpha} + L\right]}{n^{1 - 1/r}},
\]
so that the lasso estimator $\hat{\beta}_n(\lambda_n)$ satisfies (by Lemma~\ref{lem:Negahbhan12}),
%\begin{equation*}
%\norm{\hat{\beta}_n(\lambda_n) - \beta_0}_2 \le \frac{42\sqrt{2}}{\lambda_{\min}(\Sigma_n)}\left[\sigma_{n,p}\sqrt{\frac{k\log(np)}{n}} + \frac{C_{\alpha}K_{n,p}K_{\varepsilon, r}k^{1/2}(\log(np))^{1/\alpha}\left[(\log(2n))^{1/\alpha} + L\right]}{n^{1 - 1/r}}\right].
%\end{equation*}
\begin{equation}%\label{eq:L2ErrorBound}
\begin{split}
\norm{\hat{\beta}_n(\lambda_n) - \beta_0}_2 & \le \; \frac{84\sqrt{2}}{\lambda_{\min}(\Sigma_n)} \sigma_{n,p}\sqrt{\frac{k\log(np)}{n}}\\  & \qquad + C_{\alpha}K_{n,p}K_{\varepsilon, r}\frac{k^{1/2}(\log(np))^{1/\alpha}\left[(\log(2n))^{1/\alpha} + L\right]}{\lambda_{\min}(\Sigma_n) \; n^{1 - 1/r}}.
\end{split}
\end{equation}
This completes the proof of Theorem~\ref{thm:LassoPoly}.
\end{proof}

\paragraph{Proof of Remark~\ref{rem:LassoExtensions}.} The following result proves the oracle inequality stated in Remark~\ref{rem:LassoExtensions}.

\begin{prop}[Oracle Inequality for Lasso]\label{prop:Lassooracineq}
Consider the setting of Theorem~\ref{thm:LassoRate} (except the hard sparsity on $\beta_0$). For the choice of $\lambda_n$ as in~\eqref{eq:LambdaChoice}, with probability converging to 1,
\begin{equation}\label{eq:OracleIneqSupp}
\begin{split}
&\norm{\hat{\beta}_n(\lambda_n) - \beta_0}_2^2\\
&\quad\le \min_{S:\,\Xi_{n,|S|}^{(M)} = o(1)}\left[\frac{18\lambda_n^2{|S|}}{\Gamma^2_n(S)}~+~\left(\frac{8\lambda_n\norm{\beta_0(S^c)}_1}{\Gamma_n(S)} + \frac{3456\Xi_{n,|S|}^{(M)}\norm{\beta^*(S^c)}_1^2}{|S|\Gamma_n(S)}\right)\right],
\end{split}
\end{equation}
where
\[
\Gamma_n(S) := \lambda_{\min}(\Sigma_n) - 1755\Xi_{n,|S|}^{(M)}.
\]
\end{prop}

%%\par\medskip   %% NOTE: Having to add this due to strange spacing issues in the IAI template for the Supp! -- 2/12/2022.
\begin{proof}
The proof closely follows the arguments of Theorem 11.1 of \cite{Hastie15} and Section 4.3 of \cite{Neg10}. Set for $\nu\in\mathbb{R}^p$,
\[
G(\nu) := \frac{1}{2n}\sum_{i=1}^n \left(Y_i - X_i^{\top}(\beta_0 + \nu)\right)^2 + \lambda_n\norm{\beta_0 + \nu}_1,
\]
and $\hat{\nu} := \hat{\beta}_n(\lambda_n) - \beta_0.$ Also, fix any subset $S\subseteq\{1,2,\ldots,p\}$ with $\Xi_{n,|S|}^{(M)} = o(1)$. Note that with probability at least $1 - 3(np)^{-1}$,
\[
\lambda_n \ge 2\norm{\frac{1}{n}\sum_{i=1}^n X_i\varepsilon_i}_{\infty},
\]
as shown in the proof of Theorem~\ref{thm:LassoRate}. On this event the following calculations hold true. By definition $G(\hat{\nu}) \le G(0)$ and so,
\begin{equation}\label{eq:FirstBasic}
\frac{\hat{\nu}^{\top}\hat{\Sigma}_n\hat{\nu}}{2} \le \hat{\nu}^{\top}\left(\frac{1}{n}\sum_{i=1}^n X_i\varepsilon_i\right) + \lambda_n\left[\norm{\beta_0}_1 - \norm{\beta_0 + \hat{\nu}}_1\right].
\end{equation}
Now observe that
\begin{align*}
\norm{\beta_0 + \hat{\nu}}_1 &\ge \norm{\beta_0(S) + \hat{\nu}(S)}_1 - \norm{\beta_0(S^c)}_1 + \norm{\hat{\nu}(S^c)}_1\\
&\ge \norm{\beta_0(S)}_1 - \norm{\hat{\nu}(S)}_1 - \norm{\beta_0(S^c)}_1 + \norm{\hat{\nu}(S^c)}_1.
\end{align*}
Since $\norm{\beta_0}_1 = \norm{\beta_0(S)}_1 + \norm{\beta_0(S^c)}_1$, the above inequality substituted in~\eqref{eq:FirstBasic} implies
\begin{equation}\label{eq:Combination}
\begin{split}
\frac{\hat{\nu}^{\top}\hat{\Sigma}_n\hat{\nu}}{2} &\le \hat{\nu}^{\top}\left(\frac{1}{n}\sum_{i=1}^n X_i\varepsilon_i\right) + \lambda_n\left[2\norm{\beta_0(S^c)}_1 + \norm{\hat{\nu}(S)}_1 - \norm{\hat{\nu}(S^c)}_1\right]\\
&\le \norm{\hat{\nu}}_1\norm{\frac{1}{n}\sum_{i=1}^n X_i\varepsilon_i}_{\infty} + \lambda_n\left[2\norm{\beta_0(S^c)}_1 + \norm{\hat{\nu}(S)}_1 - \norm{\hat{\nu}(S^c)}_1\right]\\
&\le \frac{\lambda_n}{2}\norm{\hat{\nu}(S)}_1 + \frac{\lambda_n}{2}\norm{\hat{\nu}(S^c)}_1 + \lambda_n\left[2\norm{\beta_0(S^c)}_1 + \norm{\hat{\nu}(S)}_1 - \norm{\hat{\nu}(S^c)}_1\right]\\
&\le \frac{3\lambda_n}{2}\norm{\hat{\nu}(S)}_1 - \frac{\lambda_n}{2}\norm{\hat{\nu}(S^c)}_1 + 2\lambda_n\norm{\beta_0(S^c)}_1.
\end{split}
\end{equation}
This inequality has two implications that prove the result. Firstly, the left hand side of~\eqref{eq:Combination} is non-negative and so,
\begin{equation}\label{eq:ConeLike}
\norm{\hat{\nu}(S^c)}_1 \le 3\norm{\hat{\nu}(S)}_1 + 4\norm{\beta_0(S^c)}_1.
\end{equation}
For the second implication, note that inequality~\eqref{eq:ConeLike} implies that
\begin{align*}
\norm{\hat{\nu}}_1 &= \norm{\hat{\nu}(S)}_1 + \norm{\hat{\nu}(S^c)}_1\\ &\le 4\norm{\hat{\nu}(S)}_1 + 4\norm{\beta_0(S^c)}_1\\ &\le 4\sqrt{|S|}\norm{\hat{\nu}(S)}_2 + 4\norm{\beta_0(S^c)}_1 \le 4\sqrt{|S|}\norm{\hat{\nu}}_2 + 4\norm{\beta_0(S^c)}_1.
\end{align*}
Therefore, applying Theorem~\ref{cor:REBoundUnified} with $s = |S|$, %$k = |S|$,
we get that with probability at least $1 - |S|(np)^{-1}$,
\begin{equation}\label{eq:REConfirm}
\begin{split}
\hat{\nu}^{\top}\hat{\Sigma}_n\hat{\nu} &\ge \left(\lambda_{\min}(\Sigma_n) - 27\Xi_{n,|S|}^{(M)}\right)\norm{\hat{\nu}}_2^2 - \frac{54\Xi_{n,|S|}^{(M)}}{|S|}\left(32|S|\norm{\hat{\nu}}_2^2 + 32\norm{\beta_0(S^c)}_2^2\right)\\
&= \left(\lambda_{\min}(\Sigma_n) - 1755\Xi_{n,|S|}^{(M)}\right)\norm{\hat{\nu}}_2^2 - \frac{1728\Xi_{n,|S|}^{(M)}}{|S|}\norm{\beta_0(S^c)}_1^2\\
&= \Gamma_n(S)\norm{\hat{\nu}}_2^2 - \frac{1728\Xi_{n,|S|}^{(M)}}{|S|}\norm{\beta_0(S^c)}_1^2.
\end{split}
\end{equation}
Combining inequality~\eqref{eq:REConfirm} with inequality~\eqref{eq:Combination}, we obtain
\begin{align*}
\frac{\Gamma_n(S)}{2}\norm{\hat{\nu}}_2^2 &\le \frac{3\lambda_n\sqrt{|S|}}{2}\norm{\hat{\nu}}_2 + 2\lambda_n\norm{\beta_0(S^c)}_1 + \frac{864\Xi_{n,|S|}^{(M)}}{|S|}\norm{\beta_0(S^c)}_1^2
\end{align*}
Hence,
\[
\norm{\hat{\nu}}_2 \le \frac{3\lambda_n\sqrt{|S|}}{\Gamma_n(S)} + \sqrt{\frac{2}{\Gamma_n(S)}}\left(2\lambda_n\norm{\beta_0(S^c)}_1 + \frac{864\Xi_{n,|S|}^{(M)}}{|S|}\norm{\beta_0(S^c)}_1^2\right)^{1/2},
\]
and so the result follows.
\end{proof}

\section{Proofs of All Results in Appendix %Section
\ref{sec:EmpProcess}}\label{AppSec:EmpProcess}

%%\par\medskip   %% NOTE: Having to add this due to strange spacing issues in the IAI template for the Supp! -- 2/12/2022.
\begin{proof}[Proof of Proposition \ref{prop:BddProc}]
By Theorem 3 of \cite{Adam08}, we get for all $t\ge0$ that
\[
\mathbb{P}\left(\left(Z - \mathbb{E}[Z]\right)_+ \ge t\right) \le \exp\left(-\frac{t^2}{2(\Sigma_n(\mathcal{F}) + 2U\mathbb{E}[Z]) + 3Ut}\right).
\]
Set $A := 2(\Sigma_n(\mathcal{F}) + 2U\mathbb{E}[Z])$ and $B := 3U$. Then using the arguments in Proposition~\ref{prop:BernsteinMoment}, we get for $p\ge 2$, (and any $\delta > 0$)
\begin{align*}
\mathbb{E}\left[\left(Z - \mathbb{E}[Z]\right)_+^p\right] &\le \int_0^{\infty} pt^{p-1}\exp\left(-\frac{t^2}{A(1 + \delta)}\right)dt + \int_0^{\infty} pt^{p-1}\exp\left(-\frac{t\delta}{B(\delta + 1)}\right)dt\\
&\le \left(\sqrt{A(1 + \delta)/e}\right)^p\frac{\sqrt{2\pi}}{e}\left(1 + \frac{p}{2}\right)^{(p+1)/2}\exp\left(\frac{1}{12(1 + p/2)}\right)\\
&\quad+\left(\frac{B(1 + \delta)}{e\delta}\right)^p\frac{\sqrt{2\pi}}{e}(1 + p)^{p + \frac{1}{2}}\exp\left(\frac{1}{12(p+1)}\right) =: \mathbf{I} + \mathbf{II}.
\end{align*}
So, for $p\ge 2$,
\begin{align*}
\mathbf{I}^{1/p} &\le p^{1/2}\sqrt{A(1 + \delta)/e}\left(\frac{\sqrt{2\pi}}{e}\right)^{1/p}\left(1 + \frac{p}{2}\right)^{\frac{1}{2p}}\exp\left(\frac{1}{12p(1 + p/2)}\right)\\
&\le p^{1/2}\sqrt{A}\quad\mbox{ by taking $\delta$ = 1/2.}
\end{align*}
Also, regarding $\mathbf{II}$, for $p\ge 2$,
\begin{align*}
\mathbf{II}^{1/p} &\le p\left(\frac{3B(1 + \delta)}{2e\delta}\right)\left(\frac{\sqrt{2\pi}}{e}\right)^{1/p}(1 + p)^{1/(2p)}\exp\left(\frac{1}{12p(p+1)}\right)\le 2Bp.
\end{align*}
Therefore, for $p\ge 2$,
\[
\norm{(Z - \mathbb{E}[Z])_+}_p \le \left(2\Sigma_n(\mathcal{F}) + 4U\mathbb{E}[Z]\right)^{1/2}\sqrt{p} + 6pU,
\]
% Letting
% \[
% A_n := 2\sqrt{(2U\mathbb{E}[Z] + \Sigma_n(\mathcal{F}))},\quad B_n := 3U,\quad \mbox{and}\quad L_n := B_n/A_n,
% \]
% this inequality implies that
% \begin{align*}
% \mathbb{E}\left[\Psi_{1, L_n}\left(\frac{(Z - \mathbb{E}[Z_n])_+}{A_n}\right)\right] &= \int_0^{\infty} \mathbb{P}\left((Z - \mathbb{E}[Z])_+ \ge A_n\sqrt{\log (1 + t)} + B_n\log(1 + t)\right)dt\\ &\le 1.
% \end{align*}
% This proves \eqref{eq:NormBound} by definition of $\Psi_{1, L_n}$. By Proposition \ref{prop:EquivalenceNormMoment} (in particular \eqref{eq:PGe2Bound}), it follows that for $p\ge 2$,
% \[
% \norm{(Z - \mathbb{E}[Z])_+}_p \le \left(2\Sigma_n(\mathcal{F}) + 4U\mathbb{E}[Z]\right)^{1/2}\sqrt{p} + 3Up,
% \]
and since $\norm{Z}_p \le \norm{(Z - \mathbb{E}[Z])_+}_p + \mathbb{E}[Z]$, for $p\ge 1$,
\[
\norm{Z}_p \le \mathbb{E}[Z] + \left(2\Sigma_n(\mathcal{F}) + 4U\mathbb{E}[Z]\right)^{1/2}\sqrt{p} + 6Up,
\]
proving \eqref{eq:MomentBound}.
\end{proof}

%\par\medskip   %% NOTE: Having to add this due to strange spacing issues in the IAI template for the Supp! -- 2/12/2022.
\begin{proof}[Proof of Theorem \ref{thm:UnBddProc}]
By triangle inequality, $Z \le Z_1 + Z_2$. Note that $Z_1$ is a supremum of bounded empirical process and so, by Proposition \ref{prop:BddProc} for $p\ge 2$
\begin{align}
\norm{(Z_1 - \mathbb{E}[Z_1])_+}_p &\le p^{1/2}\left(2\Sigma_n(\mathcal{F}) + 8\rho\mathbb{E}\left[Z_1\right]\right)^{1/2} + 12\rho p\nonumber\\
&\le \sqrt{2}p^{1/2}\Sigma_n^{1/2}(\mathcal{F}) + 2\sqrt{2}p^{1/2}\rho^{1/2}\left(\mathbb{E}\left[Z_1\right]\right)^{1/2} + 12\rho p\nonumber\\
&\le \sqrt{2}p^{1/2}\Sigma_n^{1/2}(\mathcal{F}) + \left(2p\rho + \mathbb{E}\left[Z_1\right]\right) + 12\rho p\nonumber\\
&= \mathbb{E}\left[Z_1\right] + \sqrt{2}p^{1/2}\Sigma_n^{1/2}(\mathcal{F}) + 14 p\rho,
\label{eq:BoundedBound}
\end{align}
where we used the arithmetic-geometric mean inequality and the fact that
\[
\mbox{Var}\left(f(X_i)\mathbbm{1}\{|f(X_i)| \le \rho\}\right) \le \mathbb{E}\left[f^2(X_i)\mathbbm{1}\{|f(X_i)| \le \rho\}\right] \le \mathbb{E}\left[f^2(X_i)\right] \le \mbox{Var}\left(f(X_i)\right).
\]
To deal with $Z_2$, observe that
\[
\norm{Z_2}_{\psi_{\alpha_*}} \le 2\norm{\sum_{i=1}^n F(X_i)\mathbbm{1}\{F(X_i) \ge \rho\}}_{\psi_{\alpha_*}}
\]
Since $\alpha_* \le 1$ for all $\alpha > 0$ and $\norm{F(X_i)}_{\psi_{\alpha_*}} < \infty$, it follows from Theorem 6.21 of \cite{LED91} that
\[
\norm{Z_2}_{\psi_{\alpha_*}} \le 2K_{\alpha_*}\left\{\mathbb{E}\left[\sum_{i=1}^n F(X_i)\mathbbm{1}\{F(X_i) \ge \rho\}\right] + \norm{\max_{1\le i\le n} F(X_i)}_{\psi_{\alpha^*}}\right\},
\]
with the constant $K_{\alpha_*}$ as in the cited theorem. Additionally by Hoffmann-Jorgensen inequality combined with the definition of $\rho$, we have
\[
\mathbb{E}\left[\sum_{i=1}^n F(X_i)\mathbbm{1}\{F(X_i) \ge \rho\}\right] \le 8\mathbb{E}\left[\max_{1\le i\le n} F(X_i)\right] \le 8\norm{\max_{1\le i\le n} F(X_i)}_{\psi_{\alpha}}.
\]
And by Problem 5 of Chapter 2.2 of \cite{VdvW96},
\[
\norm{\max_{1\le i\le n} F(X_i)}_{\psi_{\alpha_*}} \le (\log 2)^{1/\alpha - 1/\alpha_*}\norm{\max_{1\le i\le n} F(X_i)}_{\psi_{\alpha}}.
\]
Therefore,
\[
\norm{Z_2}_{\psi_{\alpha_*}} \le 2K_{\alpha_*}\left[8 + (\log 2)^{1/\alpha - 1/\alpha_*}\right]\norm{\max_{1\le i\le n} F(X_i)}_{\psi_{\alpha}}.
\]
This implies that for $p\ge 1$,
\begin{equation}\label{eq:UnboundedBound}
\norm{Z_2}_p \le 2\sqrt{2\pi}(p/\alpha_*)^{1/\alpha_*}K_{\alpha_*}\left[8 + (\log 2)^{1/\alpha - 1/\alpha_*}\right]\norm{\max_{1\le i\le n} F(X_i)}_{\psi_{\alpha}}.
\end{equation}
Note that for all $\alpha > 0$, $(p/\alpha_*)^{1/\alpha^*} \ge p$ for all $p \ge 1$ and so, for all $\alpha > 0$,
\[
14p\rho \le \sqrt{2\pi}(p/\alpha_*)^{1/\alpha_*}K_{\alpha_*}\left[8 + (\log 2)^{1/\alpha - 1/\alpha_*}\right]\norm{\max_{1\le i\le n} F(X_i)}_{\psi_{\alpha}}.
\]
Therefore, combining bounds \eqref{eq:BoundedBound} and \eqref{eq:UnboundedBound}, we obtain for $p\ge 2$,
\begin{align*}
\norm{Z}_p &\le \mathbb{E}\left[Z_1\right] + \norm{(Z_1 - \mathbb{E}[Z_1])_+}_p + \norm{Z_2}_p\\
&\le 2\mathbb{E}\left[Z_1\right] + \sqrt{2}p^{1/2}\Sigma_n^{1/2}(\mathcal{F})\\ &\qquad+ 3\sqrt{2\pi}(p/\alpha_*)^{1/\alpha_*}K_{\alpha_*}\left[8 + (\log 2)^{1/\alpha - 1/\alpha_*}\right]\norm{\max_{1\le i\le n} F(X_i)}_{\psi_{\alpha}}.
\end{align*}
This proves \eqref{eq:MomentEmpBound}. Using the reasoning as in Proposition \ref{prop:BddProc},
\[
\mathbb{E}\left[\Psi_{\alpha_*, L_n(\alpha)}\left(\frac{(Z - 2e\mathbb{E}[Z_1])_+}{3\sqrt{2}e\Sigma_n^{1/2}(\mathcal{F})}\right)\right] \le 1,
\]
with $L_n(\alpha)$ is as defined in the statement. This proves \eqref{eq:NormEmpBound}.
% Actually, from the proof a better (or slightly sharper) bound can be obtained by using the inequality
% \[
% Z \le \mathbb{E}[Z_1] + (Z_1 - \mathbb{E}[Z_1])_+ + Z_2,
% \]
% and by the moment bounds. The proof is complete. %We leave the details for the reader to figure out.
\end{proof}

%\par\medskip   %% NOTE: Having to add this due to strange spacing issues in the IAI template for the Supp! -- 2/12/2022.
\begin{proof}[Proof of Proposition \ref{prop:MaximalUniform}]
By Theorem 3.5.1 and inequality (3.167) of \cite{GINE16},
\[
\mathbb{E}\left[\sup_{f\in\mathcal{F}}\left|\mathbb{G}_n(f)\right|\right] \le 8\sqrt{2}\mathbb{E}\left[\int_0^{\eta_n(\mathcal{F})}\sqrt{\log(2N(x, \mathcal{F}, \norm{\cdot}_{2,P_n}))}dx\right],
\]
where $P_n$ represents the empirical measure of $X_1, X_2, \ldots, X_n$, that is, $P_n(\{X_i\}) = 1/n$. Here
\[
\eta_n^2(\mathcal{F}) := \sup_{f\in\mathcal{F}}\norm{f}_{2, P_n} = \sup_{f\in\mathcal{F}}\left(\frac{1}{n}\sum_{i=1}^n f^2(X_i)\right)^{1/2}.
\]
Using a change-of-variable formula,
\[
\mathbb{E}\left[\sup_{f\in\mathcal{F}}\left|\mathbb{G}_n(f)\right|\right] \le 8\sqrt{2}\mathbb{E}\left[\norm{F}_{2,P_n}J\left(\delta_n(\mathcal{F}), \mathcal{F}, \norm{\cdot}_{2}\right)\right],
\]
where
\[
\delta_n^2(\mathcal{F}) := \frac{1}{\norm{F}_{2,P_n}^2}\sup_{f\in\mathcal{F}}\, \frac{1}{n}\sum_{i=1}^n f^2(X_i),\quad\mbox{and}\quad \norm{F}_{2,P_n}^2 := \frac{1}{n}\sum_{i=1}^n F^2(X_i).
\]
Now an application of Lemma 3.5.3 (c) of \cite{GINE16} implies that
\begin{equation}\label{eq:FirstInequalityMaximal}
\mathbb{E}\left[\sup_{f\in\mathcal{F}}\left|\mathbb{G}_n(f)\right|\right] \le 8\sqrt{2}\norm{F}_{2,P}J\left(\frac{\Delta}{\norm{F}_{2,P}}, \mathcal{F}, \norm{\cdot}_2\right),
\end{equation}
where
\[
\Delta^2 := \mathbb{E}\left[\sup_{f\in\mathcal{F}}\frac{1}{n}\sum_{i=1}^n f^2(X_i)\right]\quad\mbox{and}\quad \norm{F}_{2,P}^2 := \frac{1}{n}\sum_{i=1}^n \mathbb{E}\left[F^2(X_i)\right].
\]
Note that by symmetrization and contraction principle
\begin{align*}
\Delta^2 &\le \sup_{f\in\mathcal{F}}\frac{1}{n}\sum_{i=1}^n \mathbb{E}\left[f^2(X_i)\right] + \mathbb{E}\left[\sup_{f\in\mathcal{F}}\frac{1}{n}\left|\sum_{i=1}^n \left\{f^2(X_i) - \mathbb{E}\left[f^2(X_i)\right]\right\}\right|\right]\\
&\le n^{-1}\Sigma_n(\mathcal{F}) + \frac{16U}{\sqrt{n}}\mathbb{E}\left[\sup_{f\in\mathcal{F}}\left|\mathbb{G}_n(f)\right|\right].
\end{align*}
See Lemma 6.3 and Theorem 4.12 of \cite{LED91}. Substitute \eqref{eq:FirstInequalityMaximal} in this inequality, we obtain
\[
\frac{\Delta^2}{\norm{F}_{2,P}^2} \le \frac{n^{-1}\Sigma_n(\mathcal{F})}{\norm{F}_{2,P}^2} + \frac{128\sqrt{2}U}{\sqrt{n}\norm{F}_{2,P}}J\left(\frac{\Delta}{\norm{F}_{2,P}}, \mathcal{F}, \norm{\cdot}_2\right).
\]
For notation convenience, let
\[
H(\tau) := J\left(\tau, \mathcal{F}, \norm{\cdot}_2\right),\quad A^2 := \frac{n^{-1}\Sigma_n(\mathcal{F})}{\norm{F}_{2,P}^2},\quad\mbox{and}\quad B^2 := \frac{128\sqrt{2}U}{\sqrt{n}\norm{F}_{2,P}}.
\]
Following the proof of Lemma 2.1 of \cite{vdV11} with $r = 1$, it follows that
\[
H\left(\frac{\Delta}{\norm{F}_{2,P}}\right) \le H\left(A\right) + \frac{B}{A}H\left(A\right)H^{1/2}\left(\frac{\Delta}{\norm{F}_{2,P}}\right).
\]
Solving the quadratic inequality, we get
\[
H\left(\frac{\Delta}{\norm{F}_{2,P}}\right) \le 2\frac{B^2}{A^2}H^2(A) + 2H(A).
\]
Substituting this bound in \eqref{eq:FirstInequalityMaximal}, it follows that
\[
\mathbb{E}\left[\sup_{f\in\mathcal{F}}\left|\mathbb{G}_n(f)\right|\right] \le 16\sqrt{2}\norm{F}_{2,P}J\left(\delta_n(\mathcal{F}), \mathcal{F}, \norm{\cdot}_2\right)\left[1 + \frac{128\sqrt{2}UJ\left(\delta_n(\mathcal{F}), \mathcal{F}, \norm{\cdot}_2\right)}{\sqrt{n}\delta_n^2(\mathcal{F})\norm{F}_{2,P}}\right].
\]
This proves the result.
\end{proof}

%\par\medskip   %% NOTE: Having to add this due to strange spacing issues in the IAI template for the Supp! -- 2/12/2022.
\begin{proof}[Proof of Proposition \ref{prop:BracketingEntropy}]
In the proof of Theorem 3.5.13 of \cite{GINE16}, the decomposition (3.206) holds as it is and the calculations that follow have to be done for averages of non-identically distributed random variables. For example, the display after (3.206) should be replaced by Lemma 4 of \cite{Poll02}. (The inequality in Lemma 4 of \cite{Poll02} is written for $\sqrt{n}\mathbb{G}_n(f)$ not $n^{-1/2}\mathbb{G}_n(f)$). The variance calculations after (3.209) of \cite{GINE16} should be done as
\[
\mbox{Var}\left(\mathbb{G}_n\left(\Delta_kfI_{\{\tau f = k\}}\right)\right) = \frac{1}{n}\sum_{i=1}^n \mathbb{E}\left[(\Delta_kf)^2I(\Delta_kf \le \alpha_{n,k-1}, \Delta_k(f) > \alpha_{n,k})\right].
\]
See, for example, Lemma 5 of \cite{Poll02}. There is a typo in Proposition 3.5.15 in the statement; it should be $Pf^2 \le \delta^2$ for all $f\in\mathcal{F}$. In our case this $\delta$ would be the one defined in the statement. The final result follows by noting the concavity of~$J_{[\,]}(\cdot, \mathcal{F}, \norm{\cdot}_{2,P})$,
\[
J_{[\,]}\left(2\delta_n(\mathcal{F}), \mathcal{F}, \norm{\cdot}_{2,P}\right) \le 2J_{[\,]}\left(\delta_n(\mathcal{F}), \mathcal{F}, \norm{\cdot}_{2,P}\right).
\]
This completes the proof.
\end{proof}

%\par\medskip   %% NOTE: Having to add this due to strange spacing issues in the IAI template for the Supp! -- 2/12/2022.
\begin{proof}[Proof of Proposition~\ref{prop:ExpectationUnbddBdd}]
It is clear by the triangle inequality that $Z \le Z_1 + Z_2$ and so,
\[
\mathbb{E}\left[Z\right] \le \mathbb{E}\left[Z_1\right] + \mathbb{E}\left[Z_2\right].
\]
From the definition~\eqref{eq:SplitZ1Z2} of $Z_2$, we get
\begin{align*}
\mathbb{E}\left[Z_2\right] &\le 2\mathbb{E}\left[\sup_{f\in\mathcal{F}}\sum_{i=1}^n |f(X_i)|\mathbbm{1}_{\{|f(X_i)| \ge \rho\}}\right] \le 2\mathbb{E}\left[\sum_{i=1}^n F(X_i)\mathbbm{1}_{\{F(X_i) \ge \rho\}}\right].
\end{align*}
Using Hoffmann-Jorgensen's inequality along with the definition of $\rho$, we have
\[
\mathbb{E}\left[\sum_{i=1}^n F(X_i)\mathbbm{1}_{\{F(X_i) \ge \rho\}}\right] \le 8\mathbb{E}\left[\max_{1\le i\le n}\, F(X_i)\right].
\]
Therefore,
\[
\mathbb{E}\left[Z\right] \le \mathbb{E}\left[Z_1\right] + 8\mathbb{E}\left[\max_{1\le i\le n}\, F(X_i)\right].
\]
This completes the proof.
\end{proof}

%\bibliography{References}
%\end{document} 

% \end{appendix}
\end{document}